\documentclass[11pt]{amsart}

\usepackage{amsfonts,amsmath,latexsym,verbatim,amscd}

\usepackage{amsthm,amsmath,amssymb,amscd}
\usepackage{amssymb}
\usepackage{graphics}
\usepackage{bbm}

\newcommand{\bbN}{\mathbb{N}}

\newcommand{\bbR}{\mathbb{R}}
\newcommand{\bbS}{\mathbb{S}}
\newcommand{\bbT}{\mathbb{T}}

\newcommand{\calA}{\mathcal{A}}

\newcommand{\calC}{\mathcal{C}}

\newcommand{\calL}{\mathcal{L}}
\newcommand{\calM}{\mathcal{M}}

\newcommand{\calO}{\mathcal{O}}
\newcommand{\calP}{\mathcal{P}}

\theoremstyle{plain}
\numberwithin{equation}{section}

\newcommand{\moduli}{\mathcal{M}}
\def\mi{{\mathbbm i}}

\def\a{{\alpha}}

\def \l{\lambda}

\def \o{\omega}
\def \g{\gamma}
\def\e{{\varepsilon}}

\def\P{{\mathbb P}}

\def \ch{ {\cosh } \, \o}

\def \th{ {\tanh } \, \o}

\def\un{{\mathbb I}{\rm d}}

\newcommand{\ov}[1]{\overline{#1}}

\def\si{\sigma}

\newcommand{\R}{{\mathbb R}}
\newcommand{\SY}{{\mathbb S }}

\newcommand{\N}{{\mathbb N}}
\newcommand{\C}{{\mathbb C}}
\newcommand{\Z}{{\mathbb Z}}

\newcommand{\pk}{{\calP}_g}

\newcommand{\SL}{\mathrm{SL}_{\mbox{\tiny{$2$}}}}
\newcommand{\Sl}{\mathfrak{sl}_{\mbox{\tiny{$2$}}}}
\newcommand{\SU}{\mathrm{SU}_{\mbox{\tiny{$2$}}}}
\newcommand{\su}{\mathfrak{su}_{\mbox{\tiny{2}}}}

\newcommand{\Sp}{\mathbb S}
\newcommand{\rmd}{{\rm d}}

\newcommand{\LSU}{\Lambda  \SU}

\theoremstyle{plain}
\newtheorem{theorem}{Theorem}[section]
\newtheorem*{theorem*}{Theorem}
\newtheorem{corollary}[theorem]{Corollary}
\newtheorem*{corollary*}{Corollary}
\newtheorem{proposition}[theorem]{Proposition}
\newtheorem*{proposition*}{Proposition}
\newtheorem{lemma}[theorem]{Lemma}
\newtheorem*{lemma*}{Lemma}

\newtheorem*{example*}{Example}
\newtheorem{definition}[theorem]{Definition}
\newtheorem*{definition*}{Definition}

\newtheorem*{notation*}{Notation}
\newtheorem{remark}[theorem]{Remark}
\newtheorem*{remark*}{Remark}

\numberwithin{equation}{section}

\title{Properly embedded minimal annuli in $\SY^2 \times \R$}
\author{L. Hauswirth \and M. Kilian \and M. U. Schmidt}

\address{\tiny L. Hauswirth, Universit Paris-Est, LAMA (UMR 8050), UPEMLV, UPEC, CNRS, F-77454, Marne-la-Valle, France }
\email{hauswirth@univ-mlv.fr}

\address{\tiny M. Kilian, School of Mathematical Sciences, Western Gateway Building, University College Cork, Ireland.}
\email{m.kilian@ucc.ie}

\address{\tiny M. U. Schmidt, Institut f\"ur Mathematik, Universit\"at Mannheim, A5, 6, 68131 Mannheim, Germany.}
\email{schmidt@math.uni-mannheim.de}

\begin{document}

\begin{abstract} 
In $\SY^2 \times \R$ there is a two-parameter family of properly embedded minimal annuli foliated by circles. In this paper we show that this family contains all properly embedded minimal annuli.
We use the description of minimal annuli in $\SY^2 \times \R$ by periodic harmonic maps $G: \C \to \SY^2$
of finite type. Due to the algebraic geometric correspondence of Hitchin  \cite{HIT}, these harmonic maps are 
parametrized by
hyperelliptic algebraic curves $\Sigma$ together with Abelian differentials $dh$ with prescribed poles. 
We deform annuli by deforming spectral data in the corresponding moduli space. Along this deformation 
we control the flux and we preserve embeddness. The center of the theory concerns the study of
singularities of the flow. In particular we open and close nodes of singular spectral curves  $\Sigma$. 
This approach applies also to mean convex Alexandrov embedded {\sc{cmc}} annuli in $\SY^3$ \cite{HKS3}.

\end{abstract}

\maketitle
\tableofcontents

\section{Introduction}
There is a two-parameter family of embedded minimal annuli foliated by horizontal constant
curvature curves of $\SY^2$. The simplest non compact examples are the totally geodesic $\Gamma \times \R$, where $\Gamma$ is a simple closed geodesic on $\SY^2$. There exists a one-parameter family of periodic properly embedded annuli which are small graphs over $\Gamma \times \R$. These examples were described analytically by Pedrosa and Ritore \cite{PedRit} and they called them unduloids. They appear in the isoperimetric profile of $\SY^2 \times \SY^1$. These examples are rotational surfaces around vertical geodesics. A one-parameter helicoidal family, obtained by rotating a great circle on $\SY^2$ at a constant rate in the third coordinate about an axis passing through a pair of antipodal points on the rotated great circle was constructed by Rosenberg \cite{HR}.

A two-parameter family of deformations of previous examples was constructed by the first-named author \cite{hau}, and by Meeks and Rosenberg (see \cite{MRstable}, section 2) using  variational arguments. This involves solving a Plateau problem with boundaries given by two geodesics $\Gamma_1$ and $\Gamma_2$ in parallel sections $\SY^2 \times \{ t_1\}$ and $\SY^2 \times \{t_2\}$. The stable annulus bounded by these geodesics is foliated by horizontal constant curvature curves (see theorem \ref{shiffman}). Schwarz symmetry along boundary geodesics then gives a complete and properly embedded example. The two parameters (up to  isometries of $\SY^2 \times \R$) of such compact annuli are the distance between the two sections including the boundary, and the position of one geodesic in $\SY^2 \times \{t\}$, keeping the other fixed. They are periodic in the third direction, foliated by constant curvature curves of $\SY^2$ and have a vertical plane of symmetry. 

These examples are very similar to the minimal Riemann's staircase in $\R^3$, which has been classified by Meeks, Perez and Ros \cite{MPR}.
In this Euclidean case they use an integrable system point view related to the holomorphic Gauss map. They define algebra-geometric harmonic
solution to a complex KDV type equation and study holomorphic flow associate to this equation.

Constant mean curvature ({\sc{cmc}}) tori in $\R^3$ give rise to minimal annuli in $\SY^2 \times \R$ under certain conditions as follows. The Gauss map of a {\sc{cmc}} torus is a harmonic map $G: \bbT^2 \to \SY^2$. Its holomorphic quadratic differential is
$$Q=\langle G_z,\,G_z \rangle (dz)^2$$
where $z=x+ \mi y$ is a global holomorphic coordinate on the torus. Since $\langle G_z,\,G_z \rangle :\bbT^2 \to \C$ is holomorphic, it is a non-zero constant $\langle G_z,\,G_z \rangle \equiv c \in \C^*$.

After a linear change of coordinate we can assume $c =\pm 1/4$.  Then the map $X: \C \to \SY^2 \times \R$,
$$X(z)=(G(z), {\rm Re} (- 2\mi\sqrt{c}z))$$
is conformal and harmonic, and thus locally a minimal surface in $\SY^2 \times \R$ (possibly branched).
We can choose the sign of $c \in \R$ in such a way that the large curvature line on the ({\sc{cmc}}) torus corresponds to a horizontal curve in  $\SY^2 \times \{t\}$. In this case $X(z)$ is an immersion. If the Gauss map $G$ is periodic along this horizontal curve then we have a minimal annulus.

The flat {\sc{cmc}} cylinder in $\R^3$ yields a flat minimal annulus in $\SY^2 \times \R$, and the Gauss maps of Delaunay surfaces ($G$ is periodic as defined on a torus) yield the unduloid, and the {\sc{cmc}} nodoids yield the helicoid in $\SY^2 \times \R$ under this correspondence by the harmonic map $G$.

Supplementing the rotational examples, there is a two-parameter family of harmonic maps studied by Abresch \cite{abresch}. To describe the equations of Wente tori in $\R^3$, Abresch studies conformal {\sc{cmc}} $H=1/2$ immersion with large line or small line of curvature included in a plane. He studies constant mean curvature surfaces parameterized by $\R^2$, with the coordinate axes $x$ and $y$ yielding the principal lines of curvature, and solves closing conditions for the surface to obtain {\sc{cmc}} tori. On these lines, the image of the Gauss map 
is a circle of $\SY^2$.

In a conformal parametrization with $|c|=1/4$, the metric of the minimal annulus in $\SY^2 \times \R$ is given by $ds^2 = \cosh^2 \o |dz|^2$ where $\o : \C \to \R$ is a solution of the sinh-Gordon equation
\begin{equation}
\label{sinh}
\Delta \o + \sinh \o \cosh \o =0\,,
\end{equation}
where $\Delta$ denotes the Laplacian of the flat metric $|dz|^2$.  The metric of the corresponding {\sc{cmc}} $H=1/2$ surface in $\R^3$ is given by $d\tilde s^2 = e^{2\o} |dz|^2$.
Abresch classified all real analytic solutions $\o : \C \to \R$ of the  system
\[
(I) \quad \left\{\begin{array}{l}
\Delta \o  +\sinh  \o \cosh \o =0  \\[3mm]
\sinh (\o)  ( \o_{xy}) -\cosh (\o) (\o_x)( \o_y) =0
\end{array} \right.
\]
where the second equation is the condition that the small curvature lines are planar. Then he proves that this family contains {\sc{cmc}} $H=1/2$ tori. These tori give doubly periodic harmonic maps $G$ and so minimal immersions by considering $X(z)=(G(z),x)$ (the map $X(z)=(G(z),y)$ is branched). The horizontal curves of this family have non-constant curvature.

In a second part, Abresch studies solutions of the system
\[
(II) \quad\left\{\begin{array}{ll}
\Delta \o  +\sinh  \o \cosh \o =0 & \\[3mm]
\cosh (\o)  (\o_{xy}) -\sinh (\o)  ( \o_x)(\o_y) =0&
\end{array} \right.
\]
where the second equation is the condition that large lines of curvature are planar.
These solutions $\o$ induces a {\sc{cmc}} immersion of $\C$ in $\R^3$ and a doubly-periodic Gauss map $G: \R^2 \to \SY^2$. The second equation is the condition that the immersion $X(z)=(G(z),y)$
has horizontal constant curvature curve and parameterizes the whole family of properly embedded annuli of $\SY^2 \times \R$. For these reasons,

\begin{definition}
An Abresch annulus of $\SY^2 \times \R$ is an embedded minimal annulus foliated by horizontal constant curvature curves.
\end{definition}

It was conjectured by Meeks-Rosenberg \cite{MRstable} that any properly embedded genus zero minimal surface in $\SY^2 \times \R$ is an Abresch annulus. In this direction, Hoffmann and White \cite{hoffmanwhite} proved that if a properly embbeded annulus of $\SY^2 \times \R$ contains a vertical geodesic, then this is a helicoid type example. The first author in \cite{hau}  characterized Abresch annuli as the only annuli which are foliated by horizontal curvature curves. Our main result confirms this conjecture, and we prove the following

\vspace{1mm}\noindent{\bf Main Theorem.}
{\it A properly embedded minimal annulus in $\SY^2 \times \R$ is an Abresch annulus.}\vspace{1mm}

The proof combines methods from geometric analysis with techniques of integrable systems. The first ingredient is a linear area growth and curvature estimate of Meeks-Rosenberg. Due to  theorem 7.1 in \cite{MPR} (see Theorem \ref{curvest} below), the curvature and thus the area growth of a properly embedded minimal annulus $A$ in $\SY^2 \times \R$ are bounded by constants depending on the flux of the third coordinate $h: A \to \R$ along horizontal sections. Properly  immersed annuli in $\SY^2 \times \R$ are parabolic (see theorem~\ref{immersionsinhgordon}) and the flux
corresponds to the length of the period $\tau$ of the corresponding solution of the sinh-Gordon
equation~\eqref{sinh}(see lemma \ref{flux}). If the flux $|\tau| \geq \epsilon _0$, there is a constant $C_1>0$ depending only on
$\epsilon_0$ such that
$$|K| \leq C_1 (\epsilon_0)$$
We improve the linear area growth estimate of theorem 1.1 in \cite{MPR} using parabolicity, and prove in lemma \ref{tubular} that there is a constant $C_2>0$ depending only on $\epsilon_0$ such that for any $t >0$,
$${\rm Area}(A \cap [-t,t]) \leq C_2(\epsilon_0) t.$$
This estimate has two consequences. Firstly it implies that properly embedded annuli are of finite type. This means that the periodic immersion $X : \C \to \SY^2 \times \R$  with period $\tau \in \C^*$ can be described by algebraic data. Up to some finite dimensional and compact degree of freedom the immersion is determined by the so-called {\emph{spectral data}} $(a,\,b)$. They consist of two polynomials of degree $2g$ respectively $g+1$ for some $g \in \bbN$. The polynomial $a(\l)$ encodes a hyperelliptic Riemann surface called {\emph{spectral curve}}. The genus of the spectral curve is called {\emph{spectral genus}}. The other polynomial $b(\l)$ encodes the closing conditions $X(z+\tau)=X(z)$ for some $\tau \in \C^*$ depending on $(a,b)$. This correspondence is called the algebro-geometric correspondence.

Using this algebro-geometric correspondence, one can deform a minimal annulus by deforming the corresponding spectral data.
Hitchin \cite{HIT} introduce the algebra geometric correspondence for harmonic map $G: \bbT^2 \to \SY^2$. Krichever introduced the Whitham deformations as a tool to deform spectral data. Starting with an embedded minimal annulus in $\bbS ^2 \times \bbR$, the Whitham deformation allows us to deform the annulus preserving minimality, closing condition as well as embeddedness. Applying this to the flat embedded minimal annulus allows us to flow through the path-connected component of embedded annuli. In this way we are able to construct the whole family of Abresch annuli via Whitham deformation theory (see Appendix C). Applying this deformation to annulus gives us additional freedom parameter
of deformation in comparison with the doubly periodic case $G: \bbT^2 \to \SY^2$.

The second consequence of the curvature estimate is the compactness of the space of spectral data corresponding to properly embedded minimal annuli with periods $\tau$ bounded away from zero. As a direct consequence each connected component of the space of such spectral data $(a,b)$ contains a maximum of the period $\tau$, since connected components are always closed.

If the spectral genus is larger than zero, then we show in Lemma~\ref{increasetau} that there always exists a Whitham deformation of $(a,b)$ increasing the period $\tau$. The flat annulus is the only surface whose spectral curve has spectral genus zero. Therefore there is only one connected component of spectral data $(a,b)$ corresponding to properly embedded minimal annuli.

Helicoidal and rotational unduloids are of spectral genus one, while the Riemann's type examples are of spectral genus two. This family of Abresch annuli is characterized by an additional symmetry in the corresponding spectral data $(a,b)$, and we will see that it is not possible to break this symmetry by continuous deformation while preserving a closing condition of the annulus. Therefore the Abresch annuli form the only connected component of the embedded flat minimal annuli.

The curvature and area growth estimate of Meeks-Rosenberg helps us to simplify the proof of the Main Theorem in the present paper. However such an estimate relating the flux and the curvature is not necessary to control the deformation. In \cite{HKS3} the present authors construct a Whitham deformation connecting the spectral data of an arbitrary Alexandrov embedded annulus in $\SY^3$ with the spectral data of the Clifford torus without using such an estimate. It is shown that no geometrical accident can appear along this deformation.

\vskip 0.3cm
Each of the remaining sections contains a step of the proof. In the subsequent Section~\ref{sec:proof} we present the results of all remaining sections and explain how they fit  together to a proof of the Main Theorem. Section~\ref{sec:proof} contains the proof of the Main Theorem and a summary of the remaining sections.

\section{Proof of the Main Theorem}\label{sec:proof}
The proof of the theorem proceeds in several steps. Each section in the paper contains a step of the proof. We outline the results of the individual sections and explain how they fit together to prove the theorem.

\noindent{\bf Section~\ref{preliminary}.} We first set the notation and discuss
some local and global aspects of minimal surfaces in $\bbS^2 \times
\bbR$. In Theorem~\ref{immersionsinhgordon} we describe conformal minimal immersions. We show that a proper annulus is parabolic, and in a conformal parametrization $X: \C/ \tau \Z \to \SY^2 \times \R,\,z \mapsto (G(z),\,h(z))$ minimality is equivalent to the harmonicity of both $G$ and $h$. Properness implies $dh \neq 0$ and we can then find a conformal parametrization with constant Hopf differential
$Q= \langle G_z,\,G_z \rangle \,(dz)^2=\tfrac{1}{4} \exp(\mi\Theta)\,(dz^2)$ and linear vertical component $h(z) = {\rm Re} (-\mi e^{\mi\Theta/2} z)$. 
\vskip 0.3cm
In particular $h: A \to \R$ is a proper harmonic map. The intersection $A \cap \{x_3=t\}$ has only one connected
component, a topological circle. The metric of the immersion is given by $ds^2 = \cosh^2 \o \,dz \otimes d\bar z$, and the third coordinate of the unit normal vector given by $n_3 =\tanh \o$. We prove that in these special coordinates the real function $\o : \C \to \R $ is a solution of sinh-Gordon equation~\eqref{sinh}
where $\Delta$ is the Laplacian in the flat Euclidean metric. We conclude by computing $\o$ for the Abresch family.

\noindent{\bf Section~\ref{sec:curvature estimate}.}
Due to a theorem by Meeks-Rosenberg \cite{MPR}, the curvature $K$ of a properly
embedded annulus is bounded by a constant depending on the flux $F_3$ of $h:A \to \R$.
If $|F_3| \geq \epsilon_0$, there is a constant $C_1 >0$ depending
on $\epsilon_0$ with
\begin{equation}\label{eq:curvature bound}
    |K| \leq C_1(\epsilon_0)
\end{equation}
If $\tau \in \C$ is the period of the annulus, and the Hopf differential is $Q=\frac{e^{\mi\Theta}}{4}(dz)^2$,
then we show in Lemma~\ref{flux} that $F_3=|\tau|$. The estimate of curvature on parabolic annuli implies (see lemma \ref{tubular})  the existence of a constant $C_2(\epsilon_0)$ depending only on the lower bound of the flux of the height function $h: \C/\tau Z \to \R$ and the upper bound on the absolute Gaussian curvature with
\begin{equation}\label{eq:area bound}
{\rm Area}(A \cap [-t,t]) \leq C_2(\epsilon_0) t.
\end{equation}

This estimate gives a uniform bound of the third coordinate of the normal unit vector in ${\mathcal C}^{k, \alpha}$ norm  
$|\omega|_{A,k,\alpha} \leq C_0$(see proposition \ref{omegabound}). 
This implies that any horizontal curve $A \cap \{x_3=t\}$ has an embedded tubular neighborhood depending only on $C_0$ and 
$\epsilon_0$.


\noindent{\bf Section~\ref{pinkallsterling}.}
Since the function $\o$ is a solution of the sinh-Gordon equation, we can apply the iteration of Pinkall-Sterling in section \ref{pinkallsterling} to obtain an infinite hierarchy of solutions $u_1,\,u_2,...$  of the linearized sinh-Gordon equation (LSG for short):
\begin{equation}\label{eq:lsg}
    \calL\,u_n = \Delta u_n+ u_n\,\cosh(2 \o) = 0\,.
\end{equation}
Finite type (see Definition~\ref{finite type}) means that the set of such solutions forms a finite dimensional vector space in the kernel of $\calL$.
Combining this with the Meeks-Rosenberg curvature and area growth estimate~\eqref{eq:curvature bound},
this gives us the first step in the proof of our Theorem: Properly embedded minimal annuli in $\SY^2 \times \R$ are of finite type (Theorem~\ref{parabolic}).

Finite type solutions of the sinh-Gordon equation give rise to algebraic objects which we call potentials. They are defined as follows.
\begin{definition}\label{def pot}
The open subset of a $3g+1$ dimensional real vector space of matrices called potentials is
\begin{align*}
{\calP}_g= \{ \xi_\l=\sum_{d=-1}^{g}\hat \xi_d\lambda^d  \mid  \hat \xi_{-1}\in  \bigl( \begin{smallmatrix} 0 &  \mi\R^+ \cr 0 & 0 \end{smallmatrix} \bigr)  , {\rm trace}(\hat \xi_{-1}\hat  \xi_0) \neq 0,\, \\
\hat  \xi_d=-\bar{\hat \xi}^t_{g-1-d}\in\Sl(\C)\mbox{ for }d=-1,\ldots,g  \}.
\end{align*}
\end{definition}
There is a correspondence between finite type solutions of the
sinh-Gordon equation and solutions $\zeta_\l:\C\to\pk$ of a
Lax equation
$$d \zeta_\l (z) + [ \a (z), \zeta_\l (z)]=0 \hbox{ with } \zeta_\l (0) =\xi_\l \hbox{ where }$$
\begin{equation} \label{matrixalpha}
    \alpha (z) = \frac{1}{4}\,\begin{pmatrix}
    2  \o_z &  \mi \lambda^{-1}  e^\o\\
    \mi\gamma e^{-\o}&  -2\o_z
    \end{pmatrix}\, dz + \frac{1}{4}\,\begin{pmatrix}
    -2 \o_{\bar{z}} &
    \mi\bar{\gamma} e^{-\o}\\
    \mi \lambda  \,e^{\o}&
    2\o_{\bar{z}}
    \end{pmatrix} d \bar{z}
\end{equation}
The Lax equation ensures that
$\det \zeta_\l (z) = \det \xi_\l $ does not depend on the variable
$z$. In Definition~\ref{spectral curve} we define for a potential
$\xi_\l=\zeta_\l(0)$ the polynomial $a(\l)=-\l\det \xi_\l $ and the
spectral curve $\Sigma$ as the 2-point compactification of
\begin{equation}\label{eq:spectral curve}
\Sigma^* = \{(\nu,\l) \in \C^2 \mid \det(\nu\; \un - \zeta_\l)=0\} = \{
(\nu,\l) \in \C^2\mid \nu ^2 =\lambda^{-1} a(\lambda)=-\det \xi_\l \}
\end{equation}
We note that the points over $\l=0$ and $\l=\infty$ are branch points of
$\Sigma$. Here $a$ is a polynomial of degree $2g$ obeying the reality
condition
\begin{equation}\label{eq:reality}
|a(0)|=\frac{1}{16}, \;\;\;  \lambda^{2g} \overline{a(\bar{\lambda}^{-1})} =a(\lambda) \hbox{ and } \frac{a(\l) }{\l ^g} \leq 0 \hbox{ for all } \l \in \SY^1.
\end{equation}

\vskip 0.3cm
\noindent{\bf Section~\ref{sec:polynomial killimg field}. Construction of a minimal surface with potential $\xi_\l$.} We identify the sphere $\SY^2$ with $\SU / U (1)$ and $(x,y,z) \in \R^3$ with the matrix $ M= \bigl( \begin{smallmatrix} \mi z  & x+\mi y \\ -x+\mi y  & -\mi z \end{smallmatrix} \bigr)$. The Lie group $\SU$ (with associated Lie algebra $\su$), acts on $\su \cong \R^3$ isometrically by $M \longrightarrow g M g ^{-1}$ with $|M|^2=\det M$. The map  $g \to g  \sigma_3 {} ^t \bar g$ maps $\SU$ into $\SY ^2 \subset \R^3$ with $\sigma_3=  \bigl( \begin{smallmatrix} \mi & 0 \\ 0 & -\mi \end{smallmatrix} \bigr)$.
An extended frame $F : \C \longrightarrow \SU$ associate to a map $G:\C \to \SY^2$ is defined by
$$G=F\sigma_3F^{-1}.$$
To obtain a minimal surface we first solve for an extended frame  $F_\l: \C \to  \SL (\C)$
$$\label{equation}
F^{-1}_\l dF_\l = \a (z) \hbox{ with } F_\l (0)=\un$$
for parameter $\l \in \C$. For $\l \in \SY^1$, we have $F_\l \in \SU$ (see Section 6). We choose to evaluate at $\l =1$.
We say then that we 'parametrize the immersion by its Sym point' (see Proposition \ref{sym}).
We apply a construction of Sym-Bobenko type which gives us the minimal immersion at $\l=1$:
\begin{equation}\label{eq:immersion}
X(z)=(F_1 (z) \sigma_3 F_1^{-1} (z) , {\rm Re}(-\mi e^{\mi\Theta /2}z)\, \hbox{ and }\,
Q_1=-4\beta_{-1} \gamma_0 (dz)^2=\tfrac{1}{4}\, e^{\mi\Theta}(dz)^2.
\end{equation}
\vskip 0.3cm
We remark that by formal computation if we set
\begin{equation}
\label{eq:conj}
\zeta_\l (z):=F_\l^{-1}(z) \xi_\l F_\l (z),
\end{equation}
we obtain a polynomial Killing field $\zeta_\l : \C \to \pk$ solution of the Lax equation
$$d \zeta_\l + [F^{-1}_\l dF_\l , \zeta_\l]=0 \hbox{ and } \zeta_\l (0) = \xi_\l.$$
\vskip 0.3cm
Conversely from a potential $\xi_\l$, the generalized Weirstrass data of Dorfmeister, Pedit, Wu \cite{DPW} and a correspondence
between potentials, polynomial Killing fields and extended frames known as "Symes method" \cite{BurP_adl,BurP:dre,Symes_80} solve the Lax equation. This method gives a polynomial Killing field $\zeta_\l : \C \to {\calP}_g$, a solution of sinh-Gordon equation $\o : \C \to \R$ and a
map $F_\l : \C \to \SL (\C)$ which satisfy equation (\ref{eq:conj}) and which in turn gives us the immersion. 
The main tool is a decomposition of Iwasawa type for loops group. The symmetry of the matrices
$\xi_\l \in {\calP}_g$ contains the information of $F_\l \in \SU$ for $\l \in \SY^1$. This construction 
is detailed in \cite{HKS1} and explained in Section 6.

\vskip 0.3cm
Consequently we shall
characterize in the remaining sections those potentials $\xi_\l$,
which induce for $\l=1$ properly embedded minimal annuli in $\SY^2\times\R$.
\vskip 0.3cm
\noindent{\bf Section~\ref{sec:spectral data}. Isospectral set and spectral data.}
The condition on a potential $\xi_\l$ to induce a periodic immersion turns out to be a condition on the corresponding polynomial $a(\l)$ and the spectral curve \eqref{eq:spectral curve}. Therefore we investigate the
set $I(a)$ \eqref{eq:isospectral set} of all potentials with the same spectral curve:
\begin{equation}\label{eq:isospectral set}
I(a):=\left\{ \xi_\l \in \pk \mid
\l \det \xi_\l= -a(\l)\right\}
\end{equation}
\noindent{\bf Curvature estimate on finite type annuli.} The compactness of the isospectral set $I(a)$ implies an estimate of curvature. This comes from the fact that $\zeta_\l : \C \to I(a)$ depends on $\o$ and its higher derivatives. We bound these derivatives by the modulus of the roots of $a(\l)$.
\begin{proposition}
\label{corcurva}
If $\o:\C\to\R$ is the solution of the sinh-Gordon equation of a minimal immersion $X : \C \to \SY^2 \times \R$, then the curvature of $X$ is equal to
\begin{align}\label{eq:def curvature}
    K=\tanh ^2 \o - \frac{| \nabla \o|^2}{\cosh^4\o}.
\end{align}
For polynomials $a$ obeying \eqref{eq:reality} the set of finite type
immersions corresponding to all potentials $\xi_\l \in I(a)$ have uniform bounded curvature and uniform bounded
function $\o$ with
$$|K| \leq C \hbox{ and } |\o|�\leq C$$
if and only if the roots of $a$ are bounded away from
$\infty$ and $0$.
\end{proposition}
\begin{proof}
The formula (\ref{eq:def curvature}) comes from sinh-Gordon equation applied to Gauss curvature formula for the conformal 
metric $ds^2=\cosh^2 \o |dz|^2$.
Let $\bar{\moduli}^g$ denote the subspace of all complex polynomials $a$ of degree $2g$ obeying
\begin{align*}
|a(0)|&\ne0,&\lambda^{2g}\overline{a(\bar{\lambda}^{-1})}&=a(\lambda)\mbox{ and}&\l^{-g}a(\l)&\leq 0 \hbox{ for all } \l \in \SY^1
\end{align*}
(compare with \eqref{eq:reality}). We prove in Proposition 4.9 \cite{HKS1} that the map
\begin{equation}\label{eq:map a}
  a:\pk \to\bar{\moduli}^g\,,\quad \xi \mapsto -\l\,\det\xi_\l
\end{equation}
is a proper map. The Laurent coefficients of $\xi_\l = \sum_{d=-1}^g \l^d \hat \xi_d $ are
$$
   \hat  \xi_d = \frac{1}{2\pi\mi}\int_{\SY^1} \l ^{-d} \xi_\l \,\frac{d\l}{\l}\,.
$$
Using a norm
$$
    \| \hat \xi_d\| \leq \frac{1}{2\pi\mi}\int_{\SY^1} \|\l ^{-d}\xi _\l \| \,\frac{d\l}{\l} \leq\sup_{\l\in\SY^1}\sqrt{-\l^{-g}a(\l)}\,.
$$
Thus each entry $\hat \xi_d$ of $\xi_\l$ is bounded, if $\sqrt{-\l^{-g}a(\l)}$ is bounded on $\SY^1$. For polynomials $a$ obeying~\eqref{eq:reality} this follows from the roots of $a$ being bounded. Since $\hat \xi_{-1}$ is an upper triangular matrix with coefficient $4 \beta_{-1}=\mi e^{\o}$ (see Lemma \ref{lax} and Remark \ref{laxomega}), the bound of the Laurent coefficient $\hat \xi_{-1}$ on $I(a)$ is equivalent to a uniform bound on $\o$. Schauder estimates bound the $\calC^{k,\alpha}$ norm of $\o$ on $\R /\tau\Z\times\R$.

Conversely let $\o : \C \to \R$ of the immersions corresponding to all $\xi_\l \in I(a)$ be uniformly bounded with $|\o|�\leq C$. For any roots 
$\alpha_1,\ldots,\alpha_g$ of $a$, such that $\bar{\alpha}_1^{-1},\ldots,\bar{\alpha}_g^{-1}$ are the remaining roots of $a$, there exist an off-diagonal $\xi_\l\in I(a)$ with
\begin{align*}
\alpha&=0&
\beta&=\frac{\mi}{4\l\sqrt{\prod_d|\alpha_d|}}\prod_d(1-\bar{\alpha}_d\l)&
\gamma&=\frac{\mi}{4\sqrt{\prod_d|\alpha_d|}}\prod_d(\l-\alpha_d).
\end{align*}
The corresponding $\o$ at $z=0$ is due to proposition \ref{lax} and remark \ref{laxomega} equal to
$\o(0)=-\frac{1}{2}\sum_d\ln|\alpha_d|$ and $\nabla\o(0)=0$. If $|\o (0)|�\leq C$  then all roots of $a$ are
bounded away from $\infty$ and $0$.
\end{proof}

\begin{remark} \label{algebraicestimatecurvature} 
This proposition implies that properly embedded minimal annuli with flux bounded away from zero have spectral data
with roots bounded away from infinity and zero depending only on the constant $\epsilon _0$.

\end{remark}

{\bf Group action.} We define via the Iwasawa decomposition a commuting group action
$$\pi:\C^g  \times I(a) \to I(a)$$ on the isospectral set. It integrates the family of solutions of the linearized sinh-Gordon equation $u_1,\,u_2,...$ into deformations of the metric $\o$ and then deformations of the extended frame $F_\l$.  The first solution $u_1=\o_z$ integrates as translation on the annulus. The group action corresponding to the first solution $u_1= \o _z$ represent the annulus as a two-dimensional subgroup of the isospectral set. For $(z,0,...0) \in \C^g$, we have
$$\zeta_\l(z)=\pi (z,0,...0)\xi_\l$$
the polynomial Killing field $\zeta_\l : \C \to \pk$.

An important property of spectral curves without singularities (i.e. the polynomial $a(\l)$ has only simple roots), is that the isospectral set $I(a)$ has only one orbit diffeomorphic to a real $g$-dimensional torus which is isomorphic to the real part of the Picard group (the real part of the Jacobian of $\Sigma$). This property implies that all annuli related to a smooth spectral curve have a quasi-periodic polynomial Killing
field as a flat annulus immersed in $(\SY^1)^g$ (see corollary \ref{quasiperiodic}).  Since this polynomial Killing field depends only on $\o$, this means that also the metric is quasi-periodic.

\begin{remark}
On minimal annuli of $\SY^2 \times \R$ parameterized by the third coordinate $h(z)=y$, the function $u=\cosh^2 \o (\partial_x k_g) = ( \o_{xy}) -\tanh \o ( \o_x)( \o_y)$ is a Jacobi field. In this expression $k_g$ is the geodesic curvature (see theorem \ref{shiffman}) of the horizontal curve. Integrating this normal Jacobi field on the surface gives
a variation on $\o$ which is $u_2=\o_{zzz}-2\o_{z}^3$, the second flow in the hierarchy. To integrate this Jacobi field it suffices to look for the second group action on $\xi_\l$:
$$\zeta_{\l}(z,t)=\pi (z,t,0,...0)\xi_\l=F_\l(z,t)^{-1} \xi_\l F_\l(z,t)$$
is the Killing field which integrates the Shiffman Jacobi field on the surface with extended frame $F_\l (z,t)$. In particular this flow
exists for all times and is quasi-periodic.
\end{remark}
For polynomials $a$ with higher order roots, the isospectral action is not transitive.
In this case the
isospectral set decomposes into several orbits with respect to the
action. Besides one smallest orbit all orbits contain what we call
bubbletons. We treat this case in Appendix~\ref{sec:bubbleton}. Due to
Corollary~\ref{bubbleton orbit} there always exists one largest orbit,
whose elements $\xi_\l\in I(a)$ have no roots in $\l\in\C^\times$ (see below).

The closing condition is encoded in the spectral data define in definition \ref{spectraldata}. If $\tau$ is the period for an embedded annulus induced by a potential $\xi_{0,\l}$, then any potential $\xi_\l = \pi(t) \xi_{0,\l}$ induces an embedded annulus with the same period $\tau$. This means that $\pi(\tau)\xi_\l=\xi_\l$ is a trivial action for any $\xi_\l$ in the same orbit as $\xi_{0,\l}$.
The eigenvalue $\mu(\l)$ of $\l \to F_\l (\tau)$ is a holomorphic function on $\Sigma^*$ with two essential singularities at the branch points over $\l=0$ and $\l=\infty$ of $\Sigma$. At these points the 1-form $d \ln \mu$ has second order poles with no residue and $d \ln \mu -\tfrac{1}{4} \mi \tau e^{\mi\Theta}d \sqrt{\l}^{-1}$ extends holomorphically at $\l=0$. A similar expression at $\l=\infty$ holds. This formula means that the third coordinate of the flux is encoded at the essential singularities of $\mu$ (see Propositions \ref{stab}, \ref{mu} and Definition \ref{spectraldata}).
The meromorphic differential takes the form
$$d \ln \mu =\frac{b \,d\l}{\nu \l^2}$$
for some polynomial $b(\l)$ of degree $g+1$ where
$b(0)=-\tfrac{1}{2} \tau e^{\mi\Theta} \in e^{\mi\Theta /2}\R$. We call $(a,\,b)$ the spectral data of the annulus. In Corollary~\ref{periodic immersions}
we give a complete characterization of spectral data of periodic immersions.

\vskip 0.3cm
\noindent{\bf Section~\ref{sembedded}. Isospectral set and embedded annuli. }
By the maximum principle at infinity, an embedded annulus has an embedded tubular neighborhood (see Lemma \ref{tubular}). Using this property and maximum principle at an interior point, we prove in Propositions \ref{simpleembedded} and \ref{bubbleembedded} that the isospectral action preserves embeddedness in the isospectral set $I(a)$. If we have $\hat \xi_\l= \pi (t) \xi_\l$ and $ \xi_\l $ induces an embedded annulus $X(z,\xi_\l)$, then the potential $\hat \xi_\l $ induces an embedded annulus $X(z,\hat \xi_\l)$. We provide first a proof of this property when $a(\l)$ has only simple roots. This implies that if $\xi_\l \in I(a)$ induces an embedded annulus then the whole orbit of $\xi_\l$ in $I(a)$ has potentials all inducing embedded annuli. If $a (\l)$ has only simple roots, $I(a)$ has only one orbit. In the case where $a(\l)$ has higher order roots, it can happen in general that only few orbits contains potentials inducing embedded annuli.
  \vskip 0.3cm
 \noindent
{\bf Isospectral set for polynomial $a(\l)$ having higher order roots.}
 Different $\xi_\l$ of different isospectral sets may give the same extended frame $F_\l$.  This is the case in particular if an initial value $\xi_\l$ has a root
at some $\lambda=\alpha_0 \in\C ^*$. Then also the corresponding polynomial Killing field $\zeta_\l (z)$ has a root at  $\l=\alpha_0$ for all $z\in\C$. In
this case we may reduce the order of $\xi_\l$ and $\zeta_\l (z)$ without changing the corresponding extended frame $F_\lambda$.

This configuration corresponds to a singular spectral curve i.e. the polynomial $a(\l)$ has a root of order at least two at $\alpha_0$. We can remove this singularity without changing the surface. There is a polynomial $p(\l)$ such that $\tilde \xi _\l= \xi_\l/ p $ does not vanish at $\alpha_0$ and is the initial value of a polynomial Killing field $\tilde \zeta _\l (z)$ without zeroes at $\alpha_0$. We show in Proposition 4.4 \cite{HKS1} that 
both polynomial Killing fields $\zeta_\l$ and $\zeta_\l /p$ induce congruent minimal surfaces in $\SY^2 \times \R$:
\begin{proposition}(proposition 4.4,\cite{HKS1})
\label{remove}
If a polynomial Killing field $\zeta_\l$ with initial value $\xi_\l \in I(a) \subset \pk$
has zeroes in $\lambda \in\C^{*}$, then there is a
polynomial $p(\lambda)$, such that the following conditions
hold:
\begin{enumerate}
\item $\xi_\l/p$ has no zeroes in $\lambda \in \C^{*}$, has degree $(g-{\rm deg } p)$.
\item If $F_\l$ and $\tilde F_\l$ are the unitary decomposition factors of $\zeta_\l$ and $\zeta_\l/p$ respectively then
$\widetilde {F}_\l (p(0)z)=F_\l(z)$ and the induced immersions parametrized by  Sym point $X(z)$ and $\widetilde{X} (z)$ are congruent.

\end{enumerate}
\end{proposition}
Hence amongst all polynomial Killing fields that give rise to a
minimal surface of finite type there is one of smallest possible degree
(without adding further poles).
\begin{proposition}(lemma 4.7  \cite{HKS1})
Let $I(a)$ be the isospectral set associated to a polynomial $a(\l)$, which satisfies the reality condition.
\begin{enumerate}
\item If $a (\l)$ has a double root $\alpha_0 $ with $|\alpha_0|=1$, then $I(a)= \{\xi_\l \in I(a) \mid \xi_{\alpha_0}=0\}$ and there is an isomorphism
$$I(a) \longrightarrow I({\frac{\alpha_0 a}{( \l-\alpha_0)^2}}) \,\hbox{ defined by }\, \xi_\l \mapsto  \frac{\sqrt{\alpha_0}\, \xi _\l}{\l-\alpha_0}$$
\item If $a(\l)$ has double root $\alpha_0$ with $|\alpha_0 | \neq 1$ then $I(a)=\{\xi_\l \in I(a) \mid \xi_{\alpha_0} \neq 0\} \cup \{\xi_\l \in I(a) \mid \xi _{\alpha_0}=0\}$
and there is an isomorphism
$$\{ \xi_\l \in I(a) \mid \xi_{\alpha_0} =0\} \longrightarrow I ({\frac{a}{(\l-\alpha_0)(1-\bar{\alpha}_0 \l)}}) \,\hbox{ defined by }\, \xi_\l \mapsto \frac{1}{(\l-\alpha_0)(1-\bar \alpha_0 \l)}\xi_\l\,. $$
\end{enumerate}
\end{proposition}
\begin{proof}
(1) If $a$ has a double root at $\alpha_0$ with $|\alpha_0|=1$ , then for any $\xi _\l \in I(a)$,
we have $\xi _{\alpha_0}=0$, because the determinant is a norm for any $\l \in \SY^1$.

(2) If $a$ has a double root at $\alpha_0$, with $|\alpha_0| \neq 0$, then the isospectral set splits into a part which contain potentials with a zero at $\alpha_0$ (and we can remove again the singularity), and the set of potential not zero at $\alpha_0$. The last case mean that $\xi _{\alpha_0}$ is a nilpotent matrix,
and we say that the surface has a bubbleton. We study this case in Appendix \ref{sec:bubbleton} and in the next subsection.
\end{proof}
{\bf Remove or add bubbletons on annulus.} Bubbletons occur when the polynomial $a(\l)$ has higher order roots $\a_1, ..., \a_k$.
We can assume (repeating several times the same roots if necessary) that
$$a(\l)=\prod_{i=1}^{k} (\l - \a_i)^2(1- \l \bar \a_i)^2 \tilde a(\l)=(\l - \a_j)^2(1- \l \bar \a_j)^2 \tilde a_j(\l)$$
where $\tilde a(\l)$ has only simple roots, ${\rm deg}\; a(\l)=2g$ and  ${\rm deg}\; \tilde a(\l)=2g-2k$. The group action
$$\pi (.)\xi_\l: \C^{g}  \to I(a)$$
preserves the degree of roots of $\xi_\l$ at $\a_1,..., \a_k$ (if $\xi_{0,\alpha_i}=0$, then $\xi_{\alpha_i}(t)=\pi (t)\xi_{0,\alpha_i}=0$).
\vskip 0.3cm
Hence the isospectral action defines several orbits in $I(a)$, defined by the degree of the roots of $\xi_\l$ at roots $\a_i$. The group $\C^{g}$ acts on $I(a)$ but the stabilizer
$\Gamma_{\xi_\l}=\{ t \in \C^{g} \mid \pi(t)\xi_\l=\xi_\l \}$ is no longer isomorphic to $\Z^{2g}$.
\vskip 0.3cm
If $\xi_{0,\l} \in I(a)$ and $O=\{ \xi_\l \in I(a) \mid \xi_\l= \pi (t) \xi_{0,\l} \hbox{ for } t \in \C^{g} \}$ is the orbit of some potential $\xi_{0,\l}$
having roots at $\a_j$ i.e. $\xi_{0, \alpha_j}=0$, there is a real two-dimensional subgroup $E^2 \subset \C^{g}$ which acts
trivially on $\xi_{0, \l}$. In this case $O$ is diffeomorphic to $I({\tilde a_j})$. We can remove the roots of $\xi_{0,\l}$ without
changing the immersion given by $F_\l$. If we do that for any higher order roots of $a$, we remove all zeroes and obtains
a quasi-periodic annulus induced by a potential $\tilde \xi_\l \in I({\tilde a}) \cong  (\bbS^1)^{g-2k}$.
\vskip 0.3cm
Now, we consider the orbit $O$ of a potential $\xi_{0,\l} \in I(a)$, where $\xi_{0,\a_i} \neq 0$ and $\a_i$ is a higher order root of $a(\l)=(\l - \a_i)^2(1- \l \bar \a_i)^2 \tilde a_i(\l)$. If $\xi _{\alpha_i} \neq 0$ and $\det \xi_{\alpha_i}=0$, then the matrix is nilpotent and defines a complex line
$L' \in \C\P^1$ related to the one dimensional complex subspace $\ker \xi_{0,\alpha_i} = {\rm Im}\;  \xi_{0,\alpha_i}$ (see appendix B). We
can decompose uniquely the potential as
$$\xi_{0,\l} =(L', \tilde \xi_{0,\l})= p(\l)h_{L',\alpha_i} \tilde \xi_{0,\l} h^{-1}_{L',\alpha_i}$$
where $\tilde \xi_{0,\l} \in I({\tilde a_i})$  has a determinant with a lower degree $\alpha_i$, and $h_{L',\alpha_i} (\l) \in \Lambda _r^+ \SL (\C)$ (see r-Iwasawa decomposition with $r<|\alpha_i|<1$  Proposition \ref{riwasawa}) is a matrix depending on parameter $\l$ with pole at $\alpha_i$ and $1/\bar \alpha_i$ and $h_{L',\alpha_i} (0)$ is an upper
triangular matrix. A property of this decomposition is that  $\xi_{0,\a_i}= 0$ if and only if the eigenline $(L')^\perp$ is an eigenline of $\tilde \xi_{0,\a_i}$.

There is a two dimensional subgroup action  $\tilde \pi: \C\times I(a) \to I(a)$ which preserves the second factor of the decomposition  $(L', \tilde \xi_{0,\l})\in \C\P^1 \times I({\tilde a_j})$ and acts transitively on the first factor $L' \in (\C\P^1)^{\times}$ (see Theorem \ref{bubbleaction}).  If $(L')^\perp$ is an eigenline of  $\tilde \xi_{\alpha_0}$, the potential $\xi_\l$ has a zero at $\l=\alpha_0$ and the subgroup $\tilde \pi$ acts trivially. The set  $(\C\P^1 )^{\times}$ is $\C\P^1$  minus  the fixed point of the action $\tilde \pi$.This decomposition implies that the orbit $O$ is dense in $I(a)$. The differential
structure of $O$ comes from the differentiability of the action away from its fixed point where $\tilde \pi$ extends only continuously.

Since the action preserves embeddedness (see Proposition \ref{bubbleembedded}), there is a continuous isospectral deformation of embedded annuli $A(t)$, induced by potentials $\xi_\l (t)= \tilde \pi (\beta(t)) \xi_{0,\l}$, such that $L'(\beta (t)) \to L'_0$ with $(L'_0)^{\perp}$ are eigenlines of $\tilde \xi_{0,\a_i}$.

At the limit we have an annulus induced by $\xi_{1,\l}=p(\l)h_{L'_0,\alpha_0} \tilde \xi_{0,\l} h^{-1}_{L'_0,\alpha_0}$ which is the limit of embedded annuli obtained from $\xi_\l (t)=\tilde \pi (\beta(t)) \xi_{0,\l}= p(\l)h_{L'(\beta(t)),\alpha_0} \tilde \xi_{0,\l} h^{-1}_{L'_(\beta(t),\alpha_0}$.
The potential $\xi_{1,\l}$ induces an embedded annulus but has roots. We can divide the potential by the polynomial
$p(\l)$ and decrease the spectral arithmetic genus of $\Sigma$. Doing this process inductively we can remove all higher order roots of $a(\l)$, one by one.

\vskip 0.3cm
As a result the classification of all embedded minimal annuli relies
on two problems:
\vskip 0.3cm
\noindent{\bf Problem A.} Classify all spectral data $(a,b)$ with $a$
having only simple roots and $I(a)$ corresponding to embedded minimal
annuli.
\vskip 0.3cm
\noindent{\bf Problem B.} Determine for all such spectral data $(a,b)$ all
possibilities to add bubbletons without destroying the
embeddedness.
\vskip 0.3cm
\noindent
{\bf Compactness of the space of embedded annuli and spectral data. } We consider the space of
embedded annuli of finite type with bounded spectral genus $g \leq g_0$ and flux bounded away from $0$ with the topology of the uniform convergence on compact sets:
$${\calM}_{g_0}^{emb} (\epsilon_0)=\{ X: \C/ \tau\Z \to \SY^2 \times \R ;  |\tau| \geq \epsilon_0, \hbox{X embedded}, (a,b,\xi_\l) \in \C^{2g}[\l] \times \C^{g+1}[\l] \times \pk, g \leq g_0\}$$

\begin{theorem}\label{compact}
A sequence $(X_n)_{n \in \N} \in {\calM}_{g_0}^{emb} (\epsilon_0)$ has a subsequence converging to $X_0 \in {\calM}_{g_0}^{emb} (\epsilon_0)$ and there is a constant $C_3 >0$ such that $|\tau| \leq C_3$ on ${\calM}_{g_0}^{emb} (\epsilon_0)$.
\end{theorem}
\begin{proof}
We consider a sequence of embedded minimal annuli $(X_n)_{n \in \N}$ with associated spectral data $(a_n,b_n,\xi_{\l,n})$. Since we assume that $|\tau_n| \geq \epsilon_0$ is bounded away from zero, the curvature is bounded and
$X_n$ is locally a graph over a tangent disc of radius $\delta >0$. There is a subsequence of immersions $X_{k}: \C \to \SY^2 \times \R$
which converge on compact sets to $X_0$ a minimal surface.  The surface $X_0 : \C \to \SY^2 \times \R$ is an annulus if we bound the period $|\tau|$ above to get a subsequence $\tau_{k}$ converging to the period of the annulus $X_0$. By the maximum principle if $X_0$ is an annulus, it is embedded and $\o_0$ is uniformly bounded on $X_0$ in norm ${\mathcal C}^{k,\a}$ (otherwise the annulus has vertical unit normal vector wich contradicts the maximum principle, see Proposition \ref{omegabound})
\vskip 0.3cm
By Lemma \ref{tubular}, the linear area growth satisfies
$${\rm Area}(X_n \cap \SY^2 \times[-t,t]) \leq C_2 (\epsilon_0) t.$$
Using the fact that the metric is $ds^2=\cosh^2 \o_n |dz|^2$, we conclude that $2t |\tau_n| \leq C_3
{\rm Area}(X_n([0,\tau_n] \times [-t,t]))$, and thus the flux is uniformly bounded 
on the space of embedded annuli  ${\calM}_{g_0}^{emb} (\epsilon_0)$. Hence any sequence of annuli $X_n$ has a subsequence
converging to an embedded annulus $X_0$.
\vskip 0.3cm
The sequence of immersions $X_{k}$ is induced by a sequence of potentials $\xi_{\l,k} \in {\calP}_{g_k}$ with $g_k \leq g_0$. 
Arguing as in Proposition \ref{omegabound},  the function $\o_{k}$ is uniformly bounded for $|\tau|�\geq \epsilon_0$
and the bound on the curvature of Proposition \ref{corcurva} implies that all roots of the polynomials $a_n$ are bounded away from $0$ and $\infty$ for any $n \in \bbN$. Since $I({a_n})$ are compact sets, there is a subsequence of potentials $\xi_{\l,k}$ which converge to a potential $\xi_{\l,0} \in I({a_0})$.
\vskip 0.3cm
The potential depends on the metric $\o_{k}$ and its higher derivatives at one point. By elliptic theory $\o_{k} \to \o_0$ in the ${\calC}^{k,\alpha}$ topology. This implies that the sequence of solutions of LSG  $u_1,u_2....$ converges on any compact set. After renormalization of coefficients, the sequence of algebraic relations
$$
    \sum_{i=1}^{g=g_k} a_{i,k} u_{i,k} + b_{i,k} \bar u_{i,k}=0
$$
converge to some similar algebraic condition on the solution $u_1,u_2,...$ of the metric $\o_0$.
The spectral genus of $X_0$ is bounded by $g_0$ and the sequence of potentials $\xi_{\l,k}$ converge to a potential $\xi_{\l,0}$ which induces the embedded annulus $X_0$  with  flux $\epsilon_0 \leq |\tau| \leq C_3$. Since the limit is an annulus of period $\tau_0$,
the monodromy eigenvalue $\mu$ is well defined and there is a subsequence of spectral data $(a_{k},b_{k})$ which converge to some
$(a_0,b_0) \in \C^{2g}[\l] \times \C^{g+1}[\l] $.
\end{proof}

\vskip 0.3cm
\noindent
{\bf Deformation of spectral data preserves embeddedness. } We consider a potential $\xi_{\l,0}$ which induces an embedded annulus $X_0(z)$ with spectral data $(a_0,b_0)$ and a sequence  of annuli $(X_{n} (z))_{n\in \N}$ which converge to $X_0(z)$ on compact sets. We denote by $(a_0,b_0)$ and $(a_n, b_n)$ the corresponding spectral data, and by $\xi_{\l,0}, \xi_{\l,n}$
the corresponding potentials with polynomial Killing field $\zeta_{\l,0} (z), \zeta_{ \l,n} (z)$. We assume that  $a_n \to a_0$, $b_n \to b_0$, $\tau_n \to \tau_0$ and $\xi_{\l,n} \to \xi_{\l,0}$. We prove in the following Proposition that embeddedness is an open property in regards to the spectral data.
\begin{proposition}
\label{embeddeddeformation}
If $X_0 (z)$ is a finite type embedded minimal annulus with isospectral set having only one
orbit (i.e.  $a_0(\l)$ has only simple roots or removable double roots), then there is $n_0>0$ such that for $n \geq n_0$, the annuli
$X_n$ are embedded.
\end{proposition}
 \begin{proof}
If $X_0 (z)$ is embedded then we prove that $X_n(z)$ is an embedded annulus for $n$ large enough.
Assume the contrary, that there is a subsequence of non embedded annuli $X_{k}$ which converge to $X_0$ on compact sets and there is a diverging sequence of real numbers $y_k$ such that $x \to X_{k}(x,y_k)$ is a self-intersecting horizontal curve in $\SY^2$.

The sequence of potentials $\pi(\mi y_k,0,...0) \xi_{\l,k}$ induce annuli $\tilde X_{k}(x,y)=X_{k}(x, y+ y_k)$.
The sequence $\pi(\mi y_k,0,...0) \xi_{\l,k}$ has a subsequence converging to an element $\widehat \xi_{\l,0}$ of $I({a_0})$ since
$a_{k} \to a_0$ and $I({a_n})$ are compact subsets of $\pk$ for any $n$.
Since the action is transitive on $I({a_0})$, there is $t \in \C^g$ such that $\widehat \xi_{\l,0}=\pi (t) \xi_{\l,0}$ and
$\widehat \xi_{\l,0}$ induces an embedded annulus by proposition \ref{simpleembedded}.

The Iwasawa decomposition is a real analytic diffeomorphism and $(t, \xi_\l) \to \pi(t)\xi_\l$ is uniformly continuous on compact sets from
$\C^g \times \pk$ to $\pk$.  Hence the immersions $x: [0, \tau_k] \to \tilde X_{k} (x,0)$ is smoothly
converging to  $x:[0, \tau_0] \to \widehat X(x,0)$. Since on an isospectral set the third coordinate $n_3$ of the normal is uniformly
bounded away from the value $n_3=1$, the trace of the tubular neighborhood $T_{\epsilon_1} = Y((\C \backslash \tau \Z)\times ] - \epsilon_1,\epsilon_1[)$ induce an embedded tubular neighborhood of the curve $x \to \widehat X(x,0)$ in $\SY^2$ with width depending only on $C_0$ and $\epsilon_0$. This gives a contradiction, since the curve $x \to \tilde X_{k} (x,0)$
is a graph in this tubular neighborhood for $k >n_0$ large enough.
\end{proof}

\begin{proposition}
\label{deformembbubble}
If $X_0 (z)$ is an embedded simple bubbleton with $a_0(\l)$  having only one higher order root  (i.e. $a_0 (\l)=(\l - \alpha_0)^2(1-\l \bar \alpha_0)^2 \tilde a_0 (\l)$ and
$\tilde a_0 (\l)$ has only simple roots), then there is $n_0>0$ such that for $n \geq n_0$, $X_n$ are embedded annuli.
\end{proposition}

\begin{proof}
We apply the same procedure as in preceeding Proposition. We consider a sequence of non embedded annuli $X_n$ which has
self intersection at $\{ x_3 =y_k\}$. Using a translation, we obtain a family of annuli $\tilde X_{n_k}(x,y)$ associate to some
potential converging to $\widehat \xi_{\l,0} \in I(a_0)$.
\vskip 0.3cm
In this case $I({a_0})=\{  \xi_\l \in I({a_0}) \mid \xi_{\alpha_0}=0\} \cup \{ \xi_\l \in I({a_0}) \mid \xi_{\alpha_0} \neq 0 \}$ has two orbits, and we apply the same argument. We  have to prove that any potential $\xi _\l \in I({a_0})$,  induces an embedded annulus, independently of the orbit.

Since the bubbleton is embedded, we proved in Proposition \ref{bubbleembedded} that any $\tilde \pi (\beta) \xi_{\l,0} = (L'(\beta), \tilde \xi_\l)$ induce an embedded bubbleton. Applying the procedure of removing bubbleton we can find an isospectral action which ends in the limit in a quasi-periodic annulus induced by $\tilde \xi_{\l,0} \in I({\tilde a_0})$. Since the action is continuous at the whole isospectral set, $\tilde \xi_{\l,0}$ induces an embedded annulus. Hence we conclude that there is a potential in each orbit which induces an embedded annulus. Then $I({a_0})$ is a set of potentials which induce embedded annuli and we can repeat the argument above.
\end{proof}
\begin{remark}
The sequence $(a_n,b_n)$ of spectral data can represent two different cases: (1) It can be a sequence of polynomials with simple roots. Two roots of $a_n$ coalesce in the limit at $\alpha_0$. (2) It can be a sequence of spectral data with a double root $\alpha_0(n)$ which induce a smooth deformation of bubbletons with $\alpha_0(n)$ converging to some double root $\alpha_0$. In this case, the Proposition shows that we can deform a simple bubbleton and preserve embeddness. We will use this deformation in the sequel to prove that there are no bubbletons on Abresch annuli.
\end{remark}

\begin{remark} We remark that we can pass from the isospectral set containing the
the bubbleton to the lowest dimensional orbit by passing to the limit.
This preserves embeddness. But conversely it is not possible to pass from the lowest dimensional orbit to the higher dimensional orbit and
preserve embeddness. This is the reason of the isolated property of the Abresh annuli.
\end{remark}

\noindent{\bf Section~\ref{abdeformation}.}
We adapt the Whitham deformation to our setting. Krichever defined
this deformation in the study of moduli space theory. This deformation
changes the spectral data in such a way that the corresponding annulus
flows along integral curves defined on the set of $(a,\,b)$.
We consider an embedded annulus with spectral data $(a,b)$ which satisfies the conditions of definition \ref{spectraldata}. We deform
the spectral curve $\Sigma$ in such a way that the immersion stays periodic. We have to find a deformation of $a(\l)$ by moving its roots without destroying the global properties
of the holomorphic function $\mu: \Sigma^* \to \C$ with prescribed essential singularities at the two marked points $(\infty, 0)$ and $(\infty, \infty)$. We consider the moduli space of spectral data $\{(a,b)\}$ and we derive vector fields on these data. The integral curves of this vector field are differentiable families of spectral data of
minimal cylinders. This is the content of section \ref{abdeformation}. We consider a real polynomial $c(\l)$ of degree at most $g+1$ with $\l^{g+1} \overline{c(1/ \bar \l)} = c(\l)$ and
the equation
$$\partial _t \ln \mu =\frac{c(\l)}{\nu \l }.$$
On the other hand we have the condition $d \ln \mu = \frac{b\,d \l}{\nu \l^2}$ and these equations are compatible on the integral curves if and only if
$$-2 \dot b a+b \dot a = -2 \l a c'+ac + \l a'c.$$
This equation defines the values of $\dot a$ at the roots of $a$ and $\dot b$ at the roots of $b$. This is well defined when $a$ and $b$ have no common roots. The closing
condition of the immersion is preserved if $c(1)=0$ (the Sym point stays at $\l =1$) and ${\rm Im}\; (c(0)/b(0)=0$ for the third coordinate to stay periodic (see theorem \ref{thm:deformation}).
\vskip 0.3cm
{\bf Connectedness of the space of embedded minimal cylinders of finite type.}
We prove in this section that every connected component of ${\calM}_{g_0}^{emb} (\epsilon_0)$ has only one connected
component which contains Abresch annuli. In the following Theorem we construct a sequence of 
minimal annuli for which the flux attain its maximum at a flat cylinder.

\begin{theorem}
\label{connectedtheorem}
Every connected component of the space ${\calM}_{g_0}^{emb} (\epsilon_0)$ contains
a sequence $(X_n)_{n \in \N} $ converging to the flat cylinder.
Hence the space ${\calM}_{g_0}^{emb} (\epsilon_0)$ has only one connected component
which contains Abresch annuli.
\end{theorem}
\begin{proof}
Consider a properly embedded minimal annulus $X:\C / \tau \Z \to \SY^2 \times \R$ with
spectral data $(a,b)$.  We consider the connected component of the space ${\calM}_{g_0}^{emb} (\epsilon_0)$, which
contains $X$. Choose in this connected component a sequence of annuli which maximize the flux $|\tau|$.

By Theorem \ref{compact}, there is a subsequence $(X_n)_{n \in \N}$ which converges to some $X_0 \in {\calM}_{g_0}^{emb} (\epsilon_0)$.  
This annulus is embedded, has spectral data $(a_0,b_0)$ and is induced by a potential $\xi_{\l,0}$.  Every connected component contains a such $X_0$ which maximizes the flux $|\tau|$.

Eventually $\Sigma_0$ has singularities i.e. $a_0(\l)$ has higher order roots and $\xi_{\l,0}$ may have zeroes.
In this case, we remove zeroes of the potential $\xi_{\l,0}$ and then obtain eventually a bubbleton. In this case we use the isospectral deformation which is transitive on an orbit of a bubbleton. We apply what we call "remove a bubbleton" to get at the limit an embedded annulus with spectral curve without singularities. Consequently we have an annulus with quasi-periodic metric $\o$. Proposition~\ref{bubbleclosing} guarantees that during this process the period $\tau$ does not change, as it is an isospectral invariant and we stay in the same connected component by continuity of the isospectral action.

We denote this annulus by $X_1$ with spectral data $(a_1,b_1)$ where $a_1$ has only simple roots. 
We claim that $X_1$ is flat. Otherwise we can deform this cylinder and increase the flux of $X_1$ by the following deformation:

Due to Proposition~\ref{embeddeddeformation} the deformation of the spectral data preserves the property that the corresponding isospectral sets
correspond to embedded minimal annuli. If $a$ and $b$ have no common roots we define a vector field $c(\l)$ over the spectral data $(a,b)$
preserving the closing condition. We remark that
$${\rm Re} \frac{c(0)}{b(0)}= -2 \partial_t \ln |\tau|$$
Then assuming that ${\rm Re} \frac{c(0)}{b(0)} <0$, we increase the flux along the deformation. This is possible if we have enough degree of freedom in the choice of the polynomial $c(\l)$ which is the case only if the spectral genus $g>0$.  For any local maximum of the flux, the spectral genus is zero, see Lemma \ref{increasetau}.

In the case where $a$ and $b$ have several common roots,  we look for a special reparameterization of spectral data $(a,b)$ by the value of the period map $\mu$ at its branch points. In this part we study singularities of the Whitham deformation and describe
the main difficulty of the theorem. We prove with the Stable Manifold Theorem that there is a trajectory which moves out of the singularities (see section \ref{param} and 9.3 and proposition \ref{commonroots}). We increase the flux when we pass through this singularities when we move out by choosing the value of $c(0)$.  Since $a$ has only simple roots, the elements of $I(a)$ have no roots on $\C^\ast$. Due to Theorem~5.8 \cite{HKS1} continuous paths of spectral data can be lifted in a neighbourhood of $(a,b)$ to continuous paths of potentials and annuli. 
\end{proof}

\begin{remark} The section 9.2 and 9.3 describe how the Whitham deformation apply to open double points on a spectral curve
and how to pass eventually singularities. This section describes how find at least two trajectories in the space
of spectral data which preserve the period and permit to pass trough the singularity. The reason is that the equation is associate
to a meromorphic vector field and we can find a blow up of the singularity.
\end{remark}
{\bf Section 10. Isolated property of Abresch annuli} Consider the family ${\calA}$ of Abresch annuli.  If ${\calM}_{g_0}^{emb} (\epsilon_0) \setminus {\calA} \neq \emptyset$,
we can find a sequence of embedded annuli $(X_n)_{n \in \N} \in {\calM}_{g_0}^{emb} (\epsilon_0 ) \setminus {\calA}$ converging to $ X_0 \in {\calA}$ with
spectral data $( a_0, b_0)$ and potential $\xi_{\l,0}$. By Theorem \ref{compact}, the spectral data $(a_n,b_n, \xi_{\l,n})$ has a converging subsequence with limit $(a_0,b_0, \xi_{\l,0})$.

If $(X_n) $ with spectral data $(a_n,b_n, \xi_{\l,n})$  is not foliated by constant curvature curves then the Shiffman's Jacobi field $u (n)$ of $X_n$ is not identically zero. The idea is to use $u (n)$ in order to construct a non zero bounded Jacobi field $v_0$ on $X_0$.
Using the four vertex theorem on embedded horizontal curve we conclude that $v_0$ has at least four zeroes and then $v_0 \equiv 0$, a contradiction to lemma \ref{degenerate}.

Since $X_n $ converges to $X_0$ on compact sets, $u (n)$ converges to zero on compact sets. Since $\o (n)$ is bounded in ${\mathcal C}^{k,\alpha}$ norm by Proposition \ref{omegabound} on the whole annulus, the Jacobi field $u (n) = \o_{xy}(n)-\tanh \o (n) \o_{x}(n)\o_{y}(n)$ is bounded on 
$X_n$. We consider $c_n :=\sup _{X_n} |u(n)|$ and

$$v (n):=\frac{u (n)}{c_n}\,.$$
Then there is a subsequence which converges by Arzela-Ascoli's theorem to a bounded function $v_0$ on $X_0$.
If there is a subsequence $n_k$  such that $u(n_k)$ takes its maximum value on a compact set $\{ -t \leq x_3 \leq t\}$,
then $v_0$ is not identically zero since the point where the supremum is attained has an accumulation point $z_0$ where $v_0(z_0) =1$.

In the case where $u(n)$ take its maximum value at infinity, $v_0$ could be zero. Consider the point $\bar z(n)=(\bar x(n), \bar y (n))$ on the annulus where $u(n) (\bar x(n),\bar y(n)) = \frac12  \sup _{X_n} |u (n)|$ and $\bar y(n) \longrightarrow \infty$. We consider  curves $\gamma (\bar y(n)) = X_n\cap \{y=\bar y(n)\}$ and we can find a divergent subsequence $y _{n_k}$, such that
$ X_{k}=\pi (iy_{n_k},0,...0)X_{n_k}(x,y)=X_{n_k}(x,y+y_{n_k})$ is converging to an annulus $X_1$, with spectral data $(a,b, \xi_{\l,1})$ (since $a_n \to a$ and $\tau_n \to \tau_0$).

Since the limit $X_0$ is geometrically an Abresch annulus, we have that $a(\l)=q(\l)a_0(\l)$ where
$$q(\l)= \prod_{i,j=1}^{i,j=k,l} (\l - \a_i)^2(1-\l \bar\a_i)^2 (\l - \beta_j)^2 \hbox{ with }   |\alpha_i | \neq 1, |\beta_j|=1.$$

{\bf Case 1 $X_1$ is an Abresch annulus. } If $q(\l)$ has only roots $|\beta_j|=1$ on the unit circle, then corresponding potential $\xi_{\l,1}$ has zeroes at these roots.
We can remove these roots without changing the extended frame, and $X_1$ is a finite translation of the Abresch annulus $X_0$. In the general situation we are not able to decide which geometric limit we have.

We parameterize $X_n$ and $X_1$ by its third coordinate. We denote by $ds_n=\cosh ^2 \o (n) |dz|^2$
and $ds_1=\cosh^2 \o(1) |dz|^2$ the associated metrics with $\o (n) \to \o(1)$ uniformly.
Jacobi operators are given by
$${\calL}_n = \frac{1}{\cosh ^ 2 \o (n)}\left( \partial^2 _{x} + \partial ^2_{y} +1 + 2 \frac{|\nabla \o (n)|^2}{\cosh ^2 \o (n)}\right).$$
Since $\o (n)$ is bounded by Proposition \ref{omegabound}, the Jacobi field $u (n) = \o_{xy}(n)-\tanh \o (n) \o_{x}(n)\o_{y}(n)$ is bounded on $X_n$ and
$$v_n:=\frac{u (n)}{\sup |u (n)|}$$
converges by Arzela-Ascoli's theorem to a bounded solution of ${\calL}_1v_0=0$, a bounded Jacobi function on $X_1$.
Since $u (n)=\cosh^2 \o (n) \partial_x k_g$ (see theorem \ref{shiffman}), it has at least four zeroes on any level horizontal curve by the four vertex theorem (see \cite{jackson}). The set of curves $\Gamma=\{v_n^{-1}(0)\}=\{u (n) ^{-1}(0)\}$ describes at least four nodal domain intersecting every horizontal curve (see theorem \ref{shiffman}). By counting the number of zeroes of $v_n$ on each horizontal section $x \to X(x,y_0)$, we deduce that $v_0$ cannot have generically two zeroes on horizontal curves and we argue as follow to get a contradiction.

Otherwise, it means that two or three zeroes of $v_n$ coalesce in the limit. We will find
an open interval $y \in ]t_1,t_2[$, such that on every curve $\gamma (t)= A \cap \SY^2 \times \{t\}$ the zeroes of $v_n$ are coalesce in the limit at two zeroes. A coalescing of zeroes will produce a new nodal curve $\Gamma_0=v_0 ^{-1} (0)$ generically
transverse to horizontal section of $X_1 \cap \SY^2 \times[t_1,t_2]$. We can find a horizontal section transverse to $\Gamma_0$.
Since $\Gamma_0$ is a limit of several nodal curves collapsing together at the limit, $v_0$ will not change sign along $\gamma (t)$ crossing $\Gamma_0$, or $v_0$  will change sign but with $\partial _x v_0=0$ on $\Gamma_0$. This contradicts a Theorem of Cheng \cite{cheng} on the singularity of nodal curves for the solution of an elliptic operator which are isolated and describing equiangular curves
at the singularity.

In summary the four vertex theorem implies that $v_0$ has at least four zeroes generically on each horizontal section. Now a careful analysis on the operator of $X_1$ will give a contradiction. We conclude by the Lemma \ref{degenerate}, that a
such Jacobi field cannot exist on an Abresch annulus $X_1$.

{\bf Case 2 $X_1$ is not an Abresch annulus. } In this case, the polynomial $a(\l)$ has an isospectral set $I(a)$ with several orbits. This situation occurs when there is a double root $|\alpha_i| \neq 1$. If $X_1$ is not an Abresch annulus, it implies that
$\xi_{\l,1}$, limit of isospectral action of $\xi_{\l,n}$ is in a higher dimensional orbit in $I(a)$. The annulus $X_1$ is a bubbleton
on an Abresch annulus.
\vskip 0.3cm
\noindent{\bf Section~\ref{abresch bubbleton}.}
Using our 'removing bubbletons' procedure, we can assume the
existence of an embedded simple bubbleton on the spectral curve of an
Abresch annulus. The idea explained in section \ref{abresch bubbleton} follows from two things. 

Firstly we deform a simple bubbleton of spectral genus 1 or 2 to a simple bubbleton on the flat cylinder. 
Secondly we open the double point of the bubbleton into a spectral curve of genus two preserving embeddedness. 
We can open a node on $\Sigma$ keeping the period closed, if the node is on a double point $\alpha_0$, where $\mu (\alpha_0) =\pm 1$.
Opening the node into a higher genus curve preserves embeddedness if and only if all orbits of the isospectral set 
induce embedded annuli i.e. if the bubbleton is embedded (by hypothesis).

We do this by increasing the flux $|\tau|$. But since we are on the flat cylinder, the period is already at least $2\pi$, and by the above, we are in a component which contains a genus zero example of flux strictly greater than $2\pi$, a multi sheeted covering of the flat annulus which is not
embedded. This is a contradiction since we preserve embeddness along the deformation by increasing the flux. The proof of the isolated is
finish.

\vskip 0.3cm
\noindent{\bf{Proof of the Main theorem}.} Consider an annulus $A$ not foliated by circle. It has uniform
bounded curvature and have finite type, hence there is $\epsilon_0>0$ such that $A \in {\calM}_{g_0}^{emb} (\epsilon_0)$ and
${\calM}_{g_0}^{emb} (\epsilon_0) \setminus {\calA} \neq \emptyset$. Now we apply, Theorem \ref{connectedtheorem} and the Section 10 and 11 on isolated property  to prove that in fact $A$ has to be an Abresch annulus.

\section{Minimal annuli and the sinh-Gordon equation}
\label{preliminary}
\noindent{\bf Local parametrization.}
We consider $X = (G,\,h): \C  \to \SY^2 \times \R$ a minimal surface conformally immersed in $\SY^2 \times \R$. As usual write $z = x +\mi y$. The horizontal component $G: \C \to \SY^2$ of the minimal immersion is a harmonic map. If we denote by $(\C,\si^2 (u)|\rmd u|^2)$ the complex plane with metric induced by the stereographic projection of $\SY^2$, the map $G$ satisfies
\begin{equation*}
    G_{z \bar{z}}+ 2(\log \si \circ G )_u G_z G_{\bar{z}}=0\,.
\end{equation*}
The holomorphic {\it quadratic Hopf differential} associated to the harmonic map $G$ is given by
\begin{equation*}
    Q(G)=(\si \circ G)^2 G_z {\bar G}_z ( \rmd z)^{2}:= \phi (z) (\rmd z)^2\,.
\end{equation*}
The function $\phi$ depends on $z$, whereas $Q(G)$ does not.
Conformality reads
\[
    \vert G _x \vert^2_{\si} + (h_x)^2  = \vert G _y \vert ^2_{\si} + (h_y)^2 \quad \mbox{and} \quad
    \left<G_x ,\,G _y \right>_{\si}+ (h_x)(h_y) = 0
\]
hence $(h_z)^{2}(\rmd z)^2=-Q (G )$. The zeroes of $Q$ are double, and we can define $\eta$ as the holomorphic 1-form $\eta =\pm 2i \sqrt{Q}$. The sign is chosen so that
\begin{equation*}
    h= {\rm Re} \int  \eta\,.
\end{equation*}
The unit normal vector $n$ in $\SY^2 \times \R$ has third coordinate
$$
    \langle n ,\,\tfrac{\partial}{\partial t}\rangle = n_{3} = \frac{\vert  g\vert^2-1}{\vert g\vert^2+1} \quad \mbox{where} \quad
    g^{2}:=-\dfrac{G_z}{\ov{G _{\bar z}}}\,.
$$
We define the real  function $\o : \C \to \R$  by
$$
    n_{3}:=\tanh \o\,.
$$
We express the differential $\rmd G$ independently of $z$ by:
\begin{equation*}
\rmd G= G_{\bar z} \rmd\bar z+G_{z} \rmd z=
\frac{1}{2 \si \circ G}\overline{ g^{-1} \eta} -\frac{1}{2
\si \circ G} g \,\eta
\end{equation*}
and the metric $\rmd s^2$ is given (see \cite{[S-T]}) in a local coordinate $z$ by
\begin{equation*}
    \rmd s^2=(\vert G_{z} \vert_\si + \vert G_{\bar z} \vert_\si ) ^2 \vert \rmd z\vert ^2=\frac14(\vert g\vert^{-1}+\vert g \vert)^{2}\vert \eta \vert ^{2}=
    4 \cosh ^2 \o  |Q|\,.
\end{equation*}
We remark that the zeroes of $Q$ correspond to the poles of $\o$, so that the immersion
is well defined. Moreover the zeroes of $Q$ are points, where
the tangent plane is horizontal. The Jacobi operator is
$$
    \calL = \frac{1}{4|Q|\, \cosh ^2 \o}\left(\partial^2_{x}+\partial^2_{y} +{\rm Ric}(n) + \vert dn \vert ^2 \right)
$$
and can be expressed in terms of $Q$ and $\o$ by
\begin{equation} \label{jacobi}
    \calL = \frac{1}{4|Q|\,\cosh ^2 \o}\left(\partial^2_{x}+\partial^2_{y} +4|Q|+\frac{2 | \nabla \o |^2}{\cosh^2 \o} \right)
\end{equation}
Since $n_3= \tanh \o$ is a Jacobi field obtained by vertical translation in $\SY^2 \times \R$, we have $\calL \tanh \o=0$ and
$$
    \Delta \o +|Q| \sinh \o \cosh \o=0\,,
$$
where $\Delta = \partial^2_{x}+\partial^2_{y}$ is the Laplacian of the flat metric.


{\bf Minimal annuli.}
Consider a minimal annulus $A$ properly embedded in $\SY^2 \times\ R$.
If $A$ is tangent to a horizontal section $x_3=0$, the set $A \cap \{ x_3 =0 \}$ bounds on $A$ a compact component in some half-space $x_3 \geq 0$ or $x_3 \leq 0$ with boundary in $\SY^2 \times \{0\}$, a contradiction to the maximum principle. Hence the annulus is transverse to every horizontal section $\SY^2 \times \{t\}$ and intersect the level section in one compact connected component, topologically a circle. The third coordinate map $h : A \to \R$  is a proper harmonic map on each end of $A$, with $dh \neq 0$. Then each end of $A$ is parabolic and the annulus can be conformally parameterized by $\C / \tau \Z$. We will consider in the following conformal minimal periodic immersions $X: \C  \to \SY^2 \times \R$ with $X(z+\tau)=X(z)$.

Since $dh \neq 0$, the Hopf differential $Q$ has no zeroes. If $h^*$ is the harmonic
conjugate of $h$, we can use the holomorphic map $\mi(h+\mi\,h^*) : \C^2 \to \C$ to parameterize
the annulus by the conformal parameter $z=x+\mi y$. In this parametrization the period of the annulus is $\tau \in \R$ and
$$
    X(z)=(G(z),\,y) \hbox{ with } X(z+\tau)=X(z)\,.
$$
We say that we have parameterized the surface by its {\it third component}.
We remark that $Q=\frac14 (dz)^2$ and $\o$ satisfies the $\sinh$-Gordon equation \eqref{sinh}.
\begin{remark}
In this paper we will relax the condition $\tau \in \R$ into $\tau \in \C$, but we will parameterize our annuli conformally such that $Q$ will be constant, independent of $z$ and $4|Q|=1$.
We will say that the annulus is parameterized by its Sym point (see proposition \ref{sym}). This is a linear  change in the conformal parameter $z \to e^{\mi\Theta} z$
\end{remark}
In summary we have proven the following
\begin{theorem}\label{immersionsinhgordon}
A proper minimal annulus is parabolic and $X: \C/ \tau \Z \to \SY^2 \times \R$ has conformal parametrization  $X(z)=(G(z),\,h(z))$ with
\begin{enumerate}
\item constant Hopf differential $Q= \frac{1}{4} \exp(\mi\Theta)\,dz^2$, and $h(z) = {\rm Re} (-\mi e^{\mi\Theta/2} z)$.
\item The metric of the immersion is $ds^2=\cosh^2(\o)\, dz\otimes d\bar z$.
\item The third coordinate of the unit normal vector is $n_3 =\tanh \o$.
\item The real function $\o : \C/\tau \Z \to \R $ is a solution of $\sinh$-Gordon equation \eqref{sinh}.
\end{enumerate}
In particular the annulus intersect each horizontal section $\SY^2 \times \{t\}$ in exactly one embedded topological circle.
\end{theorem}
\noindent{\bf Annuli foliated by constant curvature curves.}
The function $\o$ determines the geometry of the annulus. We are interested in a 2-parameter family
of minimal annuli foliated by horizontal curves with constant geodesic curvature. In the case of minimal surfaces in $\R^3$ Shiffman introduced the function $u=-\lambda \partial_x ( k_g)$, called Shiffman's Jacobi field. In $\SY^2 \times \R$, this function
exists, and we cite the following 
\begin{theorem} \cite{hau} \label{shiffman}
Let $A$ be a minimal surface embedded in  $\SY^2 \times \R$, transverse to every section of $\SY^2 \times \{t\}$ and parameterized by the third coordinate. Then the geodesic curvature in $\SY^2$ of the horizontal level curve $\g_h (t)=A \cap (\SY^2 \times \{t\})$
is given by
\begin{equation}
k_g (\g_h) = \frac{-\o _y}{\ch }.
\end{equation}
The function $u= - \ch  (k_g)_x$ is a Jacobi field, so that $u$ is a solution of the elliptic equation
$$
    \calL u = \Delta _g u + {\rm Ric}(n)u + \vert dn \vert ^2 u=0\,.
$$
Here ${\rm Ric}(n)$ is the Ricci curvature of the two planes tangent to $A$,
$\vert dn \vert$ is the norm of the second fundamental form and
$\Delta _g= \frac {1}{\rho^2}\Delta$.

Let $A$ be a compact minimal annulus immersed in $\SY^2 \times \R$,
with ${\rm Index } (L) \leq 1$. If $A$ is bounded by two curves $\Gamma _1$ and  $\Gamma _2$
with constant geodesic curvature, then $u$ is identically zero and $A$ is foliated by horizontal curves of constant curvature in $\SY^2$.
\end{theorem}
This theorem proves the existence of a 2-parameter family of minimal surfaces foliated by horizontal constant geodesic curvature curves which are equivalent to Riemann's minimal example of $\R^3$. Meeks and Rosenberg \cite{MR} prove the existence by solving a Plateau problem between two geodesics $\Gamma_1$ and $\Gamma_2$ contained in two horizontal sections $\SY^2 \times \{t_1\}$ and $\SY^2 \times \{t_2\}$ and then find a stable minimal annulus bounded by the geodesics. By the theorem above
this annulus is foliated by horizontal circles and using symmetries along horizontal geodesics in $\SY^2 \times \R$, one obtains a properly embedded minimal annulus in $\SY^2 \times \R$. This annulus is periodic in the third direction.
\begin{proposition} \cite{hau}
A minimal annulus foliated by constant curvature horizontal curves admits a parametrization by the third coordinate where the metric $ds^2=\cosh(\o)\, |dz|^2$ satisfies the Abresch system
\[
    \left\{\begin{array}{ll}
    \Delta \o  +\sinh  \o \cosh \o =0 & \\[1mm]
    u=\cosh \o \partial_x(k_g)= (\o_{xy}) -\th ( \o_x)( \o_y) =0.&
\end{array} \right.
\]
\end{proposition}
Abresch \cite{abresch} solved this system using elliptic functions and the separation of variable
$$
    \partial_x \frac{\o_y}{\cosh \o}=\partial_y \frac{\o_x}{\cosh \o}=\cosh ^{-1} \o ( \o_{xy} -\th \o_x \o_y)=0\,.
$$
The solution $\omega : \C \to \R$ allows to reconstruct the immersion up to the isometry (see section \ref{symbobenko}). The period closes in $\C$ because horizontal curves are circles.
\begin{theorem}\label{abresch} \cite{abresch}, \cite{hau}
Let  $\o : \R^2 \to \R$ be a real-analytic solution of the Abresch system, then we define the functions $f,\,g$ in terms of $\o$ by
\begin{equation}
    f = \frac{-\o _x}{\ch }\quad \mbox{and}\quad g = \frac{- \o _y}{\ch }
\label{eq:3.2}
\end{equation}
Then the real functions $x \mapsto f(x)$ and $y \mapsto g(y)$ depend
on one variable and for $c\leq 0, d \leq 0$ satisfy the system

\[
\begin{array}{ll}
    - (f_x)^2=f^4 +(1 +c-d )f^2 + c \,,\quad
    -f_{xx}= 2f^3 +(1 +c-d )f \,,& \\
    - (g_y)^2=g^4 +(1 +d-c )g^2 + d \,,\quad
    -g_{yy}= 2 g^3 +(1 +d-c )g \,.&
\end{array}
\]
Conversely we can recover the solution $\o$ from functions $f$ and
$g$ by
\begin{equation}
\sinh \o=(1 +f^2+g^2)^{-1}(f_x + g_y)
\label{eq:3.1}
\end{equation}
There is a solution of the system if and only if $c\leq 0$ and $d\leq 0$,  $\o$ is doubly periodic and
exists on the whole plane $\R^2$.
\end{theorem}
%


\section{Curvature estimate of Meeks and Rosenberg}
\label{sec:curvature estimate}
Meeks and Rosenberg \cite{MRstable} study properly embedded minimal annuli
in $\SY^2 \times \R$. They prove an estimate of curvature in terms of the third coordinate of the flux.
\begin{lemma} \label{flux}
Let $\gamma$ be a simple curve (embedded) non homotopically trivial in $A$ and let
$\eta=J \gamma '$ be a unit vector field tangent to $A$ and orthogonal to $\gamma '$ along
$\gamma$, then we consider $\eta_3=\langle \eta, \frac{\partial}{\partial t}\rangle$ and the third coordinate
of the flux map
$$F_3= \int_{\g} \eta_3 ds.$$
If the annulus $A$ is conformally parameterized with $Q=\tfrac{1}{4}e^{\mi\Theta}(dz)^2$ and $X(z+\tau)=X(z)$ then $F_3 ( \gamma)=\pm |\tau|$.
\end{lemma}
\begin{proof}
After a conformal change of coordinate we assume that the annulus is parameterized by its third coordinate with real period $|\tau|$. We compute the flux along a horizontal curve $x \to X(x,y_0)$. The co-normal along this curve is $\eta= \mathrm{sech}\,(\o)\,(G_y,1)$.
Hence
$$
    F_3= \int_0^{|\tau|} \eta_3 \,ds= |\tau|\,.
$$
\end{proof}
Meeks and Rosenberg proved in theorem 7.1 in \cite{MR}
the following curvature estimate
\begin{theorem} \label{curvest}\cite{MR}
For any properly embedded minimal annulus $A$ in $\SY^2 \times \R$ with flux $|F_3| \geq \epsilon_0>0$,  there exists a constant $C_1 >0$ depending only on $\epsilon_0$ such that $|K| \leq C_1(\epsilon_0)$.
\end{theorem}
They prove a linear growth estimate for minimal surfaces embedded in general product spaces $M \times \R$, but in $\SY^2 \times \R$
the annulus is parabolic and we can improve the result with a recent result of Mazet \cite{mazet} and an estimate of Heintze-Karcher \cite{HeintzeKarcher}. This implies that geometrically, an embedded annulus has an uniform tubular neighborhood.
\begin{lemma}
\label{tubular}
If $X: \C / \tau \Z \to \SY^2 \times \R$ is an embedded annulus, then $X$ is the restriction
of an $\epsilon_1$-tubular embedded neighborhood $T_{\epsilon_1}$ of the annulus i.e. there is $\epsilon_1 >0$ such that
$$
    Y : (\C / \tau \Z) \times ]-\epsilon_1,\epsilon_1[ \to \SY^2 \times \R \hbox{ with } Y(z,s)={\rm Exp}_{X(z)}(s\,n(z))
$$
is an embedded three dimensional manifold $T_{\epsilon_1}=Y((\C / \tau \Z) \times ]-\epsilon_1,\epsilon_1[)$   into $\SY^2 \times \R$. This constant $\epsilon_1$ depends only on a lower bound of the flux $F_3=|\tau| \geq \epsilon_0>0$. Thus for any $t>0$, there is a constant $C_2 >0$ which depends only on $\epsilon_0$ such that
$${\rm Area}(X \cap \SY^2 \times[-t,t]) \leq C_2 (\epsilon_0) t$$
\end{lemma}
\begin{proof}
We denote $A(s)=Y(\C/ \tau \Z, s)$. Following \cite{HeintzeKarcher} the differential of the exponential map $Y: (z,s) \to {\rm Exp}_{X(z)} (s n(p))$ is uniformly bounded on $\C \times [-\epsilon_1,\epsilon_1]$ for $\epsilon_1 >0$  depending only on the geometry of $\SY^2 \times \R$, and the upper bound of the Gaussian curvature $K$ of $A(0)=X(\C / \tau \Z)$. Then the projection $\pi_s$ along geodesics of the equidistant surface  $A(s)$ to $A(0)$ is a quasi-isometry. There is a constant $K_1$ such that if $|K| \leq C_1(\epsilon_0)$ and $s  \in [-\epsilon_1, \epsilon_1]$ we have $K_1^{-1}(\epsilon_0) |v| \leq |d\pi_s(v)| \leq K_1 ( \epsilon_0) |v|$ for any
$v \in T_{Y(z,s)} A(s) $.

Since the Ricci-curvature is positive, we have $\tfrac {d}{ds} H_s = ({\rm Ric}(\partial_s) + |dn_s|^2) \geq 0$ and the equidistant
surface $A(s_0)$ has mean curvature vector pointing outside the tubular neighborhood $T_{s_0}$.

Each equidistant surface $A(s)$ has shape operator which satisfies a Riccati-type equation, hence by Karcher \cite{Karcher},
the second fundamental form of $A(s)$ is uniformly bounded on $[-\epsilon_1,\epsilon_1]$.

We satisfy the hypothesis of theorem 7 of Mazet \cite{mazet}. If there is a parabolic annulus $X$ such that
$T_{\epsilon_1}=Y((\C / \tau \Z) \times [\epsilon_1,\epsilon_1[ )$ is not embedded, a subregion of $A$ would produce a connected
component $S$ bounded or unbounded into the tubular neighborhood $T_{\epsilon_1}$, which contradicts the maximun principle
(see  \cite{mazet}).

This uniform bound of the minimal width of the tubular embedded neighborhood of the surface gives a linear growth estimate.
There is  a constant $C$ depending only on the geometry of $\SY^2 \times \R$ and $\epsilon_0$ such that
$$\epsilon_1 {\rm Area} (X \cap  \SY^2 \times [-t,t]) \leq  C  {\rm Vol}( T_{\epsilon_1} \cap [-t,t]) \leq 4 C \pi t$$
and the constant $C_2=4C \pi / \epsilon_1$ depends only on $\epsilon_0$.
 \end{proof}
In the following, we will deform minimal annuli keeping $F_3$ bounded away from zero. Then the curvature of the annulus will remain uniformly bounded. As a corollary we derive a uniform estimate for $\o$, hence for the third coordinate of the normal $n_3=\tanh \o$.
\begin{proposition}
\label{omegabound}
If $A$ is a properly embedded minimal annulus in $\SY^2 \times \R$ with induced metric
$ds^2=\cosh ^2 \o |ds|^2$, then there exist a constant $C_0>0$ such that for any $k \in \N$
$$
    |\o|_{A, k,\alpha} \leq C_0\,.
$$
Since $n_3= \tanh \o$ is bounded away from the value $n_3=1$, the intersection $T_{\epsilon_1} \cap (\SY^2 \times \{t\})$
is a tubular neighborhood in $\SY^2$ of the level curve $A \cap (\SY^2 \times \{t\})$ with a width $\epsilon_2 >0$ uniformly bounded
above by a constant $c>0$ depending only in $\epsilon _1$ and $C_0$.
\end{proposition}
\begin{proof}
Consider a sequence of points $p_n \in A$ such that $\o (p_n)$ is diverging to infinity, and consider
a sequence of translations $t_n e_3$ such that $A + t_n e_3$ is a sequence of annuli with $p_n + t_n e_3 \in \SY^2 \times \{0\}$. Then by the curvature estimate of Meeks-Rosenberg there is a subsequence
converging locally to an embedded minimal surface $A_0$ in $\SY^2 \times [-t,t]$.
The area estimate shows that $A_0$ is an annulus, with the same flux $F_3$. By our hypothesis
this leads to a pole occurring at the height $t =0$ since $|\o| \to \infty$. The limit normal vector
$n_3 (p_n) =\tanh \o_n (p_n) \to \pm 1$ and the annulus $A_0$ would be tangent to the height $\SY^2 \times \{0\}$, a contradiction to the maximum principle. Then
$$
    \sup _{z\in A} |\o| \leq C_0\,.
$$
Now we apply the Schauder estimate to the sinh-Gordon equation to have a $\calC^{k,\a}$ estimate on the solution of the sinh-Gordon equation on $\R / \tau \ Z \times \R$.
\end{proof}

Adapting an argument of Lockhart-McOwen \cite{lockhart-mcowen}, Meeks-Perez-Ros \cite{MPR} prove the following:
\begin{theorem}
\label{finitekernel}
An elliptic operator $L u= \Delta u + q u$ on a cylinder $\SY^1 \times \R$  has for
bounded and continuous $q$ a finite dimensional kernel on the space of uniformly bounded
$\calC ^2$ functions on $\SY^1 \times \R$.
\end{theorem}
%
%

\section{Finite type theory of the sinh-Gordon equation}
\label{pinkallsterling}
\noindent{\bf Pinkall-Sterling induction.}
Suppose $\o$ is a solution of the sinh-Gordon equation \eqref{sinh}. There is an iteration of Pinkall-Sterling \cite{PS} to obtain a hierarchy of solutions of the linearized sinh-Gordon equation \eqref{eq:lsg}. This inductively defines an infinite sequence of formal solutions of the linearized sinh-Gordon equation.  To construct this family, we consider the matrix valued 1-form
\begin{equation} 
 \alpha (\o) = \frac{1}{4}\,\begin{pmatrix}
    2 \o_z &  \mi\;  \lambda^{-1}  e^\o\\
  \mi \; \gamma e^{-\o}&  -2\o_z
  \end{pmatrix}\, dz + \frac{1}{4}\,\begin{pmatrix}
-2 \o_{\bar{z}} &
 \mi\,\bar{\gamma} e^{-\o}\\
 \mi\,  \lambda  \,e^{\o}&
2\o_{\bar{z}}
  \end{pmatrix} d \bar{z}
  \end{equation}
\begin{definition}\label{finite type}
We call a solution of the sinh-Gordon equation $\o : \C \to \R$ of finite type, if there exists $g \in \N$ and $\gamma \in \bbS^1$ and functions $u_n, \tau_n, \sigma_n : \C \to \C$ such that
\begin{equation} \label{potentielphi}
  \Phi_\l  = \frac{\l^{-1}}{4} \begin{pmatrix}   0 &  \mi e^{\o}   \cr 0 & 0 \end{pmatrix} +
  \sum_{n=0}^g \l^n \begin{pmatrix} u_n(z) & e^{\o} \tau_n (z) \\
  e^{\o} \sigma_n (z) & -u_n (z)  \end{pmatrix}
\end{equation}
is a solution of the Lax equation equation $d \Phi _\l =  [\Phi_\l,\,\a (\o)]$.
\end{definition}
\noindent
Pinkall and Sterling solve this problem and we can summarize their result in the following
\begin{proposition} \label{induction}
Suppose $\Phi_\l$ is of the form (\ref{potentielphi}) and $\a(\o)$ defined by \eqref{matrixalpha} for  a given real function $\omega : \C \to \R$. If $\Phi_\l$ satisfies $d\Phi _\l= [\Phi_\l\,, \a (\o)]$ then:
\begin{enumerate}
\item  The function  $\o$ is a solution of the sinh-Gordon equation $\Delta \o + \cosh \o \sinh \o=0$.
\item For all $n \in \N$, the function $u_n$ solves the linearized sinh-Gordon equation $\Delta u_n + \cosh (2\o) u_n =0$.
\item For given $u_n, \sigma_n, \tau_{n-1}$, with $u_n$ solution of \eqref{eq:lsg}, we solve the system
\[
    \tau_{n;\bar z}=\tfrac{1}{2}\mi \bar \gamma e^{-2\omega} u_n \,, \qquad
    \tau_{n; z}=- 2 \mi\bar\gamma\, u_{n;zz} + 4\mi\bar\gamma\,\omega_z u_{n;z}
\]
and then define
\[
    u_{n+1} = -2\mi \tau_{n;z}-4\mi \omega_z \tau_n\,, \qquad
    \sigma_{n+1}=\gamma\,e^{2\omega} \tau_n + 4\mi\gamma u_{n+1; \bar z}
\]
to obtain $u_{n+1},\,\sigma_{n+1},\,\tau_n$. This defines a formal solution $\Phi_\l$ of the Lax equation.
\item At each step, $\tau_n$ is defined up to a complex constant $c_n$ and
\item If we consider $u_n = \bar u_{n+1},  \sigma_n = \bar \tau_n , \tau_{n-1}= \bar \sigma_{n+1}$, then the iteration procedure of {\rm{(3)}} gives $u_{n+1}=\bar u_n , \sigma_{n+1}= \bar \tau_{n-1}, \tau_n= \bar \sigma_{n}$ and
$$
    \tilde \Phi_\l (z) := (\l^{g-1}){}^t \overline{ \Phi_{1/ \bar \l}(z)} $$ is also a solution of $d \Phi _\l= [\Phi_\l,\,\a(\o)]$.
\end{enumerate}
\end{proposition}
Pinkall-Sterling \cite{PS} prove that beginning the iteration with $u_{-1}=0, \sigma_{-1}=0, \tau_{-1}=1/4$ and $u_0=\omega_z$, then all solutions depend only on $\o$ and its $k$-th derivatives with $k \leq 2n+1$ (see Proposition 3.1 in \cite{PS}).  Each function constructed in this way is complex valued, and $u_n : \C \to \C$ has real and imaginary part that are both solutions of \eqref{eq:lsg}. Applying this iteration procedure with $u_{-1}=0$, we obtain the sequence with first terms
\begin{align*}
    u_{-1}&=0,\;\; u_{0}= \o_z ,\;\; u_{1}= (\o_{zzz}-2\o_z^3)\\
    u_{2} &= (\o_{zzzzz}-10 \o_{zzz}\o_{z}^3-10 \o^2_{zz} \o_{z} + 6\o_{z}^5) ,\;....
\end{align*}
Since we consider a uniform bounded solution of the sinh-Gordon equation $\o : \R\times \SY^1  \to \R$, this infinite sequence produces bounded solution of the linearized sinh-Gordon equation on $A$ by Schauder estimates. Hence this sequence is a finite dimensional family (by theorem \ref{finitekernel}) and there is a $g \in \N$ and complex coefficient $a_i, b_i$ such that
$$
    \sum _{i=0}^g   a_i u_{i} +   b_i \bar u_{i}  =0\,.
$$
By prescribing the right constants $c_0,\,c_1,...$ in the iteration procedure the $g+1$ term is zero.
Thus the relation above implies the existence of a polynomial solution $\Phi_\l$ of degree $g$, hence $\omega$ is of finite type.
\begin{theorem} \label{parabolic}
A properly embedded minimal annulus in $\SY^2 \times \R$ is of finite type
\end{theorem}
\begin{proof}
Due to the Meeks-Rosenberg theorem, the curvature of this annulus is bounded by \eqref{eq:curvature bound}
where $F_3$ is the third coordinate of the flux. From this estimate and a slide-back sequence we prove (see Lemma \ref{omegabound}) that the third coordinate defined by $n_3= \tanh \o$ is bounded away from one and $\sup_A |\o| \leq C_0$.

Since $n_3 \neq 1$, the tangent plane is nowhere horizontal and $(h_z)^2=-Q \neq 0$ on the annulus. By Theorem \ref{immersionsinhgordon} the properly embedded annulus is parabolic and we can re-parameterize by its third coordinate. We consider an immersion $X : \C \to \SY^2 \times \R$ with $X(x+\mi y)=(G(x+\mi y),y)$ and period $\tau \in \R$ defined by the smallest positive value such that $X(z+\tau)=X(z)$. The metric is given by  $ds^2 =\cosh^2 \o |dz|^2$ and $n_3= \tanh \o$ is the third coordinate of the normal. The function $\o :\C \to \R$ is uniformly bounded and satisfies the sinh-Gordon equation \eqref{sinh}. Schauder estimates apply and we have
$$
    |\o|_{\calC ^{k,\alpha}}\ \leq C_1
$$
for a constant $C_1$ depending on the annulus $A= X(\C / \tau \Z)$. Now we apply the theorem \ref{finitekernel}  which assures that the operator
$$
    \Delta + \cosh(2 \o) : \calC ^{2,\alpha} (\C / \tau \Z) \to  \calC ^{0,\alpha} (\C / \tau \Z)
$$
has bounded dimensional kernel on the subspace of uniformly bounded ${\mathcal C}^2$ functions on the annulus $A$. Hence $\o$ is of finite type by the Pinkall-Sterling iteration.
\end{proof}

\noindent{\bf Polynomial Killing field and potential.} For $\Phi_\l$ in \eqref{potentielphi} we choose the constant  $\gamma \in \SY^1$ in $\a_\l$ (see formula \ref{matrixalpha}) such that the residue at $\l =0$ of
$$
    \zeta_\l (z) = \Phi_\l (z) - \l^{g-1}\;\overline{ \Phi_{1/ \bar \l} (z)}^t
$$

is of the form $\bigl( \begin{smallmatrix} 0 & \mi \R^\times \\ 0 & 0 \end{smallmatrix} \bigr)$, so that $\zeta_\l$ takes values in the space of potentials $\pk$ (see Definition~\ref{def pot}).
\begin{definition}\label{pkf}
Polynomial Killing fields are maps $\zeta_\l : \C \to \pk$
which solve
$$d\zeta_\l  = [\zeta_\l,\,\alpha(\o)] \hbox{ with } \zeta_\l (0) = \xi_\l \in \pk.$$
\end{definition}
When $g$ is even, $\xi_0,...,\xi_{\frac{g}{2}-1}$ are independent $2 \times 2$ traceless complex matrix and the real vector space $\pk$ is of dimension $3g+1$, and has up to isomorphism a unique norm $\|\cdot\|$. For an odd $g$, the difference is in $\hat \xi_{\frac{g-1}{2}} \in \su (\C)$, the Lie algebra of $\su (\C)$. These Laurent polynomials $\xi_\l$
define smooth mappings $\xi : \l \in \SY^1 \to \Sl(\C)$.


\begin{remark} \label{reality}
The polynomial $a(\l) : = - \l \det \xi_\l$ satisfies the reality condition
\begin{equation} \label{eq:real}
    \l^{2g}\overline{ a (1/\bar \l)} = a (\l)\,.
\end{equation}
Since $\chi_\l=  \lambda^{\frac{1-g}{2}} \xi_\l$ is traceless and satisfy ${}^t \overline{ \chi _{1/\bar \l }} = -\chi _{\l}$ for any
$\xi_\l \in \pk$ and for $\l \in \SY^1$, the determinant is the square of a norm and we have $||\xi_\l||=||\chi_\l||=\det \chi_\l \geq 0$ for $\l \in \SY^1$. Thus
\begin{equation} \label{eq:hermit}
    \l^{-g}\,a(\l) \leq 0 \hbox{ for } \l \in \SY^1
\end{equation}
The condition ${\rm trace}(\hat \xi_{-1} \hat \xi_0)$ implies that $a(0)\neq 0$ and by symmetry
the highest coefficient of $a$ is non-zero (see definition \ref{pks} to see that ${\rm trace}(\hat \xi_{-1} \hat \xi_0)=\beta_{-1}\gamma_0=a(0)$).
\end{remark}
{\bf Spectral curve.} The spectral curve is defined by the determinant of a polynomial Killing field $\zeta_\l$. The main property of the Lax equation is that $a(\l)=-\l \det \zeta _\l= -\l \det \xi_\l$ is independent of $z$. Following Bobenko (\cite{bob}), the polynomial $a(\l)$ defines a hyperelliptic Riemann surface $\Sigma$ as follows:
\begin{definition}\label{spectral curve}
The spectral curve $\Sigma$ of genus $g$ associate to a potential
$\xi_\l \in\pk$ is defined by adding $(\infty,0)$ and $(\infty,\infty)$ as
branch points in the compactification of \eqref{eq:spectral curve}
\begin{equation*}
\Sigma^* = \{(\nu,\l) \in \C^2 \mid \det(\nu\; \un - \zeta_\l)=0\} = \{
(\nu,\l) \in \C^2\mid \nu ^2 =\lambda^{-1} a(\lambda)=-\det \xi_\l \}
\end{equation*}
where $a$ is a polynomial with $2g$ pairwise distinct roots which
satisfies the reality conditions
\begin{equation*}
|a(0)|=\frac{1}{16}, \;\;\;  \lambda^{2g} \overline{a(\bar{\lambda}^{-1})} =a(\lambda) \hbox{ and } \frac{a(\l) }{\l ^g} \leq 0 \hbox{ for all } \l \in \SY^1.
\end{equation*}
\end{definition}
\begin{proposition}
\label{involution} $\Sigma$ has three involutions
\begin{align}\label{eq:involutions}
\sigma&:(\lambda,\nu)\mapsto(\lambda,-\nu)&
\rho&:(\lambda,\nu)\mapsto(\bar{\lambda}^{-1},- \bar{\lambda}^{1-g}\bar{\nu})&
\eta&:(\lambda,\nu)\mapsto(\bar{\lambda}^{-1},\bar{\lambda}^{1-g}\bar{\nu})&
\end{align}
The involution $\sigma$ is the hyperelliptic involution. The involution $\eta$ has no fixed point while $\rho$ fixes all points of the unit circle $|\l|=1$.
In particular, roots $\a_1,\,\ldots ,\a_{2g}$ of $a(\l)$ are symmetric with respect to the unit circle so that $a(\alpha_i)=0$ if and only if and $a(1/\bar{\alpha}_i) =0$.
\end{proposition}
%
%
%
\section{Construction of minimal annuli via potentials}
\label{sec:polynomial killimg field}
We next explain how to reconstruct an immersion from
a potential $\xi_\l \in \pk$. Expanding a polynomial Killing field $\zeta_\l: \C \to \pk$ as
\begin{equation} \label{eq:zeta-expansion}
    \zeta_\l = \begin{pmatrix} 0 & \beta_{-1} \\  0 & 0 \end{pmatrix} \lambda^{-1} + \begin{pmatrix} \alpha_0 & \beta_0 \\ \gamma_0 & -\alpha_0 \end{pmatrix} \lambda^0 + \ldots + \begin{pmatrix} \alpha_g & \beta_g \\ \gamma_g & -\alpha_g \end{pmatrix} \lambda^g
\end{equation}
we associate a matrix 1-form defined by
\begin{equation} \label{eq:alpha-zeta}
    \a(\zeta_\l)=\begin{pmatrix}
    \a_0 & \beta_{-1}\l^{-1}  \cr
    \g_0 & -\a_0
    \end{pmatrix} {\rm d}z -\begin{pmatrix}
    \bar \a_0 & \bar \g_0  \cr
    \bar \beta_{-1} \l & - \bar \a_0
    \end{pmatrix} {\rm d} \bar z
\end{equation}
We cite Proposition 3.2 \cite{HKS1} which provides the following existence and uniqueness result
\begin{proposition} \cite{HKS1}
\label{lax}
For each $\xi_\l\in\pk$ there is a unique solution
$\zeta_\l (z) : \C \to \pk$ of
\begin{equation} \label{eq:pKf}
    d\zeta_\l (z) =[\,\zeta_\l (z)
    ,\,\alpha(\zeta_\l(z))\,]\quad\mbox{ with }\quad
    \zeta_\l(0)=\xi_\l.
\end{equation}
If we set $4\beta_{-1} (z) :=\mi e^{\o(z)}$ then $\o$ is a solution of sinh-Gordon equation and $\a ( \zeta_\l(z))=\a(z)$ has the form of
formula (\ref{matrixalpha}) and we express the coefficient of $\alpha(z)$ in terms of $\omega$.
\end{proposition}

\begin{remark}
\label{laxomega}
The Lax equation preserves  $\beta_{-1}(z)  \in \mi \R^+$ and we can define a
function $\o: \C \to \R$ by setting $4 \beta_{-1}(z) :=\mi e^\o$. Now the equation $\beta_{-1,z}=2 \alpha_0 \beta_1$ implies that $2 \alpha_0 = \o_z$.
To express $\gamma_0$ in terms of $\alpha_0$ and $\beta_{-1}$, we consider the Lax equation and find
${\rm d}''\gamma_0=-2 \bar \alpha_0 \gamma_0$. Then $\gamma_0 = q e^{-\o}$ where $q$ is a holomorphic function. 
The term $q$ is constant. The reason is that along the parameter $z$, we have $a(\l)=-\l \det \zeta_\l (z)=-\l \det \xi_\l$ and $a(0)=\beta_{-1}\gamma_0=q/4$
with $4|q|=1$. The polynomial coefficient of $\zeta_\l(z)$ depends on higher derivative of the function $\o(z)$ pointwise in $z$.
\end{remark}
From the potential we obtain the extended frame and then the immersion by Sym-Bobenko formula,
\begin{theorem}(Theorem 1.3 \cite{HKS1})
\label{symbobenko}
Let $\xi_\l$ be a potential and $\zeta_\l : \C \to \pk$ the polynomial
Killing field \eqref{eq:pKf}.
For any  constant $ \gamma$ in $\SY^1$, the immersion
$$X_\l(z) = ( F_\lambda \sigma_3 F^{-1}_\lambda, {\rm Re}( -\mi\sqrt{ \gamma \lambda^{-1}}z))$$
with $\l \in \SY^1$, defines a one-parameter family of conformal
minimal immersion in $\SY^2 \times \R$ with  metric $ds^2  = \cosh ^2
\o |dz|^2$ if the extended frame $F_\l :\Omega \to \SL (\C)$ is a
solution of
\begin{equation}\label{eq:symb}
F _\lambda ^{-1}dF_\lambda=\alpha(\zeta_\l(z)):=\alpha(\o)
\qquad\mbox{with}\qquad F_\lambda(0)= \un\\
\end{equation}
where $\alpha (\o)$ has the form of (\ref{matrixalpha}). For $\lambda \in \SY^1$ we obtain the one parameter family of isometric associated family. In particular for $\l =\gamma=1$ we have $X_1=(G_1, y)$ a conformal immersion parameterized by its third coordinate. The function  $\o:\R^2 \to \R$ solves the sinh-Gordon equation.
\end{theorem}
\begin{remark}
\label{real}
{\bf Reality condition.}
 From the relation $\bar \a_{1/\bar \l}^t = -\a_\l$, the solution of (\ref{eq:symb}) satisfies $\bar F^t _{1/\bar \l} = F^{-1}_\l$.
 For  $\lambda \in \SY^1$,  $\alpha (\o)$ takes values in $\su$ and $F_\lambda$ takes
values in $\SU$. For general $\lambda \in \C^*$, we have $F_\l \in \SL (\C)$ since $\mathrm{tr} (\a(\o)) =0$. We call $\l$
the spectral parameter.
\end{remark}

\begin{remark}
\label{normbeta}
{\bf Normalisation.}
By  conformal parametrization we can choose, $4 |Q|=1$ (the annulus is transverse to horizontal sections).
Next we discuss the constants $\beta, \gamma \in \SY^1$ which are related to the Hopf differential. We will see that $4 Q=\beta \gamma \l^{-1}$.
We can normalize the parametrization by $\beta=\mi$ and a constant $|\gamma|=1$. For a given extended frame $F_\l$ which satisfies the equation (\ref{eq:symb}),
we consider $g=\bigl( \begin{smallmatrix} \beta &0 \cr 0 & \bar \beta \end{smallmatrix} \bigr) \in U(1)$. Then $\tilde F_\l=F_\l g$ induces the same immersion $G_\l=F_\l \sigma_3 F_\l ^{-1}$ and satisfies equation (\ref{eq:symb}) with $\beta=1$, $\tilde F _\lambda ^{-1}d \tilde F_\lambda = g^{-1} \alpha_\l (z) g$.
\end{remark}

\noindent{\bf The Sym point and conformal parametrization.} The determinant of $\zeta_\l (z)$  does not depend on $z$ and we remark that $k(\l)= \l \det \zeta_\l (z)  = \l \det \xi_\l$ and $k(0)=-\beta_{-1} \gamma_0=-a(0)$. The  term $\beta_{-1} \gamma_0$ of the corresponding polynomial Killing field does not depend on the
surface parameter $z$.

If $F_\l$ is related to a minimal surface $X_\l: \R^2 \to \SY^2 \times \R$ parameterized by its third coordinate then there is  $\lambda_0 \in \SY^1$, such that
$$X_{\l_0} (z) = (G_{\l_0} (z), y)=(F_{\l_0} (z) \sigma_3 F_{\l_0}^{-1} (z), {\rm Re}(-\mi \sqrt{\beta_{-1} \gamma_0{\l_0} ^{-1}}z))$$
and we observe that then the Hopf differential is $Q=-4\beta_{-1} \gamma_0 {\l_0}^{-1}(dz)^2=\frac14 (dz)^2$ i.e.  $\beta_{-1} \gamma_0= -\frac{\l_0}{16}.$
The value of ${\l_0}=e^{\mi\theta}$ associate to an immersion $X_\l$ is called the Sym point. In the following we will prefer a conformal parametrization which fixes the Sym point to $\lambda_0 =1$. To do that we make the conformal change $z \to e^{\mi(1-g)\theta/2} z$ and the M\"obius transformation $\l \to e^{\mi\theta}\l$. We explain in Appendix \ref{sympoint} how to apply this transformation.
\begin{definition}\label{pks}
A finite type minimal immersion $X:\R^2 \to \SY^2 \times \R$ is
conformally parameterized by its Sym point if there is a polynomial
Killing field
$$\zeta_\l :\R^2\to\left\{ \xi_\l \in \pk\mid
\l \det \xi_\l= -a(\l) \hbox{ and }
\beta_{-1} \gamma_0=a(0)= -\tfrac{1}{16}\,e^{\mi(1-g)\theta}:=
-\tfrac{1}{16}\,e^{\mi\Theta} \right\}
$$
which solves the Lax equation \eqref{eq:pKf}
where $a(\l)$ is a complex polynomial of degree $2g$ which satisfies
the reality conditions \eqref{eq:reality} (see remark \ref{real})
If $F_\l$ is the unitary factor associate to $\xi_\l$, then in this
parametrization the immersion is given by \eqref{eq:immersion}
\end{definition}
%
%
\section{Isospectral group and spectral data}\label{sec:spectral data}
In this Section we characterize those potentials $\xi_\l\in\pk$, which
correspond to periodic minimal immersions. This property turns out to
be a property of the spectral curve in
definition~\ref{spectral curve}. For a given
polynomial $a$ of degree $2g$ obeying \eqref{eq:reality} either all
elements of $I(a)$ in \eqref{eq:isospectral set}
have this property or no element. These sets $I(a)$ are called
isospectral sets and consists of matrices $\xi_\l$ having the same
spectral curves $\Sigma$ and same off-diagonal product
$a(0)= \beta_{-1} \gamma_0$. It is a $g$-dimensional complex manifold
(see below)
and its tangent space at an initial value is associate to the solution $u_1, u_2....$ of LSG depending on $\o$
and its higher derivative. Each solution of LSG constructed in proposition \ref{pinkallsterling} integrates to a long time solution $\o(t)$ of the sinh-Gordon equation. This deformation gives a deformation of annuli preserving the spectral curve. The integrable system consists of this hierarchy of commuting variational fields. These variational fields are exactly the isospectral deformations, and the commuting property is the key point to describe the geometry of the anulus at infinity, as well as the embeddedness property.
\begin{definition}
Suppose $X:\C \to \bbS^2 \times \R$ is a minimal immersion of finite type with potential $\xi_\l$. An isospectral deformation of $X$ is a smooth family of finite type immersions $X(t,z):[0,T]\times \C \to \SY^2 \times \R$ with $X(0,z)=X(z)$, and corresponding smooth family of potentials $\xi_\l :[0,T] \to I(a)$ with $a(\l)=-\l \det \xi_\l (0)
=-\l \det \xi_\l (t)$ for all $t\in [0,T]$.
\end{definition}

\noindent{\bf Iwasawa decomposition.} For real $r \in (0,1]$, let $ \SY_r=\{ \l \in \C; |\l|=r \}$ and define the loop group $\Lambda_r \SL(\C)={\calO}(\SY_r,\SL (\C))$ as the set of analytic maps $\SY_r \to \SL (\C)$.
We use the following loop group and subgroup:
Let the circle $ \SY_r=\{ \l \in \C; |\l|=r \}$, the disc $I_r=\{ \l \in \C; |\l| <r \}$ and the annulus $A_r = \{ \l \in \C; r < | \l | < 1/r \}$. We consider the set of analytic maps
\begin{align*}
    &\Lambda_r \SL(\C)= \{ F : \SY _r \to \SL  ( \C) \mid F \hbox{ analytic} \}\,,\\
    &\Lambda _r\SU(\C) = \{ F : A _r \longrightarrow \SL ( \C) \mid F \hbox{ analytic on }A_r \hbox{ and } F |{_{\SY^1}}\in \SU \}\,,\\
    &\Lambda _r^+ \SL (\C) =\{  B \in \Lambda_r \SL (\C) \cap  \calO (I_r,\SL (\C)) \mid
    B(0)= \bigl( \begin{smallmatrix} \rho & c  \cr 0 & 1/\rho \end{smallmatrix} \bigr)
    \hbox{ for } \rho \in \R^+ \hbox{ and } c \in \C \}\,.
\end{align*}

Then multiplication is a real analytic diffeomorphism
$$\Lambda_r \SL(\C) \to  \Lambda _r\SU(\C) \times \Lambda _r^+ \SL (\C)$$
and any $\phi _\l \in \Lambda_r \SL(\C)$ can be uniquely factorized into $\phi _\l=F_\l B_\l$ with
$F_\l \in \Lambda _r\SU(\C)$ and $B_\l \in \Lambda _r^+ \SL (\C)$. This is the $r$-Iwasawa decomposition (or factorization) of $\phi _\l$. When $r=1$, we omit the subscript and it is referred to as the Iwasawa decomposition. Given a map $\phi_\l : \Omega \to  \Lambda_r \SL(\C)$, we can
apply the decomposition $\phi_\l=F_\l B_\l$  pointwise on the domain $\Omega$ , and then $F_\l:\Omega \to \Lambda _r\SU(\C)$ and $B_\l:\Omega \to \Lambda _r^+ \SL (\C)$.

\begin{theorem} \cite{McI:tor, PreS}
\label{riwasawa}
The multiplication $\Lambda_r \SU (\C) \times \Lambda _r^+\SL (\C)$ to $\Lambda_r \SL (\C)$ is a real analytic bijective diffeomorphism. The unique splitting of an element $\phi _\l \in \Lambda_r \SL (\C)$ into
$$\phi _\l= F_\l B_\l $$
with $F_\l \in \Lambda_r \SU (\C)$ and $B_\l \in  \Lambda _r^+ \SL (\C)$ is the r-Iwasawa decomposition.
\end{theorem}
\noindent{\bf Isospectral Group action.}
We define a smooth action of the Lie-group $\mathbb{C}^g$ on
$I(a)$. This action corresponds to a deformation by integrating the hierarchy of solutions of LSG of Proposition \ref{induction}. 
\begin{definition}
\label{groupaction}
For $t=(t_0,\ldots,t_{g-1}) \in \C^g$ we set the group action $\pi(t) :I(a) \to I(a)$ defined by
\begin{equation}\label{eq:isospectral action}
\begin{aligned}
&
\xi_\l\mapsto\pi(t)\xi_\l=B_\l(t)\xi_\l B_\l^{-1}(t)=F_\lambda^{-1}(t)\xi _\l F_\lambda(t)\\
\hbox{ with }&
\exp\left(\xi_\l \sum_{i=0}^{g-1}\lambda^{-i}t_i\right)=F_\lambda(t)B_\l(t).
\end{aligned}\end{equation}
In the last equation the right hand side is the Iwasawa decomposition
of the left hand side.
\end{definition}
\begin{remark}
i) The first solution of LSG in the hierarchy is $u_0=\o_z$, which acts transitively on the annulus.
The determinant $\det\xi_\l$ is invariant under this action and the action of the translations by $z\in\mathbb{C}$ coincides with the
action of $t_z=(z,0,\ldots,0)\in\mathbb{C}^g$.  The associate Jacobi field is obtained by changing coordinates and by vertical translations. 

The decomposition $\exp (z \xi_\l)=F_\l (z)B_\l(z)$
gives the polynomial Killing field $\zeta_\l (z) = F^{-1}_\l (z) \xi_\l F_\l (z)$.

ii) The second solution of LSG in the hierarchy is $u_1=\o_{zzz}-2\o_{z}^3$ which is associated to the Shiffmann Jacobi field on the surface.
\end{remark}
For $t,t'\in\mathbb{C}^g$ the corresponding
Iwasawa decompositions obey
$$F_\lambda(t+t')B_\l(t+t')=F_\lambda(t)B_\l(t)F_\lambda(t')B_\l(t')
=F_\lambda(t')B_\l(t')F_\lambda(t)B_\l (t).$$
Therefore the Iwasawa decomposition for the action of
$\pi(t')$ on $\pi(t)\xi$ is equal to
$$F'_\lambda(t')B_\l'(t') =\exp\left(  B_\l (t) \xi_\l  B_\l ^{-1}(t)  \sum_{i=0}^{g-1}\lambda^{-i}t_i\right)$$
and
$$
F'_\lambda(t')B_\l'(t')=B_\l(t)F_\lambda(t')B_\l(t')B_\l^{-1}(t)=
F_\lambda^{-1}(t)F_\lambda(t+t')B_\l(t+t')B_\l^{-1}(t).$$
Since the Iwasawa decomposition is unique we conclude
\begin{proposition}
\label{com}
The group action $\pi(t) :I(a) \to I(a)$  is a commuting action and
\begin{align*}
\pi(t')\pi(t)\xi_\l=
B'(t')B(t)\xi_\l B^{-1}(t)B'^{-1}(t')=
B(t+t')\xi_\l B^{-1}(t+t')&=\pi(t+t')\xi  _\l.
\end{align*}
$$F_\lambda'(t')=F_\lambda^{-1}(t)F_\lambda(t+t') \hbox{ and } B_\l'(t')=B_\l(t+t')B_\l^{-1}(t)$$
\end{proposition}
This action has an obvious extension to an action of all sequences $(\ldots,t_{-1},t_0,t_1,\ldots)$ with only
finitely many non vanishing entries. Whenever the exponential on the left hand side
of \eqref{eq:isospectral action} belongs to one factor of the Iwasawa decompostion, then the corresponding $t$ acts trivially.
For example, for any $k \geq 1$, the matrix $\exp (\l^k \xi_\l)$ belongs to the second factor of the Lie Algebra, while if $k \leq -g$,
the matrix $\exp (\l^k \xi_\l)$ is a product of two commuting factors of the Iwasawa decomposition. These actions acts trivially.

Furthermore, since all $\xi_\l\in I(a)$ obey
$$^{t}\overline{ \xi _{1/ \bar{\lambda}}}=-\lambda^{1-g}\xi _\l$$
the matrices
$$\exp \left(\xi_\l\left(\sum_{i=0}^{g-1}\lambda^{-i}t_i+\sum_{i=0}^{g-1}\lambda^{i+1-g}\bar{t}_i\right) \right) \in \LSU (\C)$$
belong to the Lie Algebra of the first factor in the Iwasawa
decomposition.  Then the related action $\pi (t_i+\bar{t}_{g-1+i})\xi_\l=\xi_\l$ is trivial. In summary we conclude with the following
\begin{remark}
Only a finite dimensional Lie group with Lie algebra isomorphic to $\mathbb{R}^{g}$ acts non-trivally on $I(a)$.
\end{remark}
This group action yields a differentiable structure on the isospectral set, and we cite 
\begin{proposition}(Theorem 4.8 \cite{HKS1})
\label{1}
Let $I(a)$ be the isospectral set with polynomial $a(\l)$, which satisfies the reality condition of definition \ref{pks}
\begin{enumerate}
\item $I(a)$ is compact.
\item If the $2g$ roots of $a(\l)$ are pairwise distinct (without double roots), then $I(a)$  is connected
 smooth g-dimensional manifold.  This manifold is diffeomorphic to a $g$-dimensional real torus; $I(a) \cong \left(\Sp^1\right)^g$.
\end{enumerate}
\end{proposition}
Proposition \ref{1} implies the following corollary:
\begin{corollary}\label{quasi periodic}
\label{quasiperiodic}
We consider a finite type annulus $X: \C \to \SY^2 \times \R$ with period $\tau \in \C$ associate to a spectral curve $\Sigma$ and potential $\xi_\l$.
%
\begin{enumerate}
\item If $a(\l)$ has only simple roots, then the solution $\o$ is the restriction of a function on a g-dimensional torus to a two dimensional subgroup, and so $\o$ and $\zeta_\l (z)$ are quasi-periodic.
\item Every solution $u_n$ of LSG induced by the Pinkall-Sterling iteration integrates into a long time solution of the sinh-Gordon equation.
\item If $a(\l)$ has only simple roots, then every annulus has a quasi-periodic metric. This means that if we consider a diverging sequence $(t_n)$, the sequence of annuli $\o(x,y+t_n)= \o_n(x,y)$ has a subsequence converging to $\tilde \o (x,y)$ defined by a potential in $I(a)$.
\end{enumerate}
\end{corollary}
\noindent{\bf Spectral data of minimal annuli.}
We study the monodromy $M_\l(\tau)=F_\l(z)^{-1}F_\l(z+\tau)$ of the extended frame $F_\l$ for a period $\tau$.
By construction the monodromy takes values in $\SU$ for $|\lambda |=1$. The monodromy depends on the choice of base point $z$, but its conjugacy
class and hence eigenvalues $\mu,\,\mu^{-1}$ do not. The eigenspace of $M_\l (\tau)$ depends holomorphically on $(\mu, \l)$.

Let $\zeta_\l$ be a solution of Lax equation (\ref{lax})
with initial value $\xi_\l\in\pk$, with period $\tau$ so that $\zeta_\l(z+\tau) = \zeta_\l(z)$ for all $z \in \R^2$. Then for $z=0$ we have

$$\xi_\l= \zeta_\l(0) = \zeta_\l(\tau) =
F_\lambda^{-1}(\tau) \,\xi\,F_\lambda(\tau) = M_\lambda^{-1} (\tau) \xi \,M_\lambda (\tau)$$
and thus
$$
    [\,M_\lambda (\tau) ,\,\xi_\l\,] = 0\,.
$$
Hence the eigenvalues of $\xi_\l$ and $M_\lambda (\tau)$
are different functions on the same Riemann surface $\Sigma$. Furthermore the eigenspaces of $M_\lambda (\tau) $ and $\xi_\l$ coincide point-wise. At $\l=0$ and $\l=\infty$, the monodromy $M_\lambda (\tau)=F_\lambda (\tau)$ has essential singularities. The period $\tau$ is related to a trivial action on the isospectral set $I(a)$, and we prove that the essential singularities of $\mu$ at $\l=0$ and $\l=\infty$ depend only on an isospectral orbit.
\begin{proposition}(Proposition 5.4 \cite{HKS1})
\label{stab}
The group $\Gamma_{\xi_\l}=\{ t \in \R^g; \pi (t) \xi_\l = \xi_\l \}$ depends only on the orbit of $\xi_\l$.
If $\gamma \in \Gamma_{\xi_\l}$ satisfies $F(\gamma)=\pm \un$ at $\l=\l_0$ for some $\xi_{\l,0} \in {\calP}_g$ then the same is
true for all $\xi_\l$ in the orbit of $\xi_{\l,0}$. The period $\tau$ is related to $t=(\tau,0,...,0) \in \Gamma_{\xi_\l}$.
\end{proposition}
\begin{proof}
This is the direct consequence of the commuting property of the group action $\pi (t)$. We have only to remark, that
$\pi(\gamma) \xi_\l = \xi_\l$ implies that $\pi(\gamma) \pi(t) \xi_\l =\pi(t)\pi(\gamma)  \xi_\l=\pi(t)\xi_\l $.
\end{proof}
\begin{proposition}(Proposition 5.5 \cite{HKS1})
\label{mu}
Let $\xi_\l \in {\calP}_g$ without roots. Then $\gamma \in \Gamma_{\xi_\l}=\{ t \in \R^g \mid \pi (t) \xi_\l = \xi_\l \}$ if and only if
there exists on $\Sigma=\{ (\nu,\l) \in \C^2 \mid \nu^2=a(\l)/\l\}$ a function $\mu$ which satisfies:
\begin{enumerate}
\item $\mu$ is holomorphic on $\Sigma - \{ 0, \infty \}$ and there exist holomorphic functions $f, g$ defined on $\C^*$ with $\mu =f  \nu +g$.
\item $\sigma ^* \mu = \mu ^{-1}, \rho ^* \mu = \bar \mu ^{-1} , \eta ^* \mu =\bar \mu ^{}$ and $\mu=\pm 1$ at each branch
point $\alpha_i$ of $\Sigma$.
\item $d \ln \mu$ is a meromorphic 1-form with $d \ln \mu -d(\sum_{i=0}^{g-1} \gamma_i \l^{-i}\nu)$ holomorphic in a neighborhood
of $\l=0$ and $d \ln \mu + d(\sum_{i=0}^{g-1} \bar \gamma_i \l^{i+1-g}\nu)$ is holomorphic at $\l=\infty$.
\item For $\gamma=(\tau, 0,...0)$, the differential $d\ln\mu$ has second order poles without residues at the two
points $\lambda=0$ and $\lambda=\infty$, and $d \ln \mu -\frac{\mi e^{\mi\Theta/2} \tau }{4}d\sqrt {\l}^{-1}$ extends holomorphically to $\l=0$, while $d \ln \mu -  \frac{\mi e^{-\mi\Theta/2} \bar \tau  }{4}d \sqrt{\l}$ extends holomorphically to $\l = \infty$.
\end{enumerate}
\end{proposition}
The existence of an annulus depends on the existence of the function $\mu$ having the correct behavior at $\l=0$ and $\l=\infty$.
In the case where $a$ has only simple roots we have only to prove that $\mu$ is holomorphic on $\Sigma - \{ 0, \infty \}$, which is a
weaker condition than $\mu =f  \nu +g$ with holomorphic function $f, g$ defined on $\C^*$ in the general case (when $\nu$ has higher order roots, $f$ could have poles). Keeping this in mind we define spectral data of a minimal cylinder by
\begin{definition}
\label{spectraldata}
The spectral data of a minimal cylinder of finite type in
$\SY^2 \times \R$ consists of  $(a,\,b)$ where
$a$ is a complex polynomial of degree $2g$, and  $b$ is a complex polynomial of degree $g+1$
such that
\begin{enumerate}
\item[(i)] $\l ^{2g}  \bar a (\bar \l ^{-1})=a(\l) $ and $\l^{-g} a (\l) \leq 0$ for all $\l \in S^1$
\item[(ii)]$\l ^{g+1} \bar b (\bar \l ^{-1}) =- b (\l)$
\item[(iii)] $b(0) =-\frac{ \tau e^{\mi \Theta}}{32} \in e^{\mi \Theta/2} \R $
\item[(iv)] $   {\rm Re}\left( \int_{\alpha_i}^{1/\bar \a_i}\frac{b d\l}{\nu \l^2} \right) =0$ for all roots $\a_i$ of $a$ where the integral is computed
on the straight segment $[\a_i, 1/ \bar \alpha_i]$.
 \item[(v)] The unique function $h:\tilde  \Sigma \to \C$ where $\tilde \Sigma= \Sigma - \cup \gamma_i$ and $\gamma_i$ are closed cycles over the straight lines connecting $\alpha_i$ and $1/ \bar \alpha_i$, such that
 $$\sigma ^* h(\l)= -h(\l) \hbox{ and } dh=  \frac{b d\lambda}{\nu \lambda^2}$$
takes values on $i\pi \Z$ at all roots of $(\l -1)a$.
 \item[(vi)] In the case when $a$ has higher order roots there are holomorphic functions $f,\, g$ defined on $\C^*$ with $e^h=f  \nu +g$.
\end{enumerate}
\end{definition}
\begin{corollary}\label{periodic immersions}
If $\xi_\l\in\pk$ corresponds to a periodic minimal immersion
$X:\C/\tau\Z\to\SY^2\times\R$, then there exists a polynomial $b$,
which obeys (i)-(vi) with $a(\l)=-\l\det(\xi_\l)$.

If $(a,b)$ obeys (i)-(vi), then all $\xi_\l\in I(a)$
correspond to periodic minimal immersions.
\end{corollary}
\section{Isospectral group and embedded annuli}
\label{sembedded}
We prove that embeddedness is an isospectral property. First recall that the maximum principle
at infinity implies the following

\begin{proposition}
\label{simpleembedded} Let $I(a)$ the isospectralset of a polynomial $a(\l)$.
Suppose $a$ has only simple roots, and $\tau \in \C$ such that the induced immersion $X: \C \to \SY^2 \times \R$ of $\xi _{\l} \in I(a)$
satisfies $X(z+\tau)=X(z)$. Then if $X (\C / \tau \Z)$ is embedded then
any $ \hat \xi_\l \in I(a)$ induces an embedded annulus $\hat X (\C / \tau \Z)$.
\end{proposition}

\begin{proof}
By the maximum principle we know that embeddedness is a closed condition.
If there is a sequence of annuli $A_n$ converging on compact sets to a limit $A_0$,
then if $A_n$ is embedded, the limit is embedded. If not there is a point of $A_0$ where
we locally have at least two disks in $A_0$ which are intersecting with
transversal tangent plane. Then for $n$ large enough we can observe two transversely
intersecting disks converging to $A_0$ and the property to be embedded is
a closed property in $I(a)$.

We prove that embeddedness is an open property in $I(a)$. Consider a potential $\xi_{\l}$ inducing an embedded annulus $X(z)$, and a potential 
$\hat \xi_\l$ close to $\xi_{\l}$ inducing an annulus $\hat X(z)$ with the same period $\tau$. We prove that $\hat X(z)$ is embedded.

Since the Iwasawa decomposition is a diffeomorphism, the map $X: \C / \tau \Z \times I(a) \to \SY^2 \times \R$ where  $X(z,\xi_{\l})$ is the immersion induced by $\xi_{\l}$ at a point $z$, is uniformly smooth on compact sets of $\C / \tau \Z$.
Then for all $\epsilon >0$ there is $R > 2\tau$ such that on $\bar B( 0,R) \times I(a)$ there exists $\delta >0$ with
$$\xi_{\l}, \hat \xi_\l \in I(a) , ||\xi_{\l}- \hat \xi_\l || <\delta \quad \hbox{ implies } \quad ||X(z,\xi_{\l}) - X(z,\hat \xi_\l)||_{{\calC}^2} <\epsilon \,\,\hbox{ for $z \in \bar B( 0,R)$}.$$

Since $I(a)$ has a structure of a $g$-dimensional manifold and the isospectral group action defines a chart around $\xi_{\l}$, there is a $t\in \C^g$ such that $\hat \xi_{\l} = \pi (t)\xi_{\l}$ and $st \in \C^g$, define a smooth deformation $ \pi (s t)\xi_{\l} , s \in [0,1]$ between $\xi_{\l}$ and  $\pi (t)\xi_{\l}$.

Locally in a neighborhood of some $\xi_{0} \in I(a)$ and some $\delta >0$, there exists $\delta_1>0$ and $\delta_2 >0$ such that
$$||\pi(t)\xi_\l - \xi_\l|| \leq \delta \hbox{ for } || \xi_\l - \xi_{0}||< \delta_2 \hbox{ and } |t| < \delta_1.$$
Using compactness of $I(a)$  there is a $\delta' >0$ such that for all $\xi_\l \in I(a)$ and $| t| <\delta' $, we have
$$||\pi(t)\xi_\l - \xi_\l|| \leq \delta.$$
Now consider annuli $X(\C / \tau \Z)$ and $\hat X(\C / \tau \Z)$ induced by potentials $\xi_{\l}$ and $\hat \xi_\l=\pi (t) \xi_{\l}$, and with polynomial Killing fields $\zeta_\l (z) = \pi(z) \xi_{\l}$ and $\hat \zeta _\l(z) = \pi(z)\hat  \xi_\l=\pi(z)\pi (t) \xi_{\l}$.
Since the action $\pi(t)\pi (z)\xi_\l= \pi (z) \pi (t) \xi_\l$ by commuting property, we have
$$||\hat \zeta_\l (z) -\zeta _{\l}(z)||=||\pi(z)\pi (t) \xi_{\l}- \pi(z) \xi_{\l}||=||\pi (t) [\pi(z) \xi_{\l}]-\pi(z) \xi_{\l}|| \leq \delta$$
This proves that the polynomial Killing fields $\zeta_{\l}(z)$ and $\hat \zeta_\l (z)$ are uniformly close in $z \in \C/\tau \Z$, and it remains to consider the corresponding immersions.

Let $z_0 \in \C/\tau \Z$ and $\exp (z\zeta _{\l} (z_0))=F^{z_0}_\l(z) B^{z_0}_\l (z)$ and 
$\exp (z\hat\zeta _{\l} (z_0))=\hat F^{z_0}_\l(z) \hat B^{z_0}_\l (z)$ the Iwasawa decompositions. Then $F^{z_0}_\l(z)=F_\l ^{-1}(z_0) F_\l (z+z_0)$ and $\hat F^{z_0}_\l(z)= \hat F_\l ^{-1}(z_0) \hat F_\l (z+z_0)$ by the commuting formula of Lemma \ref{com}.
Since $||\hat \zeta_\l (z_0) -\zeta _\l(z_0)|| \leq \delta$, we have 
$$||F^{z_0}_\l(z) - \hat F^{z_0}_\l(z)|| < \epsilon \hbox{ for } z \in \bar B(0,R).$$
This implies that
$$|| I(z_0) \hat  F_\l (z+z_0)- F_\l (z+z_0)|| <\epsilon  \hbox{ for } z \in \bar B(0,R)$$
where $I(z_0)=F_\l (z_0) \hat F_\l^{-1} (z_0)$ is an isometry of $\SY^2 \times \R$ depending on $z_0$ which preserves
the level set $\SY^2 \times \{t\}$.
This proves that for any $\epsilon >0$ and any point $z_0 \in \C / \tau \Z $, there is a compact set $K(z_0) \subset \C / \tau \Z$  such that there exists an isometry
$I(z_0)$ of $\SY^2 \times \R$, with $I(z_0)\hat X(z) $  uniformly  ${\calC}^2$ close to $X(z) $ on $K(z_0)$
$$\sup _{z \in K(z_0)}| |I(z_0) \hat X (z) - X (z) ||_{{\calC}^2} \leq \epsilon$$

A properly embedded annulus intersects any level set $\{x_3 =t\}$ in exactly one connected component which is a topological circle.
By the maximum principle at infinity for minimal surfaces in manifolds with non-negative Ricci curvature, each minimal annulus has an embedded tubular neighborhood $T_{\epsilon_1}=Y((\C / \tau \Z) \times ]-\epsilon_1,\epsilon_1[  )$. The constant $\epsilon_1>0$ depends only
on a lower bound of the flux $F_3= |\tau| \geq \epsilon_0 >0$. Since $|\o|$ is uniformly bounded in ${\mathcal C}^{k,\alpha}$
norm, the third coordinate $n_3= \tanh \o$ is bounded away from $n_3=1$. This implies that the tubular neighborhood $T_{\epsilon_1}$ intersects a
horizontal section $\{x_3 =t\}$ in an embedded tubular neighborhood of the curve $X(z) \cap \{x_3=t\}$. Hence the isometry $I(z_0)$ gives a comparison of the curves $I(z_0)\hat X(z) \cap \{x_3=t\}$ with $X(z) \cap \{x_3=t\}$ and proves that every level curve of $\hat X(z)$ are embedded. Hence the annulus $\hat X(z)$ is embedded. Hence embeddedness is an open property of $I(a)$.
\end{proof}
In the general case the action has several orbits. If an annulus is embedded in one
orbit, all elements of the orbit induce embedded annuli. Higher orbits are bubbletons. The following proposition is used to reduce bubbletons.
\begin{proposition}
\label{bubbleembedded}
Assume that there is an isospectral set $I(a)$ with $a(\l)$ having higher order roots $\alpha_0,...,\a_n$ with $|\a_i| \neq 1$ for $i=1,...,n$ and $\tau \in \C$ such that
the induced immersion $X: \C \to \SY^2 \times \R$ of $\xi _{\l} \in I(a)$, with $\xi_{\a_i} \neq 0$ for $i=1,...,n$
satisfies $X(z+\tau)=X(z)$. Then if $X (\C / \tau \Z)$ is embedded then any $\hat  \xi_\l = \pi (t) \xi_{\l} \in I(a)$  induces an embedded annulus $\hat X (\C / \tau \Z)$.
\end{proposition}

\begin{proof}
In  \cite{HKS1}, section 6 we prove that any potential $\xi_\l$ splits into an element of $\C \P^1 \times I({\tilde a})$ where $\tilde a$ is a polynomial with lower degree.
There is a group action $\pi : \C^g \times I(a) \to I(a)$, but the stabilizer $\Gamma_{\xi_\l}$ is no longer isomorphic to  $\Z^g$.
We prove in proposition \ref{bubbleaction}, that there is a subgroup action
$$\tilde \pi : \C \times I(a) \to I(a)$$
which acts transitively on the $\C \P^1$ factor, fixing the second factor of $I({\tilde a})$. We have $\tilde \pi (\beta) \xi_\l=(L(\beta), \tilde \xi_\l)$
where $\beta \in \C \P^1$. This implies that the group action $\pi$ gives different orbits in $I(a)$. Each orbit is characterized by the property that all its element have the same roots of the same orders at $\a_1,...,\a_n$, since the action $\pi(t)\xi_\l$ preserves the order of roots of $\xi_\l$.

Suppose $\xi_{0,\l}$ is a potential of an embedded annulus and assume it has no zeroes.  Consider its orbit
$O=\{ \xi_\l \in I(a) ;  \xi_\l = \pi (t) \xi_{0,\l}\}.$ We prove that if $\hat \xi_\l \in O$, then it induces an embedded annulus. In fact $\xi_\l$ is a higher order bubbleton. If the annulus of $\xi_\l$ closes with period $\tau$, we prove in Proposition \ref{bubbleclosing} that all $\hat \xi_\l \in O$ give $\tau$-periodic annuli.

Now since $(t, \xi_\l) \to \pi(t)\xi_\l$ is continuous, we have for all $\delta >0$ and $\xi_0 \in I(a)$, that there exists $\delta_1 >0$ and $\delta _2 >0$ such that $|| \pi (t) \xi_\l - \xi_\l ||  \leq \delta \hbox{ for } ||\xi_\l - \xi_0|| \leq \delta_2, |t|<\delta_1.$

If $\xi_\l$ has a root, then a subgroup of $t \in \C^g$ acts trivially on $\xi_\l$. But the important point is that the action extends continuously to this subset. Since $I(a)$ is compact, the covering by the set $B(\xi_0, \delta_2)$ has a finite subcovering. Then there is $\delta'>0$
with
$$|| \pi (t) \xi_\l - \xi_\l ||  \leq \delta \hbox{ for } \xi_\l \in I(a) , |t|<\delta'.$$
Note that if $\hat \xi_\l= \pi (t) \xi_{\l,0}$ induces a different annulus, then $|| \pi(z)\pi(t) \xi_{\l,0}-\pi(z)  \xi_{\l,0}|| \leq \delta$ for all $|t| \leq \delta'$
and for all $z \in \C$. Translations in the annulus preserve the orbit because $\pi (z) \xi_\l$ has a root if and only if $\xi_\l$ has a root.
Therefore we see that $\hat \zeta_{\l}(z)$ and $\zeta_{0,\l}(z)$ stay uniformly close and we conclude as in the preceding lemma that the orbit is a subset of $I(a)$ inducing embedded annuli.
\end{proof}
\begin{remark} We remark that in the case where $\xi_\l$ would have
a root at $\alpha_0$ then the group action $\pi (t)\xi
_\l$ would have a root too. Then at this particular point the group action does not generate a real $g$-dimensional manifold. But one can prove that $\pi$ generates a $g-2n$ dimensional subspace of the tangent space of
${\calP}_g$. The subgroup action $\tilde \pi (\beta)$ with $\beta \in \C^n$ acts trivially on the set of potentials having $2n$ roots counted with multiplicity. This is equivalent to
the situation where one can remove the singularity without changing the immersion (see proposition \ref{remove}).
\end{remark}
\section{Deformation of finite type annuli}
\label{abdeformation}
In this section we describe the deformation of minimal cylinders of finite type. We deform
the spectral curves $\Sigma$ and we preserve the closing condition.  A minimal cylinder of $\SY^2 \times \R$ possesses a monodromy $F_\l (\tau)$ whose eigenvalue
$\mu (\tau)$ is a holomorphic function on $\Sigma$ with essential singularities at $\lambda =0$ and $\lambda = \infty$.
In the following we deform  $\Sigma=\{ \nu ^2 = \lambda ^{-1} a(\lambda)\}$ by moving the roots of the polynomial $a$ without destroying the global properties of the holomorphic function $\mu$. The existence of $\mu$ on $\Sigma$ with
$$d \ln \mu = \frac{b\, d \l}{\nu \l^2}$$
which satisfy the closing condition of definition \ref{spectraldata}, assure that the spectral data $(a,b)$ parameterize a closed annulus.
Then we  derive vector fields  on open sets of spectral data
$${\calM}:=\left\{ (a,\,b)\right\}$$
and show that their integral curves are differentiable families of spectral data of periodic
minimal cylinders.

We parameterize such deformations by one parameter $t \in [0,\e)$. Along the deformation we will increase
the length of the period $|\tau|$. This will be a condition on the vector fields on the set $(a(t),b(t))$.
\vskip 0.5cm
\noindent
{\bf The Whitham deformation. } We consider the function $\ln \mu$ as a function depending on $\l$ and $t$. On $\Sigma$, the multivalued holomorphic function $\ln \mu$ has simple poles
at $\l=0$ and $\l =\infty$. Let us define  $U_1, U_2 .....U_{2g}$ a covering by open subsets of $\Sigma$,
such that each $U_i$ contains at most one branch point $\a_i$ and $U_{2g+1}$ is an open neighborhood
at $(\infty,0)$,  $U_{2g+2}$ an open neighborhood  of $(\infty, \infty)$.  We can write locally
the meromorphic function on $\Sigma$ by

$$
    \ln \mu = \left\{ \begin{array}{ll}
            f_i (\l) \sqrt{\l-\a_i} + \pi\,\mi\,  n_i  &\mbox{ on } U_i, 1\leq i \leq 2g \\
            \l^{-1/2}\, f_{2g+1}(\l) + \pi\,\mi\,  n_{2g+1}  &\mbox{ on } U_{2g+1}\\
            \l^{1/2}\, f_{2g+2}(\l) + \pi\,\mi\,  n_{2g+2} &\mbox{ on } U_{2g+2}\,.
            \end{array} \right.
$$
We can write locally on the open set $U_i$,
$$\partial_t \ln \mu = \partial_t f_i ( \lambda) \sqrt{\l -\a_i} - \frac{ \dot \a_i f_i(\l) }{2 \sqrt{\l - \a_i}}\,.$$
We remark that at each branch point $\partial_t \ln \mu$ has a first order pole on $\Sigma$.
Since the branches of $\ln \mu $ differ from each other by an integer multiple
of $2 \pi i$, then  $\partial_t\ln \mu$ is single valued on $\Sigma$ and
can have poles only at the branch points of $\Sigma$, or equivalently at the zeroes of $a$
and at $\l=0$ or $\l=\infty$. Collecting all these conditions we can write $\partial_t \ln \mu $
globally on $\Sigma$ by
\begin{equation}\label{eq:def_c}
  \partial_t\ln \mu =  \frac{c(\l)}{\nu \l }
\end{equation}
with a real polynomial $c$ of degree at most $g+1$, which satisfies the reality condition
\begin{equation}
\label{eq:realc}
\l ^{g+1} \overline{c (\bar \l ^{-1})}=c(\l).
\end{equation}
The abelian differential $d \ln \mu$ of the second kind is of the form
 \begin{equation}
 \label{eq:dlogmu}
 d \ln \mu = \frac{b \,d \lambda}{\nu \lambda ^2}
 \end{equation}
 where $b$ is a real polynomial of degree $g+1$  which satisfies the reality condition
 $$\lambda ^{g+1} \overline{ b (\bar \l ^{-1})}=-b(\l).$$
We differentiate \eqref{eq:dlogmu} with respect to $t$, and
\eqref{eq:def_c} with respect to $\l$ and derive
\begin{equation*}
 \partial^2_{t\l}\ln \mu =  \partial_{\l} \frac{ c}{\nu \l}=\frac{c'}{\nu \l}-\frac{c}{\nu \l ^2}- \frac{c \nu '}{\nu ^2 \l}=
  \frac{2\nu ^2 \l ^2 c' -2\l \nu ^2 c -c a'\l+ca}{2 \nu ^3 \l ^3}
  \end{equation*}
 \begin{equation*}
 \partial^2_{\l t }\ln \mu =  \partial_{t} \frac{ b}{\nu \l^2}=\frac{\dot b}{\nu \l^2}-\frac{b \dot \nu }{\nu ^2 \l ^2}=
  \frac{2\nu ^2 \l \dot b -b \dot a }{2 \nu ^3 \l ^3}
  \end{equation*}

\noindent
Hence second partial derivatives commute if and only if
\begin{equation} \label{eq:integrability_1}
-  2\dot{b}a + b\dot{a} = - 2 \l ac' + ac + \l a'c.
\end{equation}
Both sides in the last formula are polynomials of at most degree $3g+1$
which satisfy a reality condition. This corresponds to $3g+2$ real equations.
Choosing a polynomial $c$ which satisfies the reality condition (\ref{eq:realc})
we thus obtain a vector field on polynomials $\{(a,b)\}$.
When $a$ and $b$ have no common roots $\a_i$ and $\beta_i$ then equation
\eqref{eq:integrability_1} uniquely determines the values of
$\dot{a}$ at the roots of $a$  and the values of $\dot{b}$ at the
roots of $b$.  In the case, where $a$ and $b$ have only simple roots, we have the Whitham deformation system
\begin{equation}
\label{eq:integrability_2}
\dot a ( \alpha _i) = \frac{\a_i a' (\a_i) c(\a_i)}{b(\a_i)}\,; \qquad
\dot b( \beta_i)=\frac{2\beta_i a(\beta_i)c'(\beta_i) -a(\beta_i)c(\beta_i) - \beta_i a' (\beta_i) c(\beta_i)}{2a(\beta_i)}\,.
\end{equation}

\vskip 0.3cm
\noindent
{\bf Singularity of the Whitham deformation.} By defining such polynomials $c$ we obtain vector fields on the space of polynomials $a$ of degree $2g$
and polynomials $b$ of degree $g+1$ satisfying the reality condition.  The value of $c$ at each $\a_i$ and $\beta_i$ determine uniquely a tangent vector $(\dot a,\dot b)$ on the moduli space of minimal cylinders.  Equations \eqref{eq:integrability_2} define rational vector fields on the
space of spectral data $(a,b)$ of minimal annuli. We will study in the next subsection the condition on $c$ to preserve the period along the deformation when $a$ and $b$ have no common roots.

The integral curves have possible singularities when $a$ and $b$ have common roots but we will see below that we can pass through such singularities. The
closing conditions on $c$ are the same, but the difference now is that the vector field is meromorphic. In this setting we reparameterize the integral curves $(a(t),b(t))$ in the neighborhood of each pole by using a smooth reparametrization of the flow using the value at the branch point of the function $\ln \mu$. We construct a smooth model for the space of spectral curves preserving $\mu$ in the spirit of deformations of branched coverings of Hurwitz in section \ref{param}. In this model we reparameterize the equation (\ref{eq:integrability_2}), and obtain instead of a pole a zero in the induced vector field. At the roots of the vector field in the equation (\ref{eq:integrability_2}) we show that the linearization has  non-zero positive and negative eigenvalues and thus integral curve moving in and out of the singularity. This give us a way to pass through
common roots of $a$ and $b$ in the meromorphic equation (\ref{eq:integrability_2}) (see theorem \ref{commonroots}).

\vskip 0.3cm
\noindent
{\bf Opening double points.} Another application of this smooth model is a construction of integral curves which open double points on $\Sigma$ and increase the genus of the spectral curve. To open a double point we add a double zero to the polynomial $a(\l)$ and a simple zero to the polynomial $b(\l)$, and note that this does not change $\mu$. Thus we create a common root of $a$ and $b$. Passing through the singularity means that we open this double point. We can open double points only at points where $\mu = \pm 1$. When two roots coalesce, this is a double point of $\mu$.

There is a difference between opening double points on $\SY^1$ and opening pairs of double points away from $\SY^1$. In the first case, there exists only one direction which preserves the condition $\l^{-g} a(\l) \leq 0$ for $\l \in \SY^1$. The other direction should move the roots of $a (\l)$
on the unit circle, a deformation which is not allowed. Pairs of double points can be opened in both directions. We can pass through such singularities 
and this defines a deformation which increases the genus of the spectral curve.

To increase the flux simultaneously while opening a non-real double point we choose a $c$ which changes the flux. At the double point, we find an integral curve of the equation and we pass through the singularity.
Since for non-real double points we can open double points in both direction of the integral curve,
there exists a direction which increases the flux and opens the double point.

In the case where we open a real double point, there is only one direction which preserves the condition $\l^{-g} a(\l) \leq 0$ and leads away from the unit circle, preserving the reality condition. In general we cannot increase the flux by opening real double points. 
This is the case when we open double points on the genus zero curve (see section \ref{opendoublegenuszero}).

\vskip 0.3cm
\noindent
{\bf Increasing the flux along the deformation. } At $\l=0$, we have $-2\dot b (0)  a(0) +b(0) \dot{a}(0)=a(0)c(0)$. This equation can be expressed as
$$\partial_t( b \sqrt{-a}^{-1})(0)=\frac{-2\dot b (0)  a(0) +b(0) \dot{a}(0)}{2(-a(0))^{3/2}}=\frac{-c(0)}{2 \sqrt{-a(0)}}$$
Then
$$-2 {\rm Re}\;  \partial_t (\ln  b \sqrt{-a}^{-1} ) (0)= -2  \partial_t  \ln |\tau|  ={\rm Re}  \frac{c(0)}{b(0)}.$$
\begin{lemma}
\label{increasetau}
If $g>0$, then there exits a $c(\l)$ preserving the closing condition and reality condition which increases $|\tau|$ along its integral curves.
At a local maxima of $|\tau|$, we have $g=0$.
\end{lemma}
\begin{proof}
We choose independently of the fact that we have a singularity in the flow a vector field with ${\rm Re} \tfrac{c(0)}{b(0)} <0.$
This condition assures that we increase the third coordinate of the flux along the deformation.
Choose $c(\l)=(\l-1)(b(0)-\bar b(0) \l^g)$ and check that $c(0)=-b(0)$. By appendix C when $g=0$ then $b(0)=\frac{-\pi}{16} \in \R$ and $c=0$.
\end{proof}
\subsection{Flow without common roots of $a$ and $b$}
In the case where $b$ has higher order roots $\beta_j$ of order $k$, we need to consider the system (\ref{eq:integrability_2}), with additional equations $\dot b ' ( \beta_j), ...,\dot b^{(k-1)} (\beta_j)$. The equations can be easily derived from the derivatives
of (\ref{eq:integrability_1}), taking onto account that $b(\beta_j)=b'(\beta_j)=....=b^{(k-1)}(\beta_j)=0$.
\begin{theorem}\label{thm:deformation}
Let $U$ be an open subset of spectral data $(a,\,b)$ in ${\calM}_g$, with $a,\,b$ having no common roots.
Then a polynomial $c$ of degree $g+1$ which satisfies the reality condition \eqref{eq:realc} exists, and the equations \eqref{eq:integrability_1} define a smooth vector field on $U$.

If {\rm{(i)}} $c(1)=0$, and {\rm{(ii)}} ${\rm Im}\,(c(0)/b(0))=0$, then integral curves $(a(t),\,b(t)) \in \C^{2g}[\l] \times\C^{g+1}[\l]$ are spectral data of a continuous family of minimal annuli in $\SY ^2 \times \R$. If ${\rm Re}\,(c(0)/b(0)) <0$, the curve  $(a(t),\,b(t))$ is the spectral data of minimal annuli with flux bounded away from $0$.
\end{theorem}
\begin{proof}
The spectral data $(a,b)$ satisfy hypothesis of Definition~\ref{spectraldata}.
The solutions $\dot{a}$ and $\dot{b}$ of
equation~\eqref{eq:integrability_1} are rational expressions of the
coefficients of $a$, $b$ and $c$. If $a$ and $b$ have no common roots,
then the Taylor coefficients of $\dot{a}$ and $\dot{b}$ at the roots
of $a$ and $b$ up to the order of the roots minus one, respectively,
and the highest coefficient of $\dot{b}$
are uniquely determined by equation~\eqref{eq:integrability_2}.
Hence in this case the denominators of the rational expressions for
$\dot{a}$ and $\dot{b}$ do not vanish. Hence for smooth $c(\l)$ the
corresponding $\dot{a}$ and $\dot{b}$ are smooth too. We can thus find an integral curve in the set
of complex polynomial.

The transformation rules of $\mu$ under $\sigma$, $\rho$ and $\eta$ are
preserved under the flows of the vector field corresponding to $c$.
Hence the integrals of $d\ln\mu$ along any smooth path from one root
of $a$ to another root of $a$ is preserved too. This implies that
the subset of spectral data $(a,b)$, which determine a single valued function $\mu$ with
$\sigma^{\ast}\mu=\mu^{-1}$, $\rho^{\ast}\mu=\bar{\mu}^{-1}$ and
$\eta^{\ast}\mu=\bar{\mu}^{}$ is preserved under this flow.

Due to \eqref{eq:def_c} $\partial_t\ln \mu$ is a meromorphic function on the hyperelliptic curve $\Sigma$ and the periods of the meromorphic differential $d\ln\mu$ do not depend on $t$.  The closing condition of definition \ref{spectraldata} is preserved if $\frac{\rmd}{\rmd t}\ln \mu(\a_i(t),\,t) = 0$ for all roots $\a_i (t)$
of $a(t)$, which holds precisely when
$${\rm d} \ln
\mu(\a_i(t),\,t)\,\partial_t\a_i +
    \partial_t \ln \mu(\a_i(t),\,t) = 0\,.$$
Using equations \eqref{eq:def_c} and \eqref{eq:dlogmu}, the closing
conditions are therefore preserved if and only if
\begin{equation}\label{eq:integrability_4}
    \dot{\a}_i =
  \frac{-\a_i  \,c(\a_i)}{b(\a_i)}\,
\end{equation}
This equation describes the deformation of the spectral curve $\Sigma$. The polynomial $b$ exists and integrals along cycles are multiples of $\mi \pi \Z$.
The condition $c(1)=0$ ensures that the value of $\mu$ remains constant at $\lambda =1$, that is $\mu (1)=\pm 1$.

To close the annulus, we have to keep $b(0) \in e^{\mi\Theta/2} \R$ where $a(0) = \frac{-e^{\mi\Theta}}{16}$ for each $t\in \R$.
At $\l=0$, we observe that $-2\dot b (0)  a(0) +b(0) \dot{a}(0)=a(0)c(0)$ and as in the computation of increasing flux
$\partial_t (\ln  b \sqrt{-a}^{-1} ) (0) = \tfrac{-c(0)}{2b(0)}.$
To close the period we need that $b(0)\sqrt{-a(0)}^{-1} =\frac{ \tau e^{\mi\Theta/2}}{8} \in \R$ along the deformation. This gives condition (ii).

In order to control the flux, we need to control the period $\tau$ away from zero. This explains the condition on ${\rm Re}\frac{c(0)}{b(0)} <0$
along the deformation (see lemma \ref{increasetau}).
\end{proof}
\subsection{Smooth parametrization of spectral data and opening double points}
\label{param}
We consider a non-singular spectral curve $\Sigma$ of genus $g$ ($a$ has only simple roots) and spectral data of a minimal annulus given by $(a,b) \in {\calM}_g$. We consider $\mu =f \nu +g$ (see definition \ref{spectraldata}) with $g^2- f^2 \nu^2=1$ (using
the hyperelliptic involution). Roots of $f$ are points where $\mu^2=1$. These points are called double points. A possible higher genus 
spectral curve can be desingularized at such points. In this section we do the opposite. This means that we add singularities to the spectral data
$(a,b)$ and embedd the moduli space ${\calM}_{g} \hookrightarrow {\calM}_{g+n}$. In the next subsection we will see
that in the larger moduli space ${\calM}_{g+n}$, we can pass through a singularity of the vector field induced by $c$ (as describe in Theorem \ref{thm:deformation}), when $a$ and $b$ have common roots or when we want to increase the genus of a spectral curve $\Sigma$ 
by opening a double point.

By the hyperelliptic involution we have that $\mu^{-1}=-f \nu +g$. Hence

$$f= \frac{\mu - \mu^{-1}}{2 \nu} \hbox{ and } g =\frac{\mu + \mu^{-1}}{2}$$

When $a$ and $b$ have a common root at $\alpha_0$, then we integrate $d \ln \mu =\frac{b d \l}{\nu \l}$ near $\alpha_0$ and obtain $ \ln \mu = f_0(\l) (\l - \alpha_0)^d +g_0$ with  $d \geq 3/2$ and a local holomorphic function $f_0$ with $f_0 (\alpha_0)\neq 0$. Hence $f(\l)= \sinh (f_0(\l)(\l- \a_0)^{d-1/2})$ and $f$ vanishes at $\alpha_0$ with same order as $f_0(\l)(\l- \a_0)^{d-1/2}$.

If $\mu (\alpha_0)=\pm 1$, then we have $a(\alpha_0)=0$ or $f(\alpha_0)=0$. In the case where both are zero we have common roots of $a$ and $b$. When only $f(\alpha_0)=0$ and $a(\alpha_0)\neq 0$, then we have a simple double point (which is not a branch point).
There exist infinitely many such double points. In the last case, when $a(\alpha_0)=0$ and $f(\alpha_0) \neq 0$, then we have a simple branch point of $\Sigma$.

We consider spectral data $(\tilde a, \tilde b)= (a p^2, b p)$, such that $\mu$ can be written as $\mu =\tilde f \tilde \nu + \tilde g$ with holomorphic functions $\tilde f$ and $\tilde g$ on $\C^*$, with $\tilde \nu ^2= \frac{\tilde a (\l)}{\l}$ and $\tilde f$ does not vanish at the roots of $\tilde a$. This means that we add to $a$ an appropriate number of double roots at all common roots of $a$ and $b$. Moreover we have the choice
to do the same simultaneously at finitely many double points. Now we construct a smooth parametrization of $(\tilde a, \tilde b)$ in ${\calM}_{g+n}$.

We choose simply connected neighbourhoods $V_1,\ldots,V_M$ in
$\C^\ast$ at all roots of $\tilde b$ including the common roots with $\tilde a$.
Let $U_1,\ldots,U_M$ denote the pre-images in
$\Sigma\setminus\{0,\infty\}$ of $V_1,\ldots,V_M$ under the map $\l$.
For $m=1,\ldots,M$ we choose on $U_m$ a branch of the function
$\ln\mu$. On $U_m$ the function $\frac{1}{2\pi \mi}(\ln\mu+\sigma^\ast\ln\mu)$ is equal to a constant integer $n_m$ at the branch point.
Theses branches obey
\begin{equation}\label{eq:local description}
(\ln\mu-n_m \mi\pi)^2=A_m \quad\hbox{ for }\quad m=1,\ldots M,
\end{equation}
with holomorphic function $A_m$ on $V_m$ which vanish at the roots of $\tilde a$.
Since $\sigma^*(\ln\mu-n_m \mi\pi)=-(\ln\mu-n_m \mi\pi)$, the function $A_m$ depends only on $\l$ (see Forster \cite{Forster-RS}, theorem 8.2).
If we choose $U_m$ and $V_m$ small enough, then the derivative of $A_m$ has no roots besides
the corresponding root of $\tilde b$ ($d \ln \mu$ vanishes at roots of $b$). The roots of $\tilde b$ are exactly the roots of the derivative of $A_m$.
For small enough $U_m$ and $V_m$ there exists a biholomorphic map $\l\mapsto z_m(\l)$ from $V_m$
to a simply connected open neighbourhood $W_m$ of $0\in\mathbb{C}$, such
that $A_m$ coincides with
\begin{equation}\label{eq:pol A}
A_m(\l)=z_m^{d_m}(\l) + a_m
\end{equation}
At a root of $\tilde b$, which is not a root of $\tilde a$ (i.e. a root of $b$ which is not a root of $a$, and so a root of $d \ln \mu$),
the constant $a_m \neq 0$, and $d_m -1$ is the order of the roots of $b$. At a common root of $\tilde a$ and $\tilde b$, the constant $a_m =0$ and $d_m$ is an odd integer (i.e.  common roots of $a$ and $b$ including the case of double points).

We  describe  spectral curves in a neighbourhood of the given
spectral curve by small perturbations $\tilde{A}_1,\ldots,\tilde{A}_M$
of the polynomials $A_1,\ldots,A_M$. More precisely, we consider
polynomials $\tilde{A}_1,\ldots,\tilde{A}_M$ of the form
\begin{equation}\label{eq:deformed pol A}
\tilde{A}_m(z_m)=z_m^{d_m}+\tilde{a}_{m,1}z_m^{d_m-1}+ \tilde{a}_{m,2}z_m^{d_m-2}+\ldots+\tilde{a}_{m,m}
\end{equation}
with coefficients $\tilde{a}_{m,2},\ldots,(\tilde{a}_{m,m}-a_m)$ nearby zero. By a shift $z \to z+z_0$, we can always assume that the sum of the roots is zero and then $\tilde a_{m,1}=0$. We glue each $W_m$ of the sets $W_1,\ldots,W_M$ to $\mathbb{P}^1\setminus(V_1\cup\ldots\cup V_M)$ along the boundary of $V_m$ in such a way
that for all $m=1,\ldots,M$  the polynomial $\tilde{A}_m $ coincides with the unperturbed function $A_m $ in a tubular neighborhood of the boundary $\partial W_m$. We obtain a new copy of $\C\mathbb{P}^1$. By uniformization, there exists a new global parameter $\tilde{\l}$, which is equal to $0$ and $\infty$ at the two points corresponding to $\l=0$ and $\l=\infty$, respectively. This new parameter is unique up to multiplication with elements of $\C^\ast$. There exists a biholomorphic map $\tilde \l= \phi (\l)$ which changes the parameter $ \l \in \C\mathbb{P}^1\setminus(V_1\cup\ldots\cup V_M)$ in the global parameter $\tilde \l$. There
is a biholomorphic map $\tilde \l=\phi_m (z_m)$ which changes the local parameter $z_m \in W_m$ into $\tilde \l$.
Let $\tilde \l \to \tilde a ( \tilde \l)$ be the polynomial whose roots (counted with multiplicities) coincide with the zero set of $\tilde A _1 (\tilde \l),...,\tilde A _M (\tilde \l)$ and the roots of
$\tilde \l \to a \circ \phi ^{-1} (\tilde \l)$ on  $\mathbb{P}^1\setminus(V_1\cup\ldots\cup V_M)$.

Now $\tilde \Sigma =\{ (\tilde \nu, \tilde \l ) \in \C^2 \mid \nu ^2 =\frac{\tilde a (\tilde \l)}{\tilde \l} \}$ is a new hyperelliptic Riemann surface associated to the set of polynomials
$\tilde A_1,...,\tilde A_M$.  We say that polynomials $\tilde{A}_1,\ldots,\tilde{A}_M$ respect the reality condition if the involutions $\sigma$, $\rho$ and $\eta$ lift to involutions
of $\tilde \Sigma$ and then define a spectral curve. In this case, the parameter $\tilde \l$ is determined up to a multiplication of unimodular numbers. We call this transformation a M\"obius transformation. The equations
\begin{equation}
\label{fonctiontilde}
(\ln\mu-n_m i\pi)^2=\tilde A_m(\tilde \l)= \tilde A_m \circ  \phi _m ^{-1}  (\tilde \l )=\tilde A_m (z_m) \quad\hbox{ for }\quad m=1,\ldots M
\end{equation}
define a function $\mu$ on the pre-image of $\phi_m(W_m) \cap \mathbb{P}^1$ by the map $\tilde \l$ into  $\tilde \Sigma$. The function $\mu$ extends to the pre-image of $\C^* \setminus(V_1\cup\ldots\cup V_M)$ by $\tilde \l=\phi (\l)$ and coincides with the unperturbed $\mu$ on this set.

On $\tilde \Sigma$ the differential $d \ln \mu$ is again meromorphic and takes the form $d \ln \mu = \frac{\tilde b (\tilde \l) \,d \tilde \l}{\tilde \nu \tilde \l^2}$ with a unique polynomial $\tilde b$. By taking the derivative of (\ref{fonctiontilde}) we have
\begin{align*}
2(\ln\mu-n_m\pi \mi)\tfrac{d}{d \tilde \l}\ln\mu&=
\tilde A_m'(z_m(\tilde \l))z'_m( \tilde \l)\,.
\end{align*}
Then roots of $\tilde b$ are the roots of the derivatives of $\tilde A_1,...,\tilde A_M$.

\begin{proposition}
\label{opendoublepointprop}
The set of  polynomials $\tilde A_1,\ldots, \tilde A_M$ which respect the reality condition and coefficients $\tilde{a}_{m,2},\ldots,(\tilde{a}_{m,m}-a_m)$ nearby zero define spectral data of periodic solutions of the sinh-Gordon equation in a neighbourhood of $ \Sigma$. The spectral genus is generically  $g+n$. A $\tilde c$ depending smoothly on $(\tilde a, \tilde b)$ defines a smooth vector field on these parameters. Integral curves of $\tilde c$ in the set of spectral data deform $\tilde A_1,\ldots, \tilde A_M$.
\end{proposition}
\begin{remark}
\label{opendoublepointrem}
By definition \ref{spectraldata}, the existence of the function $\mu$ assures the periodicity of the polynomial Killing field with initial value $\tilde \xi_\l \in I(\tilde a)$.
If $\tilde c(1)=0$ and ${\rm Im}(\tilde c(0)/\tilde b(0))=0$, the deformation preserves the closing condition of the surface along the integral curves defined by $\tilde c (\tilde \l)$. The condition ${\rm Re}(\tilde c(0)/\tilde b(0))<0$ increases the length of the flux $|\tau|$.
\end{remark}

\begin{proof}
First given a smooth family of $\tilde A_1 (t) ,\ldots, \tilde A_M (t)$ which respect the reality condition, we calculate the change of $\ln\mu$ and the uniquely defined corresponding polynomial $\tilde c(\tilde \l)$.  In a second step, we will prove that a given polynomial $\tilde c (\tilde \l)$ will smoothly determine
a vector field $\dot {\tilde A}_1,\ldots, \dot {\tilde  A}_M$.
If the polynomial $A_m$ changes, then also the biholomorphic map $\l\mapsto z_m(\l)$ changes. If we
differentiate (\ref{fonctiontilde}) with $z_m (\tilde \l)=\phi^{-1}_m (\tilde \l)$ we obtain
\begin{align*}
2(\ln\mu-n_m\pi \mi)\tfrac{d}{dt}\ln\mu&=
\dot{\tilde A}_m(z_m(\tilde \l))+\tilde A_m'(z_m(\tilde \l))\dot{z}_m(\tilde \l)
\end{align*}
The equations \eqref{eq:def_c} and \eqref{eq:dlogmu} imply
\begin{equation}\label{eq:A and c}
\frac{\tilde \l \tilde c( \tilde \l)}{\tilde b(\tilde \l)}=
\frac{\dot{\tilde A}_m(z_m(\tilde \l))}{\tilde A'_m(z_m(\tilde \l))z_m'(\tilde \l)}+\frac{\dot{z}_m(\tilde \l)}{z_m'(\tilde \l)}.
\end{equation}
On the right hand side the second term has no poles at the roots of $b$ since $z_m (\tilde \l)$ is biholomorphic on $V_m$ (a derivative of a biholomorphic map doesn't vanish). The singular parts of the meromorphic function on the left hand side at the roots of $\tilde b$ are determined by the first term on the right hand side. The family ${\tilde A}_1 (t) ,\ldots,{\tilde A}_M (t)$ and the finitely many values of the derivatives
of $z_m (\tilde \l)$ at the roots of $\tilde b (\tilde \l)$ determine the values of $\tilde \l \to \tilde c(\tilde \l)$ at all roots of $\tilde b$. The polynomial $\tilde c$ is determined up to a multiple
constant of $\tilde b$  by a purely imaginary complex number ($\tilde c \to \tilde c +\mi \gamma b$).
This degree of freedom comes from the M\"obius transformation on the parameter $\tilde \l$. If $\tilde c=\mi \gamma \tilde b$, then $\tilde \l \to e^{\mi\gamma t} \tilde \l$.
The choice of the parameter $\tilde \l$ fixes the degree of freedom in the polynomial $\tilde c$.

Conversely, the choice of $\tilde c$ determines the value of $\dot {\tilde A}_m$ at all the zeroes of $\tilde b$ which coincide with the zeroes of $\tilde A '_m$. By definition of $\tilde A_m$, the degree of $\dot {\tilde A}_m$ is less than the degree of $\tilde A '_m$
(The two highest coefficients are independent of $t$). The family $\dot {\tilde A}_m$ depends linearly on $\tilde c$.
\end{proof}

\subsection{Flow with common roots of $a$ and $b$}
We assume in this section that $a$ has only simple roots.
At common roots of $a$ and $b$ the vector field defined by polynomials $c$ and equation~\eqref{eq:integrability_2} has singularities. We study this
situation and describe how to continuously extend the flow through such singularities. For this purpose we consider the embedding of ${\calM}_g \hookrightarrow {\calM}_{g+n}$ as described in section \ref{param}. Common roots of $a$ and $b$ should be considered as higher order roots of $\tilde a$. (We define  $(\tilde a, \tilde b)= (a p^2, b p)$ where $p$ is the polynomial whose roots coincide with roots of $f$ ($f$ is defined
in definition \ref{spectraldata}~(vi)) counted with multiplicity at common roots of $a$ and $b$. Then  $f=p \tilde f$, $\tilde f$ has no roots at common roots of $a$ and $b$.

The order $d_m=2 \ell_m+1$ is odd and at least three. The genus of the spectral curve is preserved, if the number of odd order roots of $\tilde a$ is preserved (see below).
This is equivalent to the condition that all $\tilde{A}_m$ of odd degree have only one odd order root. Hence we consider the polynomials $\tilde A_m$  of the form
\begin{align*}
\tilde{A}_m(z_m) &=(z_m-2\alpha_{m})p_m^2(z_m)
\quad\hbox{ with }\\
p_m(z_m)&=z_m^{\ell_m}+\beta_{m,1}z^{\ell_m-1}+\ldots+\beta_{m,\ell_m} \hbox{ with } \beta_{m,1}=\alpha_{m}
\end{align*}
Since the sum of the zeroes of $\tilde A_m$ is equal to zero, the odd order root $2 \alpha_m$ is given by
$\beta_{m,1}=\alpha_m$. In this case, all double roots of $\tilde a$ will produce double roots of the spectral curves and will not contribute
to the geometric genus of the spectral curve. So we remain in the image of the embedding space ${\calM}_g \hookrightarrow {\calM}_{g+n}$. 
In this parametrization double roots of $A_m$ are simple roots of $\tilde b$. The corresponding $(a,b) \in {\calM}_g$ have no double points and correspond to the same function $\mu$. The coefficients of $\tilde{A}_m$ are parametrize by the coefficients $\alpha_m=\beta_{m,1}$ and $\beta_{m,2},...,\beta_{m,\ell_m}$.

\begin{proposition}
\label{commonroots}
Let $(a,b)$ be the spectral data of a periodic solution of the sinh-Gordon equation. Suppose $a$ and $b$ have common roots. Assume $c$ is a polynomial of degree $g+1$ such that $\l^{g+1} \overline{c (1/\bar \l )}=c(\l)$ ($c$ can depend on $a$ and $b$ smoothly and does not vanish
at $(a_0,b_0)$) that does not vanish at the common roots of $a$ and $b$. Then there exists for some $\epsilon >0$ a continuous integral curve $(a_t, b_t)_{t \in (-\epsilon, \epsilon)}$
with $(a_0,b_0)=(a,b)$ along the vector field induced by $c$, which is smooth on $(-\epsilon, \epsilon) \backslash \{0\}$. Moreover for $t \in (-\epsilon, \epsilon) \backslash \{0\}$, $a_t$ and $b_t$ have no common roots.
\end{proposition}

\begin{proof}
For a given $c$, we consider the polynomial $\tilde c$ with $c/b=\tilde c / \tilde b$ on the set of spectral data $(\tilde a, \tilde b)$ in ${\calM}_{g+n}$. This definition assures that along any integral curve of $\tilde c$ in ${\calM}_{g+n}$, all polynomials $\tilde A_m$ keep the form  $\tilde A_m =(z-2\alpha_m)p^2_m(z)$ where $\alpha_m$
is a moving branch point along an integral curve of the vector field $\tilde c$ nearby a common root of $a$ and $b$ located at zero in this coordinate. In the following we will parameterize integral curves by moving the roots $\alpha_1,...,\alpha_m$ and coefficients of $p_1,...,p_m$ polynomials of degree $\ell_1,...,\ell_m$.

At $t=0$, the vector field induced by $\tilde c $  has a singularity. We multiply the vector field with a real function depending on the coefficient of $a$ and $b$ which vanishes at common roots of $a$ and $b$ in such a way that the new vector field becomes smooth and without poles at $t=0$. The new vector field will have a root at $t=0$. We want to find a trajectory moving into the root of the vector field and a trajectory moving out of the root. For this purpose,
we calculate the first derivative of the vector field and find non trivial stable and unstable eigenspaces. By the Stable Manifold Theorem (see e.g Teschl \cite{teschl2012}) there then exist integral curves moving in and out the zero of the vector field.

We have to collect different common roots of $a$ and $b$ and prove that the linearized vector field has non-empty stable and unstable eigenspaces. This means that the first derivative of the vector field at a common root of $a$ and $b$ has non-zero eigenvalues with positive and negative real parts. We compute the first derivative
at common roots $\alpha_1,...,\alpha_m$ separately and coefficients of $p_1,...,p_m$ at $t=0$.

\vskip 0.25cm
In $V_m$ we consider the polynomial $A_m (z_m)$ together with the parameter $z_m$ introduce in Section 9.2. In the sequel we drop the subscript $\bullet_m$ and we consider $A(z)=(z-2 \alpha)p^2 (z):=\alpha ^{2\ell+1} \left( \frac{z}{\alpha}-2\right) q^2(\frac{z}{\alpha})$ 
where $\alpha \in \C$ is a complex value close to zero, and $\ell :=\ell_m$ is the degree of the polynomial $q$.
The polynomial $q$  depends on $t$ and $\alpha (t)$ with $\alpha (0)=0$ and $\alpha ^\ell(t) q(\frac{z}{\alpha (t)})\to z^\ell$ when $t \to 0$.

We denote $w=\frac{z}{\alpha}$ and we set $q(w)=w^\ell+w^{\ell-1}+\gamma_{m,2}w^{\ell-2}+\gamma_{m,3}w^{\ell-3}+...+\gamma_{m,\ell}$. 
As long as $\alpha =0$, the polynomial $A(z)=z^{2\ell+1}$ independently of the choice of the polynomial $q$. Hence we can choose freely an appropriate $q$ at $t=0$. To apply the Stable Manifold Theorem, we need that $\dot \alpha$ does not vanish at $t=0$. In particular $\frac{\dot q}{\dot \alpha}$ should be bounded at $t=0$. The choice of $q$ at $t=0$ is specified in Lemma \ref{pol}.

\begin{equation*} \begin{split}
    \dot A (z)&= \dot \alpha \,\alpha^{2\ell}  q \left( \tfrac{z}{\alpha} \right) \left((2\ell+1)\left( \tfrac{z}{\alpha}-2 \right) q \left(\tfrac{z}{\alpha} \right)
    -\left( \tfrac{z}{\alpha} \right)q \left( \tfrac{z}{\alpha} \right)-2\left( \tfrac{z}{\alpha} \right)\left( \tfrac{z}{\alpha}-2\right)q' \left( \tfrac{z}{\alpha}\right) \right) \\
    &\qquad + 2 \alpha^{2\ell+1} q \left( \tfrac{z}{\alpha} \right)  \left( \tfrac{z}{\alpha}-2\right) \dot q \left( \tfrac{z}{\alpha} \right) \\
    A' &= \alpha^{2\ell} q \left( \tfrac{z}{\alpha}\right) \left( q \left(\tfrac{z}{\alpha} \right) +2\left( \tfrac{z}{\alpha} -2 \right)q' \left(\tfrac{z}{\alpha} \right) \right) \\
    \frac{\dot A}{A'} &=\frac{ \dot \alpha \left( (2\ell+1)(w-2)q (w) - w q (w)-2w (w-2)q' (w) \right) +2\alpha (w-2)\dot q (w)}{ q (w) +2(w-2)q' (w)}
\end{split}
\end{equation*}
Now $A'$ has $2\ell$ roots. Besides the $\ell$ double roots of $A$ it has $\ell$ additional roots, which are equal to the roots of the polynomial
$$\frac{\alpha ^\ell}{2\ell+1} \left( q  \left( \frac{z}{\alpha} \right) + 2\left( \frac{z}{\alpha} -2 \right) q' \left( \frac{z}{\alpha} \right) \right).$$
This polynomial has highest coefficient $1$. Locally in $V_m$, the function $\frac{ \tilde \l \tilde c (\tilde\l)}{\tilde b(\tilde\l)}$ may be uniquely decomposed into a rational function
depending on $z \in W_m \subset \C$, which vanishes at $z \to \infty$ and a holomorphic function $h(z)$ nearby the roots of $\tilde b$. Hence there exists a unique polynomial $C(z) $ of degree $\ell-1$, such that in $V_m \simeq W_m$, the Laurent decomposition gives
$$\frac{\tilde\l \tilde c }{\tilde b}=\frac{C(z)}{\frac{\alpha^\ell}{2\ell+1} \left( q \left(\frac{z}{\alpha}\right) + 2\left( \frac{z}{\alpha}-2 \right) q' \left(\frac{z}{\alpha} \right) \right) } + h(z)\,.$$
Therefore we obtain the differential equation in $V_m$
$$\frac{\dot A (z)}{A' (z)} = \frac{ C(z)}{\frac{\alpha^\ell}{2\ell+1} \left( q \left( \frac{z}{\alpha} \right) + 2\left( \frac{z}{\alpha} -2 \right) q' \left( \frac{z}{\alpha} \right) \right)}\,.$$
This equation is equivalent to
$$\dot \alpha \left(  (2\ell+1)(w-2)q (w)-w q (w) - 2 w (w-2)q'(w) \right) + 2\alpha (w-2)\dot q (w) = \frac{(2\ell+1) C(\alpha w)}{\alpha ^\ell}.$$
We have $\alpha^\ell q \left( \frac{z}{\alpha} \right) \to z^\ell$ as $t \to 0$ and at the starting point, we should choose $q$ in such a way, that $\frac{\dot q}{\dot \alpha}$ stays bounded. If we do that then we can neglect the second term on the left hand side of the equation which is of order $O( \alpha)$ for small $t$.
On the right hand side we have $C(0) \neq 0$ and at the starting point, we choose a polynomial $q$ with
$(2\ell+1)(w-2)q (w) - w q(w) -2 w (w-2)q' (w)$ constant. Therefore $q$ is given by the polynomial in the lemma \ref{pol}.

\begin{lemma}
\label{pol}
For each $\ell \in \N$ there exists a unique polynomial $\widehat q$ of degree $\ell$ with highest coefficients $1$, such that
$$(2\ell+1)(w-2)\widehat q (w)-w  \widehat q (w)- 2 w (w-2) \widehat q' (w) =K$$
where $K$ is a constant. This polynomial is the polynomial part of $w^\ell \left(1-\frac{2}{w} \right)^{-1/2}$.
\end{lemma}

\begin{proof}
We consider the expression $(2\ell+1)(w-2) \widehat q (w)-w \widehat q (w)- 2 w (w-2)\widehat q' (w)$, which for polynomials $\hat{q}$ of degree $\ell$ with highest coefficient one is a polynomial of degree at most $\ell$. The condition that this polynomial is constant yields $\ell-1$ linear equations on the coefficients of $\widehat q$. We can solve these equations uniquely by first defining the coefficient of $w^{\ell-1}$ in $\widehat q$ such that the coefficient of $w^\ell$ of the polynomial under consideration vanishes and then the lower order coefficients in the inverse order of their power. If we insert for $ \widehat q(w)=w^\ell \left(1- \frac{2}{w} \right)^{-1/2}$, then $A=\alpha ^{2\ell+1} \left( \frac{z}{\alpha}-2\right) q^2(\frac{z}{\alpha})$ is equal to $z^{2\ell+1}$ and independent of $\alpha$. Consequently the former expression vanishes. Moreover, if $\hat{q}(w)$ is the polynomial part of $w^\ell \left(1- \frac{2}{w} \right)^{-1/2}$, then $A(z)-z^{2\ell+1}$ is a polynomial of degree $\ell+1$ with respect to $z$ and $w=\frac{z}{\alpha}$. Differentiating with respect $\alpha$ shows that the polynomial under consideration is constant.
\end{proof}

{\emph{Continuation of the proof of Proposition~\ref{commonroots}}:} For this polynomial $\widehat q$ we have
\begin{align*}
(2\ell+1)(w-2)\widehat q (w) - w\widehat q (w) -2 w (w-2) \widehat q' (w)= & \begin{pmatrix}- \frac{1}{2} \cr \ell \end{pmatrix} (-2)^{\ell+1}(2\ell+1) \\
= & \frac{1.3...(2\ell-1)(2\ell+1)}{\ell!} (-2)
\end{align*}
The solution of the differential equation obeys at the initial value the equation
\begin{align}
\label{alpha}
\dot \alpha  & = \frac{ -(2\ell+1).\ell!}{2 \alpha ^\ell 1.3...(2\ell-1)(2\ell+1)} \left( C(0)+ O(\alpha) \right) \\
\dot q (w) & =  \frac{\dot \alpha}{2 \alpha} (h_n (w) + O ( \alpha) )
\end{align}
where $h_n (w)$ is the polynomial part of $2w \left( q' (w)-\widehat q '  (w) \right)- (2\ell+1) \left(  q (w) - \widehat q (w) \right) + \left(1-\frac{2}{w} \right)^{-1} (q (w) - \widehat q (w) )$
where $\widehat q (w)$ is the polynomial in the lemma above. We remark that the first equation is a complex valued meromorphic equation. The second equation is the equation which involves the move of coefficients $\gamma(t) -\widehat \gamma=\gamma (t) - \gamma (0)$. To remove all poles of the vector field, we have
to multiply the right hand side by $\alpha ^{\ell+1}$ such that $\dot q$ does not has a pole. As a consequence $\dot \alpha$ is of order $O(\alpha)$ and $\dot q$ is of order
$O(q-\widehat q)$ for small $t$. So the linearized vector field has block diagonal form with respect to the decomposition of $\alpha$ and $q$. Since we are interested in trajectories on which $\dot q / \dot \alpha$ is bounded, we restrict to the eigenspaces of the $\dot \alpha$ equation. For these eigenspaces $\dot \alpha$ is not zero while $\dot q$ vanishes in the lowest order.

It remains to find eigenvalues of the equation (\ref{alpha}) with non zero real parts. We collect all equations corresponding to each $m$, and end up with equations of the form
$$\dot \alpha_m = \frac{C_m (1+O(\alpha_1, ...,\alpha_M))}{\alpha_m^{\ell_m}}$$
where $\ell_m$ is the degree of $q_{m}$ corresponding to $A_m (z_m)$, with $m=1,...,M$.
Let $N$ be the least common multiple of $\ell_1+1, ...,\ell_M+1$. Then we use the parameters $\alpha_1= e^{\mi \theta_1}  s^{\frac{N}{\ell_1+1}}$
and $\alpha_m= e^{\mi \theta_m} r_m s^{\frac{N}{\ell_m+1}}$ for $m>1$ with real $s, r_m, \theta_m$  to describe the evolution of $A_1,...,A_m$ (The term $O(\alpha_1,..,\alpha_m)=O(s^p)$ with $p=\inf N/ \ell_m+1$). Here we assume that the root of $b$ of index $m=1$ is a common root of $a$ and $b$. Then with initial value $(s_1(0),\theta_1(0))=(0, \bar \theta_1)$ we have
$$\dot \alpha_ 1= e^{\mi \theta_1} s^{\left( \frac{N}{\ell_1+1}-1 \right)} \left( \frac{N}{\ell_1 +1} \dot s  + \mi s \dot \theta_1 \right).$$
Now we multiply the whole vectorfield with the real parameter $s^N$. The corresponding integral curves will be reparameterized by $ u=t  s^N$, but do not change
as subsets in the space of polynomials  $(a,b)$. Then we obtain for $m=1$ the equation
$$e^{\mi \theta_1} s^{\left( \frac{N}{\ell_1+1}-1 \right)} \left( \frac{N}{\ell_1+1} \dot s + \mi s \dot \theta_1 \right)=\frac{ C_1 s^{\frac{N}{\ell_1+1}} }{e^{\mi\ell_1 \theta_1}} (1 + O(s)).$$
This gives the system
\[
    \dot \theta_1 = {\rm Im}\,\bigl( C_1 e^{-\mi(\ell_1+1)\theta_1} \bigr) + O(s^p)\,,\qquad
    \dot s  =\tfrac{\ell_1+1}{N} s \; {\rm Re}\,\bigl(  C_1 e^{-\mi(\ell_1+1)\theta_1} \bigr) + O(s^{p+1})\,.
\]
Now choose $\bar \theta_1$ suitably to get ${\rm Im}\left( C_1 e^{-\mi(\ell_1+1)\bar \theta_1} \right)=0$, and recall that $C_1 \neq 0$ by choosing the polynomial $c$ without zeroes at branch points $\alpha_1$. This implies that
$$\frac{\partial}{\partial \theta_1} {\rm Im}\left( C_1 e^{-\mi(\ell_1+1)\theta_1} \right)_{| \theta_1=\bar \theta_1} =-(\ell_1+1)  {\rm Re}\left( C_1 e^{-\mi(\ell_1+1)\bar \theta_1} \right)= (\ell_1+1)  \widehat c_1 \neq 0\,.$$
We have to choose the initial value $\bar \theta_1$ at the starting point in the exceptional fibre of the blow up in such a way that $\widehat c_1 = \pm |C_1|$ has different sign. Then there exists different solution with negative and positive eigenvalues of the corresponding linearized equation
\[
    \dot \theta_1 = (\ell_1+1)\widehat c_1 \theta_1 \,,\qquad
    \dot s = \tfrac{\ell_1+1}{N} s \; {\rm Re}\,\bigl(  C_1 e^{-(\ell_1+1)\bar \theta_1} \bigr)=-\tfrac{(\ell_1+1)  \widehat c_1}{N} s
\]
For $m >1$, we use parameters $(s, r_ m ,\theta_m)$ with initial values $(0,\bar r_m, \bar \theta_m)$ and $\bar r_m \neq 0$, for $\alpha_m$ to get (with reparametrization by $u = t s^N$) the equation
\begin{align*}
\dot \alpha_m & =\left(\mi r_m  \dot \theta_m s+ \dot r_m s + \frac{N}{\ell_m+1}r_m \dot s \right) e^{\mi\theta_m} s^{\frac{N}{\ell_1+1}-1} \\
&  =\frac{C_m s^N (1 +O(s^p))}{s^{\frac{\ell_m N }{\ell_m +1}} e^{\mi \ell_m \theta_m} r_m^{\ell_m}} =C_m e^{- \mi \ell_m \theta_m}r_m^{-\ell_m} s^{\frac{N}{\ell_m+1}} (1+O(s^p))
\end{align*}
This implies
$$i r_m  \dot \theta_m s+ \dot r_m s + \frac{N}{\ell_m+1}r_m \dot s=C_m e^{- \mi( \ell_m+1) \theta_m}r_m^{-\ell_m} s (1+O(s^p))\,.$$
Hence we study the system
\begin{align*}
\dot  \theta_m & = {\rm Im} \left( C_m e^{-\mi (\ell_m +1) \theta_m} \right) r_m ^{-(\ell_m+1)} +O(s^p) \\
\dot r_m & = {\rm Re} \left( C_m e^{-\mi (\ell_m +1) \theta_m} \right) r_m ^{-\ell_m}-\frac{\ell_1+1}{\ell_m+1} {\rm Re} \left( C_1 e^{-\mi (\ell_1 +1) \theta_1} \right) r_m +O(s^p)
\end{align*}
Now we choose $\bar \theta_m$ in such a way that ${\rm Im} \left( C_m e^{-\mi (\ell_m +1)\bar  \theta_m} \right)=0$ and if $\widehat c_m= -{\rm Re} \left( C_m e^{-\mi (\ell_m +1)\bar  \theta_m} \right)$, we can choose different values of $\bar  \theta_m$ to get different signs of $\widehat c_m = \pm |C_m|$. For
a choice of $\bar \theta_1$, and then a sign for $\widehat c_1$, we fix a choice of $(\bar \theta_2 ..., \bar \theta_M)$ in such a way that $\widehat c_1$ and
$\widehat c_m$ have same signs for m=2,..,M. We choose $\bar r_m >0$ which satisfy
$$ {\rm Re} \left( C_m e^{-\mi (\ell_m +1)\bar  \theta_m} \right) \bar r_m ^{-(\ell_m+1)}-\frac{\ell_1+1}{\ell_m+1} {\rm Re} \left( C_1 e^{-\mi (\ell_1 +1) \bar \theta_1} \right) =0$$
The linearized system is
\[
    \dot\theta_m  = (\ell_m+1)\widehat c_m  \bar r_m ^{-(\ell_m+1)} \theta_m \,,\quad
    \dot r_m  = K_m (\bar \theta, \bar r_m) (r_m - \bar r_m)
\]
with $K_m (\bar \theta, \bar r_m)=\frac{(\ell_m+1)}{r^{\ell_m+1}} \widehat c_m$. Now there exists different choices of initial values, which fix the sign of eigenvalues for the linearized system. Thus we can find stable and unstable trajectories, and invoking the Stable Manifold Theorem concludes the proof.
\end{proof}
%
%
\section{Isolated property of the Abresch family}
\label{isolated}
Consider an Abresch annulus $X_1( \C /(\tau\Z) \times \R)$ where $\tau$ is the period of the cylinder in the $x$-direction.
By proposition \ref{abresch}, the metric of  $X_1$ is induced by two elliptic functions $x \to f(x) =\tfrac{-\o_x}{\cosh \o}$ and $y \to g(y)=\tfrac{-\o_y}{\cosh \o}$. The Jacobi operator on $X_1$ is given by
$${\calL}_1=\frac{1}{\cosh^2 \o_1}\left(\partial^2_{x} + \partial^2_{y} +1 + \frac{2 | \nabla \o _1|^2}{\cosh^2 \o_1} \right)=\frac{1}{\cosh^2 \o_1}\left(\partial^2_{x} + \partial^2_{y} +1 + 2f^2(x) + 2 g^2 (y) \right).$$
We use Fourier analysis. We define the set of periodic eigenfunctions $\{e_n\}$ associate to eigenvalues
$\lambda_0 < \lambda_1 \leq \lambda_2....$ repeated with multiplicity,
\begin{equation} \label{eq:opx}
 \partial_{x}^2 e_n (x) +2 f^2 (x)e_n(x)=-\lambda_n e_n(x).
\end{equation}
 where $f$ is an elliptic function which satisfy
$$-( f_x)^2=f^4+(1+c-d)f^2+c= (f^2-\delta_1)(f^2-\delta_2).$$
If $\tau$ denotes the period of the annulus in the $x$-direction, which coincides with a period of $f$, then
the set $\{e_n \}_{n \in \N}$ span the Hilbert space $L^2 (\R \backslash \tau \Z)$. A bounded solution of ${\calL}_1 v_0=0$ decomposes into
$$v_0(x,y)=\sum_{i \geq 0} u_n(y) e_n(x)$$
where $u_n : \R \to \R$ are uniformly bounded functions.
\begin{lemma}
\label{degenerate}
Let $u : \C/\tau \Z \to \R$ be a bounded solution of
$$
{\calL}_1 u=\cosh^{-2} \o_1 (\partial _{x} ^{2} u +\partial _{y} ^{2} u+ u +2f^2(x) u + 2 g^2 (y) u ) =0
$$
where $x \mapsto f(x)$ and $y \mapsto g(y)$ are the functions defined in Theorem \ref{abresch}. Then $u$ can not have more than 2 zeroes on horizontal sections unless it vanishes identically.
\end{lemma}
\begin{proof}
Consider elliptic equations which define functions $x \mapsto f(x)$ and $y \mapsto g(y)$
$$-( f_x)^2=f^4+(1+c-d)f^2+c= (f^2-\delta_1)(f^2-\delta_2)$$
$$-( g_y)^2=g^4+(1+d-c)g^2+d= (g^2-\beta_1)(g^2-\beta_2)$$
with roots $2 \delta_1=-(1+c-d) +\sqrt{\Delta}$,  $2\delta_2=-(1+c-d)-\sqrt{\Delta}$, $2\beta_1= -(1+d-c)+\sqrt{\Delta}$, $2\beta_2=
-(1+d-c)-\sqrt{\Delta}$ and $\Delta= (1+c-d)^2-4c=(1+d-c)^2-4d$, we see that $f$ is oscillating around zero between $-\sqrt{\delta_1}$ and $\sqrt{\delta_1}$ since $\delta_2 <0$ and $g$ is oscillating between $-\sqrt{\beta_1}$ and $\sqrt{\beta_1}$.

We solve the equation (\ref{eq:opx}) on $[0, \tau /2 ]$. The function $\wp:=\alpha - f^2$ (with $3\alpha=-(1+c-d)$) is the Weierstrass $\wp$-function which satisfies the elliptic equation
$$(\wp')^2=4\wp^3 -g_2\wp -g_3$$
for some constant $g_2,g_3$ depending only on constants $c$ and $d$. The equation (\ref{eq:opx}) transform into the Lam\'e equation
\begin{equation} \label{lame}
\partial^2_{x} e_n -2\wp e_n=-\mu_n e_n
\end{equation}
On $[0,\tau/2]$, the functions $e_0=\sqrt{f^2-\delta_2}, e_1=f ,e_2= \sqrt{\delta_1-f^2}$ are known as Lam\'e function of degree one of the first kind (see \cite{whittaker}, chapter XXIII). These fuctions extend to $[0,\tau]$ by symmetry and they are the first three eigenfunctions of the Lam\'e operator.

The function $e_0=\sqrt{ f^2- \delta_2} >0$ is an eigenfunction
associate to the eigenvalue $\lambda_0=-\delta_1$ with boundary data $e_0(0)=e_0 (\tau/2), \partial_x e_0 (0)=\partial_x e_0 (\tau/2)=0$.
The second eigenfunction $e_{1}= f$ is associate to eigenvalue $\lambda_{1}=1+c-d$ with $e_1 (0)=-e_1 (\tau /2),  \partial_x e_0 (0)=\partial_x e_0 (\tau/2)=0$. The third function is $e_{2}= \sqrt{\delta_1 -f^2}$ associate to eigenvalue $\lambda_{2}=-\delta_2$ with $e_2 (0)=e_2 (\tau /2 )=0$ and $\partial_x e_0 (0)=-\partial_x e_0 (\tau/2)$.
These eigenfunctions extend by symmetry to $\R /\tau \Z$. They are the first three eigenfunctions of the spectrum with $\l_0<\l_1<\l_2$ and have at most two zeroes on each horizontal curve.  If $e_k$ is an eigenfunction having strictly  more than two zeroes on a period $[0,\tau]$, then the associated eigenvalues $\lambda_k > \lambda_{2} $.  If not, one can argue by contradiction and compute  $W=e_k (\partial_x e_{2})-(\partial_xe_k )e_{2}$. Then $W'=(\lambda_k - \lambda_{2}) e_k e_{2}$ and by studying the behavior of $W$ between two consecutive zeroes of $e_k$, the function $e_{2}$ has to change sign.  Thus $e_{2}$ would have at least four zeroes, a contradiction.

Now we consider $u$ a bounded Jacobi field on $\Sigma_0$. By Fourier expansion we decompose  $u$ as
$$u(x,y)=\sum_{i \geq 0} u_n(y) e_n(x)$$
\noindent
Since $u$ is bounded on $A_0$ then $u_n$ is bounded on $\R$.
Inserting $u$ in the equation ${\calL}_0 u =0$ we obtain a countable set of equations for $n \in \Z$:
$$ \partial ^2_{y} u_n (y) +2g^2(y)u_n(y)+(1-\lambda_n) u_n(y)=0.$$
For $n \geq 2$, we have $(1-\lambda_n)< 1-\lambda_{2}=1+\delta_2=\frac12( 1+d-c + \sqrt{\Delta})=-\beta_1$.
But as we remarked for equation (\ref{eq:opx}), the function $\sqrt{ g^2- \beta_2} >0$ is the first periodic eigenfunction associated with the first eigenvalue $\mu_0 =-\beta_1$ of
\begin{equation} \label{eq:opy}
 \partial_{y}^2 v (y) +2 g^2 (y)\,v(y)=-\mu \,v(y)
\end{equation}
It is a well known fact (see \cite{sturm} for example) that for $\mu < \mu_0$ the equation (\ref{eq:opy}) cannot have bounded solutions on $\R$. Then $u_n = 0$ for $n \geq 2$. The function $u$ is a linear combination of $e_0,e_1,e_2$ and we obtain a contradiction with the following lemma.
\end{proof}
\begin{lemma}
For any real constants $\a_0,\a_1,\a_2$, the function $\a_0 e_0 + \a_1 e_1+\a_2 e_2$ has at most two roots on $\R/\tau \Z$.
\end{lemma}
\begin{proof}
The functions $e_0$, $e_1$ and $e_2$ obey $e_0^2+\delta_2=e_1^2=\delta_1-e_2^2$. Therefore in the expression
$$(\alpha_0e_0+\alpha_1e_1+\alpha_2e_2)(\alpha_0e_0+\alpha_1e_1-\alpha_2e_2)(\alpha_0e_0-\alpha_1e_1+\alpha_2e_2)(\alpha_0e_0-\alpha_1e_1-\alpha_2e_2)$$
is an even polynomial $p(e_1)$ of degree four with respect to $e_1$ with real coefficients not depending on $e_0$ and $e_1$. Along the period $\tau$ the function $e_1$ takes all values in $(-\sqrt{\delta_1},\sqrt{\delta_1})$ exactely twice and $\pm\sqrt{\delta_1} $ exactely once. At two points in the preimage of one value of $e_1$ in $(-\sqrt{\delta_1},\sqrt{\delta_1})$ the function $e_0$ takes the same value and $e_2$ takes values with oposite sign. Therefore every root of $p(e_1)$ corresponds to at most one root of $\a_0 e_0 + \a_1 e_1+\a_2 e_2$. For non vanishing values $e_1$ at roots of $\a_0 e_0 + \a_1 e_1+\a_2 e_2$, the negative $-e_1$ is the value at a root of $\a_0 e_0 - \a_1 e_1+\a_2 e_2$ and not of a root of $\a_0 e_0 + \a_1 e_1+\a_2 e_2$. Therefore at most two of the four roots of $p(e_1)$ are the values of $e_1$ at one root of $\a_0 e_0 + \a_1 e_1+\a_2 e_2$.
\end{proof}
%

%
\section{Bubbletons on Abresch family are not embedded}
\label{abresch bubbleton}
\noindent{\bf Bubbletons on flat annuli.}
We consider the spectral data of the flat cylinder and we prove that a flat cylinder with bubbletons cannot be embedded.
\begin{proposition}
There are no embedded bubbletons of spectral genus zero.
\end{proposition}
\begin{proof}
The flat cylinder has a spectral curve of genus zero. Assume that there is an embedded bubbleton dressed at $\alpha_0$ with $|\alpha_0| \neq 1$. Since the bubbleton is embedded, removing the bubbleton gives the embedded flat cylinder, then the period of the bubbleton is the same as the period of the flat cylinder.
The length of this period is $|\tau| = 2\pi$.  The bubbleton occurs at a double point $\alpha_0$ where $\mu (\alpha_0)^2=1$ and $\mu=f \nu +g$ with $f(\alpha_0)=0$.
We add to $a(\l)$ a double zero and to $b(\l)$ a simple zero at $\alpha_0$ and $\bar \alpha_0^{-1}$, and we apply the parametrization of section \ref{param} to obtain
$$\tilde a ( \l) = (1-\l \bar \alpha_0)^2(\l -\alpha_0)^2 a(\l) \hbox{ and }  \tilde b ( \l) = (1-\l \bar \alpha_0)(\l -\alpha_0) b(\l)\,.$$
Since the bubbleton is assumed to be embedded we can open the spectral curve at the double point $\alpha_0$, increasing strictly the length of $|\tau|$ in the family of genus 2 spectral curves by chosing $\tilde c(\tilde \l)$ as in proposition \ref{opendoublepointprop} and remark \ref{opendoublepointrem}.
Along this deformation the annulus stays embedded. Now we increase the period $|\tau|$ up to its maximun.
At this maximum we are in a spectral genus zero by lemma \ref{increasetau}, closed along a period $|\tau| > 2\pi$. This flat annulus is a covering of the embedded flat annulus since the length of horizontal geodesics are larger than $2\pi$. Since the deformation we use preserves embeddedness, the bubbleton cannot be embedded.
\end{proof}
\noindent{\bf Bubbleton on Riemann's type example.}
We consider the spectral data $(\tilde a, \tilde b)$ of an Abresch annulus. This data satisfies (see appendix)
$$\l^{2g} \overline{\tilde a(1/\bar \l)}=\tilde a(\l) \hbox{ and } \l^{g+1} \overline{ \tilde b(1/\bar \l)}=-\tilde b(\l)  $$
\begin{enumerate}
\item[a)] $\l^{2g}  \tilde a(1/ \l)=\tilde a(\l)$ and  $\l^{g+1}\tilde b(1/\l) =\tilde b(\l)$ if $\tilde a$ has a root $\alpha \in \R^+$ and $\tilde b(0) \in \mi \R$
\item[b)]  $\l^{2g}  \tilde a(1/ \l)=\tilde a(\l)$ and $\l^{g+1}\tilde b(1/\l) =- \tilde b(\l)$ if $\tilde a$ has only roots in $\R^-$ and $\tilde b(0) \in  \R$
\end{enumerate}
Now consider a bubbleton on $(\tilde a, \tilde b)$. Bubbletons occur at $\alpha_0$ away from $\SY^1$. The spectral curve of a simple bubbleton is singular and
$$\Sigma=\{ (\nu, \l) \in \C^2 ; \nu^2= \l^{-1} a (\l) =\l^{-1}  (\l -\alpha_0)^2(1-\l \alpha_0)^2 \tilde a (\l) \}$$
To close the period, we need  $\tilde F_{\alpha_0}(\tau)=\pm \un$, where
$\tau$ is a period of the annulus $(\tilde a, \tilde b)$. The set of bubbletons at $\alpha_0$ is described by an orbit homeomorphic to $\C \P ^1$. If the bubbleton closes with period $\tau$, then there is a holomorphic map  $\mu=f \nu +g$ on $\Sigma-\{0, \infty \}$.
The map $\nu$ has simple roots at $\alpha_0$ and $1/\bar \alpha_0$ and $\nu=(\l -\alpha_0)(1-\l \alpha_0) \tilde \nu$ where $\nu^2=\l^{-1} \tilde a(\l)$. We see that $\mu= f \nu +g = (\l -\alpha_0)(1-\l \alpha_0)f  \tilde \nu + g$ is holomorphic on $\tilde \Sigma -\{0, \infty\}$ and $\mu (\alpha_0)= g(\alpha_0) = \pm 1$ (by hyperelliptic involution).

This means that if there is a periodic bubbleton with period $\tau$, then we have an isospectral set of periodic bubbletons with the same period. In this isospectral orbit there is a potential $\xi_\l$ with a root at $\alpha_0$. This annulus has a removable singularity
and we see that $\mu$ defines a holomorphic periodic map, and the same period $\tau$ closes the annulus defined by the potential $\tilde \xi_\l$.
In summary, bubbletons occur at double points of $\mu= \tilde f \tilde \nu + \tilde g$, that is where $\mu (\alpha_0)=\pm 1$.

\begin{theorem}
There is no embedded bubbleton in the family of Abresch annuli.
\end{theorem}

\begin{proof}
We assume the existence of an embedded bubbleton. This bubbleton has $2n$ higher order roots. We can apply the isospectral action to inductively reduce the order of the bubbleton. To do that assume the potential  $\xi_\l=p(\l) h_{L',\alpha_0}\tilde \xi_\l h^{-1}_{L',\alpha_0}$ induces the bubbleton. Then we can find an isospectral action, which changes $L'$ and preserve $\tilde \xi_\l$. At the closure of the limit we produce a potential with removable singularity at $\alpha_0$. If the bubbleton is embedded, then we obtain an embedded bubbleton with $2n-2$ higher order roots. Inductively we find a bubbleton on an Abresch annulus with a double root $\alpha_0$ of order $2$.

The point $\alpha_0$ is not a root of $\tilde a(\l)$, because if it were, then would have $d \mu (\alpha_0)=0$ and then $\alpha_0$ would be a common root of $\tilde a$ and $\tilde b$.
But $\tilde b$ never has roots at the roots of $\tilde a$ by Proposition \ref{example}. Thus bubbletons occur only away from $\SY^1$, and away from the roots of $\tilde a(\l)$.
A study of spectral data in Lemma \ref{doublepointabresch} show that at roots of $\tilde b$ the function $\mu$ cannot have double points.
Thus double points do not occur at roots of $\tilde a$ and roots of $\tilde b$.

This property holds for every spectral genus 0, 1 or 2 spectral curve describing an Abresch annulus. Now we can apply deformation of Theorem \ref{thm:deformation} where $\tilde a$ and $\tilde b$ have no common roots, increasing the flux in order to reduce the spectral genus and end with spectral data $(\tilde a, \tilde b)$ of the flat annulus. We deform the bubbleton on the Abresch annulus (of spectral genus 1 or 2) into a bubbleton on a flat cylinder. It suffices to keep the period closed of the bubbleton along the deformation so that $\tilde F_{\alpha_0}(\tau)=\pm \un$ where $\alpha_0$ is a double point. Here double points of $\mu$ evolve by the equation (see proof of theorem \ref{thm:deformation})
$$\frac{ \dot \alpha_0}{\alpha_0} =- \frac{\tilde c(\alpha_0)}{\tilde b(\alpha_0)}.$$
Since $\tilde b(\alpha_0)\neq 0$, the equation is well defined. The double point moves along the trajectory of spectral data and the bubbleton remains closed and embedded by proposition \ref{deformembbubble}. We can follow the double roots along the deformation of Abresch annuli, increasing the flux and ending at the flat annulus.

To prove that at the end we have a bubbleton on a flat cylinder, there remains to prove that the double point $\alpha_0$ is not converging to $\SY^1$ along the deformation
(if not the double roots can be removed without changing the geometry of the annulus).

By symmetry $\mu (\bar \alpha_0 ^{-1}) = \pm 1$ and if during the deformation the point $\alpha_0$ is going to $\SY^1$, then the two double points $\alpha_0$ and $\bar \alpha_0 ^{-1}$ must coalesce.
Then we have a root of $d \ln \mu$ and thus a zero of $\tilde b$ on $\SY^1$. This can happen only at $\l =-1$ or $\l =1$ by proposition \ref{example}, where zeroes of $\tilde b$ are given.

Now we study zeroes of the function $\mu ^2 -1$ on the spectral curve $\Sigma$. For a flat annulus we have two roots of $\mu ^2 -1$. Since $\ln \mu =\tfrac{\mi \pi}{2} (\sqrt{\l}+1/ \sqrt{\l})$, we have two roots of order two at $\l=1$ and two roots of order one at $\l=-1$ (by hyperelliptic involution we have roots on $(\l_0,\nu)$ and $\l_0,-\nu)$ when $\l_0$ is not a zero of $a(\l)$).

For genus one with $0<\alpha <1$, we have two roots of order one at $\l=-1$ and two roots of order one at the branch points $\alpha, \bar \alpha ^{-1}$, and two roots of order one
at $\l=1$.

For genus one with $-1<\alpha <0$, we have two roots of order two at $\l =1$ and two roots of order one at the branch points and no roots at $\l=-1$.

For genus two, we have one root of order one at each branch point $\alpha, \beta, \bar \alpha ^{-1}, \bar \beta ^{-1}$, two roots of order one at $\l =1$ and no roots at $\l=-1$.

Beside these roots there are no additional roots of $\mu^2 -1$ nearby $\l = \pm 1$. This means that along a deformation converging to the flat annulus, the roots of $\mu^2 -1$
in $\C^* -\SY^1$ away from the branch points cannot converge to $\SY^1$. The roots which sit at $\l = \pm 1$ are limits of roots listed above and thus there is no additional one,
since $\alpha_0$ is not amongst the branch points of the spectral curve.

Along this deformation the roots of $\mu^2 -1$ cannot go to infinity because zeroes of $b$ are bounded. We remark that the equation $\frac{ \dot \alpha_0}{\alpha_0} =- \frac{c(\alpha_0)}{b(\alpha_0)}$
cannot provide a solution where $\alpha_0 (t) \to \infty$ in finite time, since the vector field $c$ is polynomial in $\alpha_0$, and $b(\alpha_0)$ is bounded away from zero when
$\alpha_0$ goes to inifinity.

Hence our deformation deforms a bubbleton on a Riemann type example to a bubbleton on a flat cylinder. Since the deformation preserves embeddness, this yields an embedded bubbleton on the flat cylinder, a contradiction.
\end{proof}

\appendix
\section{Sym point}
\label{sympoint}
\begin{proposition}
\label{sym}Let $\zeta_\l:\R^2\to \pk$ be a polynomial Killing field \eqref{eq:pKf} with initial condition $\xi_\l \in \pk$. Let $F_\l$  be the corresponding extended frame $F_\l^{-1}dF_\l=\alpha (\zeta_\l)$ and the immersion $X_{\l_0}(z)=(F_{\l_0} (z)  \sigma_3 F_{\l_0}^{-1}(z) ,Re(-\mi z))$ is parameterized by its third coordinate with $\lambda_0=e^{\mi\theta}$. The Hopf differential is given by $Q_{\l_0}(z)=-4\beta_{-1} \gamma_0 \l_0^{-1}(dz)^2=\frac{1}{4}(dz)^2$. Then the map $\widetilde F _\l (z) = F_{e^{\mi \theta}\l}( e^{\mi(1-g)\theta/2}z)$ is the unitary factor of $\exp(z \widetilde{\xi}_\lambda)$ with

$$\widetilde \xi _\l = e^{\mi(1-g)\theta/2} \xi_{e^{\mi\theta}\l}.$$
In particular we have
$$\det \widetilde \zeta_\l (z)=\det \widetilde \xi_\l=-\l^{-1}\widetilde a( \l)= -\l^{-1} e^{- \mi g\theta} a(e^{\mi\theta}\l).$$
The immersion is locally given by $$\widetilde X_1(z)=X_{\l_0}(e^{\mi(1-g) \theta/2} z)=(\widetilde F_1 (z)  \sigma_3 \widetilde F_1^{-1}(z) ,\mathrm{Re}\,(-\mi e^{\mi(1-g) \theta/2}z)).$$
\end{proposition}
\begin{proof}
The matrix $\widetilde \xi_\l=\bigl( \begin{smallmatrix}
\widetilde \a(\l) &  \widetilde \beta(\l) \cr
\widetilde \g(\l) & -\widetilde \a(\l)
\end{smallmatrix} \bigr) \in\pk$ satisfies the reality condition with
$$\widetilde \alpha ( \l)=e^{\mi(1-g)\theta/2} \alpha (e^{\mi\theta}\l)\,,\,\, \widetilde \beta( \l)=e^{\mi(1-g)\theta/2} \beta (e^{\mi\theta}\l)\,\, \hbox{ and }\,\,\widetilde \gamma ( \l)=e^{\mi(1-g)\theta /2} \gamma (e^{\mi \theta } \l).$$
\noindent
The matrix $\widetilde \xi _\l$ has unitary factor in Iwasawa decomposition  $\widetilde F_\l (z)=  F_{e^{i \theta}\l}( e^{i(1-g)\theta/2}z)$. The Hopf differential associate to $\widetilde F_1 (z)$ is given by
$$\widetilde Q_1 = -4 \widetilde \beta_{-1} \widetilde \gamma_0 (dz)^2 = -4 \widetilde a(0) (dz)^2
=-4 e^{-ig \theta} \beta_{-1} \gamma_0 (dz)^2 =\tfrac{1}{4}\, e^{\mi(1-g) \theta}  (dz)^2 .$$
which proves the formula for $\tilde X_1$.
\end{proof}
\section{Bubbletons}\label{sec:bubbleton}
Bubbletons occur when $a(\l)=(\l - \a_0)^2(1-\l \bar \a_0)^2 \tilde a ( \l)$ has roots of higher order, and then the spectral curve $\Sigma$ is singular.
The idea is to construct a group action which deforms an embedded bubbleton associated to $\xi_\l \in I(a)$ into an embedded annulus induced by a potential $\tilde \xi_\l \in I({\tilde a})$. Inductively we can remove all roots of higher order of $a(\l)$.

If $\xi _{\alpha_0} \neq 0$ and $\det \xi_{\alpha_0}=0$, then the matrix is nilpotent and defines a complex line
$L=\ker \xi_{\alpha_0} = {\rm Im}\;  \xi_{\alpha_0}  \in \C\P^1$. This complex line $L$ determines a simple factor having a pole at $\alpha_0$ and $\bar \alpha_0^{-1}$. Technical details on simple factor dressing are given in section 6 of \cite{HKS1}. Roughly speaking, the idea is to consider the following matrix
$$\pi _{\alpha_0} (\l) := \begin{pmatrix}
 \sqrt{\frac{\lambda-\alpha_0}{1-\bar{\alpha}_0\,\lambda}} & 0 \\
 0 & \sqrt{\frac{1-\bar{\alpha}_0\,\lambda}{\lambda-\alpha_0}} \end{pmatrix} $$
and $\pi_{L,\alpha_0} = Q_L \pi_{\alpha_0} Q_L^{-1}$, where $L,L^{\perp}$ are eigenlines of $\pi_{L,\alpha_0}$ by changing the basis of $\C^2$ ($Q_L \in \SU$ and $<Q_L e_1>=L$).
To get a simple factor, we need to find the correct matrix $Q_{1,L'}$ in $\SU$ (by Gram-Schmidt orthogonalization) depending only on $L'$ in order to satisfy at $\l=0$
$$h_{L',\alpha_0}(0)=Q_{1,L'} \pi ^{-1}_{L',\alpha_0}(0) \in \Lambda^+_r\SL$$
with $r < |\alpha_0| <1$
This means that at $\l=0$ the matrix $h_{L',\alpha_0}(0)$ is an upper triangular matrix. Moreover, if $L$ is given, there is a unique $L'$ depending only on $L$ such that $h_{L',\alpha_0}(\l)(L')=L$ (see Lemma 6.1, 6.2, section 6, \cite{HKS1}). We denote this change by $L'=\hat Q_{L}L$ for some $\hat Q_L \in \SU$.

Any potential $\xi _\l \in I(a)$ decomposes uniquely into an element $(L',\tilde \xi_\l ) \in \C\P^1 \times  I({\tilde a})$ by
$$\xi_\l= p(\l)h_{L',\alpha_0} \tilde \xi_\l h^{-1}_{L',\alpha_0}$$
where $p(\l)=(\l - \alpha_0)(1-\bar \alpha _0 \l)$ and $L'=\hat Q_LL$ with $L={\rm Ker}\; \xi_{\alpha_0}$.
In the case where $(L'_0)^{\perp}$ is an eigenline of $\tilde \xi_{\alpha_0}$, the potential $\xi_\l$ has a zero at $\alpha_0$ and we can remove the zero without changing its extended frame $F_\l$.

For any potential $\xi_\l$ which decomposes into $(L', \tilde \xi_\l)$ we consider the change $(L'(t), \tilde \xi_\l)$ and change $L'(t)$ to $L'_1$ where $(L'_1)^{\perp}$ is an eigenline of $\tilde \xi_{\alpha_0}$. This change preserves embeddedness and periodicity as it is isospectral. The isospectral action acts transitively on the first factor $L' \in \C\P ^\times =\C \P ^1 -\{ L'_1,L'_2\}$ where  $(L'_1)^{\perp}, (L'_2)^{\perp}$
are eigenlines of $\tilde \xi_{\alpha_0}$ (eventually there is only one eigenline $(L'_1)^{\perp}$).

We consider the unitary factor $F_\l: \R^2 \rightarrow \Lambda_r \SU (\C)$ of the r-Iwasawa decomposition
$$\exp (z\xi_{\l})=F _\l B_\l$$
and we define $\tilde F_\l: \R^2 \rightarrow \Lambda \SU (\C)$ the unitary factor of the r-Iwasawa decomposition
$$\exp (z p(\l) \tilde \xi_{\l})=\tilde F_\l \tilde B_\l.$$
Terng-Uhlenbeck \cite{TerU} provide the following useful formula.
\begin{proposition}
Let $h_{L',\alpha_0}$ the simple factor with $\alpha_0 \in \C$, $r< |\alpha_0|<1$ and $L \in \C\P^1$. Then
$$ F_\l (z)= h_{L',\alpha_0} \tilde F_\l (z) h^{-1}_{L'(z),\alpha_0}     \hbox{ with } L'(z)={ {}^t \overline{\tilde F} }_{\alpha_0}(z) L' $$
\end{proposition}
It turns out that changing the line $L'$ is isospectral and preserves the closing conditions. To investigate this restricted isospectral action we have the following
\begin{theorem}
\label{bubbleaction}(Theorem 6.8 \cite{HKS1})
We consider a decomposition of a potential $\xi_\l \in I(a)$ in $(L', \tilde \xi_\l) \in \C \P^1 \times I({\tilde a})$ and the two-dimensional subgroup action  of $\R^{2g}$,
$\tilde \pi(.)\xi_\l : \C \to I(a)$  given  by
$$ \tilde \pi (\beta) \xi_\l=\left\{
\begin{array}{ll}
\left( \frac{\beta}{\l-\alpha_0} + \frac{\bar \beta \l}{1-\bar \a _0 \l} \right) \l^{\frac{1-g}{2}} \xi_\l & \hbox{ when } g=2k+1 \\
\left( \frac{\beta}{\l-\alpha_0} + \frac{\bar \beta \l}{1-\bar \a_0 \l} \right) (\l^{\frac{-g}{2}}+ \l^{1-\frac{g}{2}})  \xi _\l& \hbox{ when } g=2k\,.
\end{array} \right.
$$
This subgroup acts on $\xi_\l=(L', \tilde \xi_\l)$ by preserving the second term  $\tilde \xi_\l$ of the decomposition. If we denote by  $\tilde \pi(\beta) \xi_\l=(L' (\beta), \tilde \xi_\l)$, the
subgroup $\C$ acts on $L' \in \C \P^1 \backslash \{L'_1,L'_2\} $ transitively where  $(L'_1)^{\perp}, (L'_2)^{\perp}$ are eigenlines of $\tilde \xi_{\alpha_0}$ and fixed point of the action.
\end{theorem}

We show that the action acts transitively on the first factor and preserves the second factor $\tilde \xi_\l$ to prove that along this deformation embeddedness and closing conditions are preserved.

\begin{proposition}
\label{bubbleclosing}(Proposition 6.7 \cite{HKS1})
If there is $(L'_1, \tilde \xi_\l) \in \C\P^1 \times I({\tilde a})$, such that  $\xi_{1,\l} =p(\l) h_{L'_1,\alpha_0}\tilde \xi _\l h^{-1}_{L'_1,\alpha_0}$
induces an embedded minimal annulus with period $\tau$, then for any $L'_2 \in \C\P^1$, the potential $\xi_{2,\l} =p(\l)h_{L'_2,\alpha_0}\tilde \xi _\l h ^{-1}_{L'_2,\alpha_0}$ yields an embedded minimal annulus with the same period $\tau$.
\end{proposition}
\begin{corollary}\label{bubbleton orbit}\cite{HKS1}
This group action defines an orbit which preserves the degree of the roots of the potential $\xi_\l$. If $a(\l)=(\l-\a_0)^{2}(1-\l \bar \alpha_0)^{2} \tilde a(\l)$,
then the closure of the set $N= \{ \xi_\l \in I(a) \mid  \xi_{\alpha_0} \neq 0 \}$ is  isomorphic to the set $(\C \P^1) \times I({\tilde a})=\bar N$.
\end{corollary}

\section{Abresch annuli}

\subsection{Spectral data}

\begin{proposition}(Proposition 7.2 \cite{HKS1})
\label{example}
The spectral curve of genus 0 associate to an embedded annulus parameterized by its Sym point at $\l =1$ is given by

\begin{enumerate}
\item[1)]  $a(\l)=\frac{-1}{16}$ and $b(\l)= \frac{\pi}{16} ( \l-1)$
\end{enumerate}

The spectral curve of genus 1 associate to an embedded annulus parameterized by its Sym point at $\l =1$ is given by

\begin{enumerate}
\item[2)] $a(\l)=\frac{1}{16\a}(\l-\alpha)(\a \l -1)$ for $\a \in (0,1)$ and $b(\l)= \frac{b(0)}{\gamma} (\l - \gamma)(\gamma \l -1)$ with $\gamma \in (\alpha,1)$ and $b(0) \in \mi\R$
both determined by $\a$.
\item[3)] $a(\l)=\frac{-1}{16\beta}(\l+\beta)(\beta \l +1)$ for $\beta \in (0,1)$ and $b(\l)= \frac{b(0)}{\gamma} (1- \l)(1+ \l)$ and $b(0) \in \R$ determined by $\beta$.
\end{enumerate}
The spectral curve of genus 2 associate to an embedded annulus parameterized by its Sym point at $\l =1$ is given by
\begin{enumerate}
\item[4)] $a(\l)=\frac{1}{16\beta \alpha}(\l-\a)(\a \l -1) (\l +\beta)(\beta \l +1)$ for $\a, \beta \in (0,1)$ and $b(\l)=\frac{b(0)}{\gamma}(1+\l)(\l - \gamma)(\gamma \l -1)$ for $\gamma \in (\alpha,1)$ and $b(0) \in \mi \R$ both determined by $\a$ and $\beta$.
\end{enumerate}
In conclusion, the polynomial $a$ satisfies the additional symmetry $\l^{2g} a(1/\l) = a(\l)$ and
\begin{enumerate}
\item[a)] $\l^{g+1} b(1/\l) = b(\l)$ if $a$ has a root $\alpha \in \R^+$ and $b(0) \in \mi \R$
\item[b)]$\l^{g+1} b(1/\l) =- b(\l)$ if $a$ has only roots in $\R^-$ and $b(0) \in  \R$
\end{enumerate}
\end{proposition}
\begin{proof}
The proof is given in \cite{HKS1}. The Abresch system of proposition \ref{abresch} gives the relation $\o_{zzzz}-2\o_z^3=-\frac{1}{4} \o_{\bar z} +\frac{c-d}{2} \o_z$.
We apply the iteration of Pinkall-Sterling described in proposition \ref{induction} and obtain a corresponding polynomial $a(\l)$. Finally we prove in \cite{HKS1} that the polynomial $b$ satisfies the closing conditions.
\end{proof}

\begin{lemma}
\label{doublepointabresch}
If $\gamma$ is a root of $b$ then the corresponding function $|\mu (\gamma)| \neq 1$.
\end{lemma}
\begin{proof}
For $\l \in [\alpha, \bar \alpha^{-1}]$, the function $h= \ln \mu$ is real and
$$\int_{\a}^{1 / \a} \frac{b \,d\l}{\nu \l^2}=2 \int_{\a}^{1} \frac{b \,d\l}{\nu \l^2}= \int _{\alpha} ^1 dh = 0.$$
Then $\gamma$ is a root of $dh$ and is contained in the interval $(\alpha,\,1)$. Since ${\rm Re}\, h(\alpha)={\rm Re}\, h(1)=0$, the value $\gamma$ is the local critical point of $h$ and then ${\rm Re}\, h(\gamma) \neq 0$, so $|\mu| \neq 1$.
\end{proof}

\subsection{Whitham deformation of spectral genus $0,1$ and $2$}
\label{opendoublegenuszero}

We apply the deformation in the case where the spectral curve has spectral genus one or two. We show that the family of Riemann type examples is a two parameter family.

The space of polynomials $a$ of degree $2g$ which obey $\l^{2g} \overline{a(\bar \l^{-1})}=a(\l)$ is a real
$2g+1$ dimensional vector space with $a_g \in \R, a_0, ...,a_{g-1} \in \C$ and
$(a_{2g},...,a_{g+1})=(\bar a_0,...,\bar a_{g-1})$. The value $|a_0|$ is independent of the roots of $a$. All other
$2g$ degrees of freedom are determined by the roots of $a$. The space of polynomials $b$ or $c$ is a real $g+2$ dimensional vector space. Then the set of spectral data $(a,b)$ has $(3g+3)$-degrees of freedom. By choosing a polynomial $c$, we have
\begin{equation}
- 2\dot{b}a + b\dot{a} = - 2 \l ac' + ac + \l a'c=p(\l)
\end{equation}
which is a set of $(3g+2)$-real equations (since $\l^{3g+1}\overline{ p(1/ \bar \l)}=-p(\l)$). These equations determine $(\dot a, \dot b)$ in terms of $c$. To preserve the closing condition we need $c(1)=0$ and
one real condition ${\rm Im} \tfrac{c(0)}{b(0)}=0$.

We consider the embedded flat cylinder as described in Proposition \ref{example}. The spectral data are given by $a_0(\l)=-1/16$ and $b_0(\l)=\frac{\pi}{16}(\l-1)$. We look for all deformations of  $(a_0,b_0)$ preserving embeddedness in the family of finite type annuli of genus $0,1$ or $2$.
We determine the values of $\mu =\pm 1$ on the unit circle. The point at $\l =+1$ and $\l =-1$ are available and we can open nodes of the spectral curve preserving the closing condition and embeddedness.

Therefore we can deform $a(\l)=\frac{1}{16}(\l-1)^2(\l+1)^2$ and $b(\l)=\frac{\pi \mi}{16}(\l-1)^2(\l+1)$. The corresponding $c$'s have to obey $c(1)=0$ and ${\rm Re}\,c(0)=0$. The solution space is the two-dimensional space spanned by $\mi(\l^3-1)$ and $\mi(\l^2-\l)$. Therefore they obey
$$\l^3c(1/\l)=-c(\l).$$
This implies that all of them preserve the symmetry
\begin{align*}
\l^4a(1/\l)&=a(\l)&\l^3b(1/\l)&=b(\l)&\l^3c(1/\l)&=-c(\l).
\end{align*}
The solution is calculated in \cite[section 7]{HKS1}. It is a
two-dimesnional family parameterized by $(\alpha,\beta)\in(0,1]\times(0,1]$
\begin{align*}
a(\l)&=\frac{1}{\beta\alpha}(\l-\alpha)(\alpha\l-1)(\l+\beta)(\beta\l+1)&
b(\l)&=\frac{b(0)}{\gamma}(1+\l)(\l-\gamma)(\gamma\l-1)
\end{align*}
with $b(0)\in \mi\mathbb{R}$ and $\gamma\in[\alpha,1]$ determined by
$\alpha$ and $\beta$. The unique genus zero example corresponds to
$\alpha=1=\beta$ and the two genus one families to $\alpha=1$ and
$\beta\in(0,1)$ and $\alpha\in(0,1)$ and $\beta=1$. In these cases
double roots of $a$ can be cancelled with simple roots of $b$.
Due to the normalization, in the first case together with the
cancellation $a$ is multiplied with $-1$ and $b$ with $\mi$ and the
action of the symmetry changes. This is a two parameter family of embedded annuli, describing the space moduli of Riemann type examples.
\begin{remark}
If there were double points on the unit circle not situated at $\l=1$ or $\l=-1$, then these could be used to deform an Abresch annulus into a higher spectral genus annulus. In an other way this is impossible by uniqueness of Abresch annuli. This would produce a Jacobi field with four zeroes on each horizontal section. This argument proves that the only double points are at $\l=1$ and $\l=-1$. An alternative way to prove the isolated property is to determine exactly the set of double point on the unit circle.
\end{remark}


\def\cydot{\leavevmode\raise.4ex\hbox{.}} \def\cprime{$'$}
\providecommand{\bysame}{\leavevmode\hbox to3em{\hrulefill}\thinspace}
\providecommand{\MR}{\relax\ifhmode\unskip\space\fi MR }
\providecommand{\MRhref}[2]{%
  \href{http://www.ams.org/mathscinet-getitem?mr=#1}{#2}
}
\providecommand{\href}[2]{#2}



%
%
%
%
%

\end{document}